%% file: main.tex
\documentclass[a4paper, 11pt, cleardoublepage=empty]{scrartcl} %twoside, 

\input{commands.tex}

\begin{document}
\let\OLDthebibliography\thebibliography
\renewcommand\thebibliography[1]{
  \OLDthebibliography{#1}
  \setlength{\parskip}{0.8pt}
  \setlength{\itemsep}{0pt plus 0.2ex}
}
\pagestyle{plain}

\begin{center}
    
\large{\textbf{Kontsevich's Cocycle Construction and Quantization of the Loday-Quillen-Tsygan Theorem}}
\end{center}
\hspace{0.5cm}
\begin{center}
    Jakob Ulmer
  \end{center}
\hspace{1cm}

\textbf{Abstract.} We relate graph complexes, Calabi-Yau $A_\infty$-categories and Kontsevich's cocycle construction. Our main result produces a commutative square of shifted Poisson algebras; one of its edges is the Loday-Quillen-Tsygan map, generalized to $A_\infty$-categories. We describe a quantized version via Beilinson-Drinfeld algebras. The larger context is to provide categorical methods which relate enumerative geometry (as in mirror symmetry) and large $N$ gauge theories.  
\tableofcontents
\section{Introduction}
\input{intro}

\section{Notation and Preliminaries}
\input{algebra}
\input{LQT}

\section{Various Graph Complexes}
 \subsection{Ribbon Graph Complexes}\label{rbgs}
We begin by introducing stable ribbon graphs, originally introduced in \cite{Kon92a}, their simpler cousins ribbon graphs and ribbon trees and by giving names to some of their features. 
Afterwards we equip complexes built from these graphs with shifted Poisson and Beilinson-Drinfeld structures, induced by gluing leafs. 
Lastly we identify certain Maurer-Cartan elements for these structures.
\begin{df}
A \emph{stable ribbon graph} $\Gamma$ is given by the data of
\begin{itemize}
\item a finite set $H$, called the set of half edges,
\item $V(\Gamma)$ a partition  of $H$, called the set of vertices. We denote the cardinality of a vertex $v$ by $|v|$ and the total number of vertices by $v(\Gamma)$.
\item an involution $\sigma_1$ acting on $H$ whose fixed points are called leafs, the set of those denoted $L(\Gamma)$ and whose 2-cycles are called internal edges, the set of those denoted $E(\Gamma)$.
We further denote the total number of internal edges by $e(\Gamma)$ and the total number of leafs by $l(\Gamma)$.
\item for each vertex $v$, a partition $C(v)$ of its half edges. We denote the  cardinality of such a partition at a vertex $v$ by $c(v)$ and by $c(\Gamma)$ the sum of those over all vertices.
\item a cyclic ordering for the half edges in each such partition,
\item a function $(g,b):V(\Gamma)\rightarrow \mathbb{N}_{\geq 0}\times \mathbb{N}_{\geq 0}$ called the genus and boundary defect
such that for any vertex $v$ if $g(v)=b(v)=0$ and $c(v)=1$ we have $|v|\geq 3$.
\end{itemize}

\end{df}
\begin{df}\label{faces}
Given a stable ribbon graph $\Gamma$ denote $\sigma_0: H\rightarrow H$ the permutation which defines the cyclic ordering of the elements in the partitions $C(v)$ at each vertex $v$.
Furthermore recall the permutation $\sigma_1: H \rightarrow H$, whose cycles are the interior edges respectively leaves.
We define $$\sigma_b:=\sigma_0^{-1}\circ\sigma_1.$$ The cycles of $\sigma_b$ are called the \emph{faces} of $\Gamma$, we denote the number of faces of a graph by $f(\Gamma)$.
    \end{df}
 Following remark 4.9 of \cite{Ha07} we make   
    \begin{df}\label{armg}
 The \emph{arithemtic genus} of a stable ribbon graph is defined by 
 $$g(\Gamma)=1-v(\Gamma)+\frac{1}{2}(e(\Gamma)+c(\Gamma)-f(\Gamma))+\sum_{v\in V(\Gamma)}g_v$$
    \end{df}  
We call the number of leafs in a given cycle defining a face the number of open boundaries in that face.

        A \emph{ribbon graph} is a stable ribbon graph such that $g(v)=b(v)=0$ for all vertices and thus $c(v)$=1 for all vertices.
        That is, we  recover graphs, possibly with leaves, together with a cyclic ordering on the half edges at each vertex. 
   
    \begin{remark}
        Note that for ribbon graphs we have the standard $2-2g(\Gamma)=
        v(\Gamma)-e(\Gamma)+f(\Gamma)).$
    \end{remark}
    \begin{ex}
  The ribbon graph given by the figure $8$ has 4 half edges, 1 vertex,  and 3 faces.     
    \end{ex}
    
        A \emph{ribbon tree} is a ribbon graph that has $f(\Gamma)=1$ and $g(\Gamma)=0$.
    
    \begin{df}\label{fBd}
Note that the half edges belonging to a given face of a stable ribbon graph $\Gamma$ which are leafs inherit a cyclic ordering.
We call a tuple of consecutive leafs in that cyclic ordering $(l_{i},l_{i+1})$ a \emph{free boundary} of that face.
Denote by $\partial(\Gamma)$ the set of free boundaries of a given stable ribbon graph. 
    \end{df}
\begin{ex}
    The ribbon tree given by the figure $+$ has four free boundaries.
\end{ex}
\begin{df}\label{Dec}
Let $B$ be a set, which we call the set of boundary conditions.
 A $B$-colored stable ribbon graph is a stable ribbon graph $\Gamma$ together with a map $c:\partial(\Gamma)\rightarrow B$. Furthermore we actually ask for a map $(g,b):V(H)\rightarrow \mathbb{N}_{\geq 0}\times \mathbb{N}_{\geq 0}^B$, that is a boundary defect $b_\lambda(v)\in\mathbb{N}_{\geq 0}$ for each element $\lambda$ of $B$ at each vertex $v$.
\end{df}  

        Denote the set of faces with zero open boundaries by $F_0(\Gamma)$ and its cardinality by $f_0(\Gamma).$ We call 
        $$f_{tot}(\Gamma):=f_0(\Gamma)+\sum_{\lambda\in B,v\in V(\Gamma)}b_\lambda(v)$$
        the total number of empty faces.

    An orientation of a (B-clolored) stable ribbon graph is an ordering of its internal edges.
   
\begin{remark}
    In the literature this may be better known as a twisted orientation.
    An orientation traditionally refers to an ordering of the edges and vertices.
    Since we only consider the latter we stick to the simpler name of orientation and hope that it does not cause confusion.
\end{remark}

    Two oriented B-colored stable ribbon graphs are called isomorphic if there is a bijection between their set of half edges that respects all other structures.

    The notions of a $B$-coloring, orientation and isomorphism extend to ribbon graphs and ribbon trees.

We define various chain complexes using the previous definitions. For that we introduce the following rule - slightly generalizing definition 4.3 of \cite{Ha07} to B-colored stable ribbon graphs, which we will use to define a differential.
\begin{df}[Contracting an Edge] \label{cont}
Let $\Gamma$ be an oriented $B$-colored stable ribbon graph and let $e\in E(\Gamma)$ be an edge.
We define the oriented B-colored stable ribbon graph $\Gamma/e$ to be the graph obtained by contracting this edge according to the following rules:
\begin{enumerate}
\item \label{item_contractedge}
Suppose that $e$ is not a loop, that is it joins distinct vertices $v_1,v_2\in V(\Gamma)$.
The halfedges of these vertices are partitioned into cycles, so the endpoints of $e$ lie in distinct cycles $c_1\subset v_1$ and $c_2\subset v_2$.
By contracting the edge $e$, the vertices $v_1$ and $v_2$ become joined, and the cycles $c_1$ and $c_2$ coalesce to form a new cycle with a naturally defined cyclic ordering.
The genus and boundary defects for the vertices $v_1$ and $v_2$ are added to give the defects for the new vertex created by joining $v_1$ and $v_2$.
The orientation is defined in an obvious way.
All the other combinatorial structures elsewhere on the graph are left alone.

Note that when both $c_1$ and $c_2$ each consist of a single half-edge, they define a face with one free boundary on it, which by definition is colored by an element of $B$.
In this case $c_1$ and $c_2$ do not coalesce when contracting the relevant edge, but instead are deleted.
The boundary defect at the new vertex is defined to be the sum of the boundary defects of $v_1$ and $v_2$ plus one, with respect to the element by which the face was colored.
If, furthermore, $c_1$ and $c_2$ are the only cycles of $v_1$ and $v_2$, then the edge $e$ cannot be contracted.

\item \label{item_contractloop1}
Now suppose that $e$ is a loop, in which case both its endpoints lie in a single vertex $v$.
Suppose furthermore, that they join distinct cycles $c_1,c_2\subset v$.
By contracting $e$, these cycles coalesce to form a single cycle as before.
In so doing the genus defect of $v$ increases by one.
No other combinatorial structures are changed.

As before, care must be taken when both $c_1$ and $c_2$ consist of a single half-edge.
Again these half-edges together define a face with one free boundary on it, which is by definition colored by an element of $B$.
In this case $c_1$ and $c_2$ are annihilated and \emph{both} the genus \emph{and} the boundary defect corresponding to that element are increased by one.

\item \label{item_contractloop2}
Finally, suppose that $e$ is again a loop, but that now both of its endpoints lie in the same cycle $c$ contained in some vertex $v$.
By contracting this loop the cycle $c$ splits up into two cycles $c_1$ and $c_2$, with naturally defined cyclic orderings.
All the other combinatorial structures remain unchanged.

Again, care must be taken with this definition when the endpoints of $e$ lie next to each other in the cyclic ordering- in this case they again define at least one face with one free boundary on it, which is colored by an element of $B$.
In this case, the cycle $c$ does not split up, but the boundary defect corresponding to that element is increased by one.
Furthermore, if the cycle $c$ consists of just the two half-edges of $e$ (and thus actually defines two faces, colored by an element each), then the cycle $c$ is annihilated and the boundary defect actually increases by one for each of these elements.
Finally, if in the lattermost situation the vertex has no other cycles than $c$, then the loop $e$ actually cannot be contracted at all.
\end{enumerate}
\end{df}
\begin{remark}\label{totfac}
    Note that the property being a ribbon tree is preserved under this differential, whereas being a ribbon graph is not.
    The differential preserves the arithmetic genus of a stable ribbon graph.
    It further preserves the 'total number of empty faces' $f_{tot}(\Gamma)$ of a stable ribbon graph, in sense of definition \ref{faces}.
\end{remark}

\begin{df}
Denote by $(s\mathcal{RG}_{c,\bullet}^B,\partial)$ the complex whose underlying vector space is freely generated by
isomorphism classes of connected B-colored oriented stable ribbon graphs, modulo the relation that reversing the orientation on a stable ribbon graph is equivalent to multiplying by $(-1)$.
This complex is graded by the number of internal edges.
The differential $\partial$ is given by summing over all possible contractions of the edges:
\[ \partial(\Gamma):=\sum_{e\in E(\Gamma)} \Gamma/e.\]
Note that some edges cannot be contracted, in which case the corresponding term in the sum is defined to be zero.

  Denote by $(\mathcal{RG}_{c,\bullet}^B,\partial)$ the complex given by the subspace (which is not subcomplex) of connected ribbon graphs, which uniquely becomes a chain complex by requiring that the projection map
  $$s\mathcal{RG}_{c,\bullet}^B\rightarrow \mathcal{RG}_{c,\bullet}^B$$
  is a map of chain complexes.

    Denote by $(\mathcal{RT}_{c,\bullet}^B,\partial)$ the subcomplex of $s\mathcal{RG}_{c,\bullet}^B$ given by ribbon trees.
\end{df}
We could have equivalently defined $\mathcal{RT}_{c,\bullet}^B$ as the subcomplex of $\mathcal{RG}_{c,\bullet}^B$
given by ribbon trees, which can be verified directly.
That is, following diagram of chain complexes commutes
   \begin{equation}\label{inc_proj}
   \begin{tikzcd}
(\mathcal{RT}_{c,\bullet}^B,\partial)\arrow[bend right=25]{rr}\arrow{r}&(\mathcal{RG}_{c,\bullet}^B,\partial)&
(s\mathcal{RG}_{c,\bullet}^B,\partial)\arrow{l}.
    \end{tikzcd}
    \end{equation}
Let us now further fix an odd integer $d$ and $\gamma$, a formal variable of degree $6-2d$.
By taking the dual of the previous complexes we arrive at the definitions we will mainly work with:
 \begin{df}\label{sRG}
The \emph{d-shifted stable ribbon graph complex colored by $B$} is defined to be the cochain complex 
$$s\mathcal{RG}^d_B\llbracket\gamma\rrbracket:=\Big(Sym\big( {s\mathcal{RG}^{B}_{c,\bullet}}^\vee[2d-6]\big)\llbracket \gamma\rrbracket[3-d],d\Big).$$ 
Its differential $d$ is given by the one induced from $(s\mathcal{RG}_{c,\bullet}^B,\partial)$. 
 \end{df}   
 \begin{df} The \emph{d-shifted ribbon graph complex colored by $B$} is defined to be the cochain complex
 $$\mathcal{RG}^d_B\llbracket\gamma\rrbracket:=\Big(Sym\big( {\mathcal{RG}^{B}_{c,\bullet}}^\vee[2d-6]\big)\llbracket \gamma\rrbracket[3-d],d\Big),$$
 its differential $d$ induced from  $(\mathcal{RG}^{B}_{c,\bullet},\partial)$
 \end{df}
 \begin{df}\label{rtc} The \emph{d-shifted ribbon tree complex colored by $B$} is defined to be the cochain complex 
 $$\mathcal{RT}^d_B:=\Big(Sym\big( {\mathcal{RT}^{B}_{c,\bullet}}^\vee[d-3]\big),d\Big),$$
 its differential induced from  $(\mathcal{RT}^{B}_{c,\bullet},\partial)$. 
 \end{df}
 \begin{remark}
   We refer to an element of eg $s\mathcal{RG}^d_B\llbracket\gamma\rrbracket$ as a not necessarily connected stable ribbon graph.
   Indeed, note that above three graph complexes have a basis given by symmetric powers of $B$-colored connected stable ribbon graphs, ribbon graphs, respectively ribbon trees.
   In this basis the respective differential $d$ expands edges in all possible ways,
   thereby possibly also changing the other properties of the graph, ie joining and splitting cycles and changing boundary and genus defect in the case of stable ribbon graphs.
 \end{remark}
 In the following section we will make extensive use of these bases:
 \subsubsection{Shifted Poisson Structure}
 We equip the graph complexes from the previous section with shifted Poisson structures:
\begin{cons}
    Define 
    $$\{\_,\_\}:{s\mathcal{RG}^{B}_{c,\bullet}}^\vee\otimes {s\mathcal{RG}^{B}_{c,\bullet}}^\vee\rightarrow {s\mathcal{RG}^{B}_{c,\bullet}}^\vee$$
    by 
    $$\{\Gamma_1,\Gamma_2\}=\sum_{ (*)}{\Gamma_1\cup_{l_1,l_2} \Gamma_2}.$$
    \begin{itemize}
        \item 
    Here $(*)$ indicates that we are summing over all isomorphism classes of stable ribbon graphs, suggestively denoted ${\Gamma_1\cup_{l_1,l_2} \Gamma_2}$, whose (unoriented) underlying stable ribbon graph is obtained by gluing $\Gamma_1$ and $\Gamma_2$ together along one of their respective leafs, thus creating a new internal edge.

  \item The orientation of $\Gamma_1\cup_{l_1,l_2} \Gamma_2$ is defined by taking first the given ordering of internal edges of $\Gamma_1$, then the newly created internal edge, and then the given ordering of internal edges of $\Gamma_2$.
    \end{itemize}
\end{cons} 
Note that
  $\Gamma_1\cup_{l_1,l_2} \Gamma_2$ has $f(\Gamma_1) +f(\Gamma_2)-1$ faces.  

\begin{lemma}
Given $B$ a set 
$\big((s\mathcal{RG}^{B}_{c,\bullet})^\vee,d,\{\_,\_\}\big) $
is a 1-shifted dg Lie algebra.
\end{lemma}
\begin{proof}
$\{\_,\_\}$ is of degree 1 since we are creating a new internal edge.
The graded anti-symmetry follows from the fact that we quotient out by the relation that reversing the orientation is equivalent to multiplying by $(-1)$ and the definition of the bracket.
The shifted Jacobi-identity, ie the fact that the bracket is a derivation with respect to the bracket can be verified easily.
Furthermore it is straightforward to see that the bracket and the differential together define a shifted dg Lie algebra.    
\end{proof}
\begin{remark}\label{nütz}
    In the same way we can define a shifted Lie bracket on ribbon graphs and ribbon trees such that the maps induced by diagram \ref{inc_proj} respect these brackets.
\end{remark}
  By first shifting the Lie bracket and then extending it according to the Leibniz rule to the symmetric algebra we obtain
  a bracket $\{\_,\_\}$ on $\mathcal{RT}^d_{B}$. Then, considering the previous lemma, it follows by standard arguments that: 
\begin{theorem}\label{RS}
For every set $B$ the ribbon tree graph complex from definition \ref{rtc} 
$$\big(\mathcal{RT}^d_{B},-d,\{\_,\_\},\cdot\big)$$
 is a $(d-2)$-shifted Poisson dg algebra. 
 \end{theorem}

 By abuse of notation we will denote this $(d-2)$-shifted Poisson dg algebra by $\mathcal{RT}^d_{B}$.

Define
 $$D:=\sum_{n>2}\sum_{(\#)}{D_{n,C}},$$
where $(\#)$ indicates that we are summing over all isomorphism classes of B-colored ribbon trees with one vertex, no internal edges and n leaves.
We denote  by ${D_{n,C}}$ a such determined isomorphism class. 
 \begin{theorem}\label{RS_mce}
 The corresponding element
$$D \in \mathcal{RT}^d_B$$
is a Maurer-Cartan element.
\end{theorem}
\begin{proof}
This will follow from remark \ref{prak_1} and theorem \ref{sRG_mce}.
Indeed, we wil see that $D$ is the image of a Maurer-Cartan element under a map of dg shifted Lie algebras.
\end{proof}

    We denote by ${\mathcal{RT}^{d,tw}_{B}}$ the $(d-2)$-shifted Poisson dg algebra
    obtained by twisting the $(d-2)$-shifted Poisson dg algebra ${RT}^d_{B}$ by the MCE $D$.

\subsubsection{Beilinson Drinfeld Structure}
\begin{cons}
   Again making use of our basis, define a map 
   $$\nabla:{s\mathcal{RG}^{B}_{c,\bullet}}^\vee\rightarrow {s\mathcal{RG}^{B}_{c,\bullet}}^\vee$$
by 
$$\nabla\Gamma=\sum_{ (\#)}\circ_{l_1,l_2}\Gamma.$$
    \begin{itemize}
        \item 
    Here $(\#)$ indicates that we are summing over all isomorphism classes of stable ribbon graphs whose underlying (unoriented) stable ribbon graph is obtained by gluing  together two leafs of $\Gamma$ which lie in the same face.
   
 \item The orientation of $\circ_{l_1,l_2}\Gamma$ is defined by taking first the new internal edge and then the given ordering of internal edges of $\Gamma$.
    \end{itemize}
\end{cons}

    Note that in case $(\#)$ the stable ribbon graph $\circ_{l_1,l_2}\Gamma$ has one more face as compared to $\Gamma$, the genus stays the same, however.

One easily verifies:
\begin{lemma}
    We have $[d,\nabla]=0$.
\end{lemma}
We extend $\nabla$ to a degree 1 map by first shifting it, then as a derivation and then shifting again: 
$$\nabla: s\mathcal{RG}_B^d\rightarrow s\mathcal{RG}_B^d
.$$
\begin{cons}
    Define a map $$\delta_{i}:{s\mathcal{RG}^{B}_{c,\bullet}}^\vee\rightarrow {s\mathcal{RG}^{B}_{c,\bullet}}^\vee$$
   by $$\delta_{i}\Gamma=\sum_{  (\%)}\circ_{l_1,l_2}\Gamma.$$
    \begin{itemize}
        \item 
    Here $(\%)$ indicates that we sum over all isomorphism classes of graphs $\circ_{l_1,l_2}\Gamma$ whose
    underlying (unoriented) stable ribbon graph is obtained by gluing together two leafs of $\Gamma$ that lie in different faces. 
    
    \item The orientation of $\circ_{l_1,l_2}\Gamma$ is defined by taking first the new internal edge and then the given ordering of internal edges of $\Gamma$.
    \end{itemize}
\end{cons}

    Note that the stable ribbon graph $\circ_{l_1,l_2}\Gamma$ obtained
    by gluing together two leafs of $\Gamma$ that lie in the same face has one less face compared to $\Gamma$ and the genus increased by one.

It is easily verified that:
\begin{lemma}
    We have $[d,\delta_{i}]=0$.
\end{lemma}

We extend $\delta_{i}$ to a degree 1 map by first shifting it, then as a derivation and then shifting again: 
$$\delta_{i}: s\mathcal{RG}_B^d\rightarrow s\mathcal{RG}_B^d.$$  

Denote by $\delta_e$ the Chevalley-Eilenberg differential on 
$Sym\big( (s\mathcal{RG}^{B}_{c,\bullet})^\vee[(2d-6)]\big)$
associated to the (2d-5)-shifted Lie bracket $\{\_,\_\}$ on 
$(s\mathcal{RG}^{B}_{c,\bullet})^\vee[(2d-6)]$.
Denote by the same latter the induced degree $(2d-5)$ map  
$$\delta_e: s\mathcal{RG}^d_{B}\llbracket\gamma\rrbracket\rightarrow s\mathcal{RG}^d_{B}\llbracket\gamma\rrbracket.$$

\begin{lemma}\label{rbc}
    We have $[d,\delta_e]=(\nabla+\delta_i)^2=[\nabla+\delta_i,\delta_e]=0$ as maps on $s\mathcal{RG}^d_{B}\llbracket\gamma\rrbracket.$
\end{lemma}
\begin{proof}
 The fact that the first commutator is zero follows since $\delta_e$ is the Chevalley-Eilenberg differential to $\{\_,\_\}$, which was compatible with $d$.
 
Note that $\nabla+\delta_i$ is given by the sum over all possible self-gluing - independent of a condition on the endpoints of the leafs.
Then the fact that the second term is zero follows since for $\Gamma$ a stable ribbon graph $(\delta_i+\nabla)^2(\Gamma)$ is given by a sum of connected stable ribbon graphs in which each two summands have the same underlying unoriented stable ribbon graph, but their orientation differs by exactly swapping the first two internal edges.
Since we mod out by the relation that reversing the multiplication is equivalent to multiplying by (-1) the claim follows.

The fact that the third term is zero follows by similar arguments.
\end{proof}
Thus we can infer 
\begin{theorem}\label{sRGt}
The B-colored stable ribbon graph complex defined as 
$$\big(s\mathcal{RG}^d_{B}\llbracket\gamma\rrbracket,\cdot, -d+\nabla+\delta_i+\gamma\delta_e,\{\_,\_\}\big)$$
is a $(d-2)$-twisted BD algebra over $k\llbracket\gamma\rrbracket$ whose multiplication has even degree $(3-d)$. 
\end{theorem}
\begin{proof}
 Indeed the fact that $-d+\nabla+\delta_i+\gamma\delta_e$ defines a map of degree 1 that squares to zero follows from the previous two lemmas.
 The BD relation follows since $\nabla$, $\delta_i$ and $d$ are derivations and by the definition of $\delta_e$ as the Chevalley Eilenberg differential associated to $\{\_,\_\}$. 
\end{proof}

    By abuse of notation we also denote by $s\mathcal{RG}^d_{B}\llbracket\gamma\rrbracket$ above $(d-2)$-twisted BD algebra. 

\begin{df}\label{deq_rib}
Denote by 
\begin{equation}\label{ncgdq}
p:s\mathcal{RG}^d_{B}\llbracket\gamma\rrbracket\rightarrow {\mathcal{RT}^{d}_{B}}\end{equation}
the composite 
$$Sym\big( {s\mathcal{RG}^{B}_{c,\bullet}}^\vee[2d-6]\big)\llbracket \gamma\rrbracket[3-d]\rightarrow Sym^1\big( {s\mathcal{RG}^{B}_{c,\bullet}}^\vee[2d-6]\big)[3-d]\cong {s\mathcal{RG}^{B}_{c,\bullet}}^\vee[d-3]$$
$${s\mathcal{RG}^{B}_{c,\bullet}}^\vee[d-3]\rightarrow {\mathcal{RT}^{B}_{c,\bullet}}^\vee[d-3]\rightarrow Sym\big({\mathcal{RT}^{B}_{c,\bullet}}^\vee[d-3]\big).$$
Here the first map in the first line is the projection on symmetric words of length one and the $\gamma=0$ part.
The first map in the second line is the one induced by the curved map of \ref{inc_proj} and the last map is just the inclusion. 
\end{df}
\begin{lemma}
    Given a set B
$$p:s\mathcal{RG}^d_{B}\llbracket\gamma\rrbracket\rightarrow {\mathcal{RT}^{d}_{B}} $$
    is a dequantization map in the sense of definition \ref{dq}.
\end{lemma}
\begin{proof}
The fact that $p$ intertwines the shifted Poisson bracket follows from remark \ref{nütz}.
The fact that $p$ intertwines the BD differential with the differential of ${\mathcal{RT}^{d}_{B}}$ can be verified directly:
$\nabla$ and $\delta_i$ get send to zero under the projection to ribbon trees since the prior produces a graph with more than one face and since the latter produces a graph with positive genus.
Lastly $\gamma\delta_e$ gets send to zero since it involved a positive power of $\gamma$.
\end{proof}
Define
 $${S}:=\sum_{n>0}\sum_{(\#)}{S}_{n,C},$$
where the second sum is over isomorphism classes of a B-colored stable ribbon graphs with no edges, one vertex and $n$ leaves and we denote a such determined class by ${S}_{n,C}$. 
\begin{theorem}\label{sRG_mce}
The corresponding element $${S}\in s\mathcal{RG}^d_{B}\llbracket\gamma\rrbracket$$
is a Maurer-Cartan element.
\end{theorem}
\begin{proof}
This will follow from theorem \ref{GALt} and theorem \ref{sG_mce}.
Indeed we will see that $S$ is the image of a Maurer-Cartan element under a map of dg-shifted Lie algebras.\end{proof}
\begin{remark}\label{prak_1}
    We have that $p(S)=D$, where $p$ is the dequantizaiton map from \ref{ncgdq}.
    Thus indeed theorem \ref{sRG_mce} implies theorem \ref{RS_mce}.
\end{remark}
\begin{df}\label{sRG_tw}    
Denote by $s\mathcal{RG}^{d,tw}_{B}\llbracket\gamma\rrbracket$ the $(d-2)$-twisted BD algebra obtained by twisting $s\mathcal{RG}^{d}_{B}\llbracket\gamma\rrbracket$ by the Maurer-Cartan element S.
\end{df}
Twisting the map \ref{ncgdq} by the compatible Maurer-Cartan elements $D$ respectively $S$ gives an induced map 
$$p:s\mathcal{RG}^{d,tw}_{B}\llbracket\gamma\rrbracket\rightarrow \mathcal{RT}^{d,tw}_{B}\llbracket\gamma\rrbracket.$$
 Note that we can restrict the operators $\nabla$, $\delta_i$ and $\delta_e$ to ribbon graphs. Exactly as in theorem \ref{sRGt} we can deduce that this defines a BD algebra on the ribbon graph complex. Furthermore diagram \ref{inc_proj} induces a map of BD algebras
$$v:\mathcal{RG}^d_{B}\llbracket\gamma\rrbracket\rightarrow s\mathcal{RG}^d_{B}\llbracket\gamma\rrbracket$$
and a dequantization map, defined exactly as in \ref{deq_rib}
$$p:\mathcal{RG}^d_{B}\llbracket\gamma\rrbracket\rightarrow \mathcal{RT}^d_{B}.$$
Summarizing, we have 
\begin{theorem}\label{rbg_comp}
The ribbon graph complex $\big(\mathcal{RG}^d_{B}\llbracket\gamma\rrbracket,\cdot,-d+\nabla+\delta_i+\gamma\delta_e,\{\_,\_\}\big)$ is a $(d-2)$-twisted BD algebra. Furthermore diagram \ref{inc_proj} induces a diagram

 \begin{equation} \begin{tikzcd}
\mathcal{RT}^d_{B}&\mathcal{RG}^d_{B}\llbracket\gamma\rrbracket\arrow[swap]{l}{p}\arrow{r}{v}&
s\mathcal{RG}^d_{B}\llbracket\gamma\rrbracket\arrow[bend left=25]{ll}{p}
    \end{tikzcd},
    \end{equation}
    where the arrows pointing to $\mathcal{RT}^d_{B}$ denote dequantization maps and the third one denotes a map of $(d-2)$-twisted BD algebras.
\end{theorem}
\begin{remark}
    Note that the Maurer-Cartan element $S\in s\mathcal{RG}^d_{B}\llbracket\gamma\rrbracket$ does not come from an element in $\mathcal{RG}^d_{B}\llbracket\gamma\rrbracket$ under the map $v$.
\end{remark}
\subsection{(Ordinary) Graph Complexes}\label{grs}
In this section we basically repeat the constructions from the previous section, but for stable graphs, introduced in \cite{geKa96} section 2.8 , their simpler cousins graphs and trees. That is we equip complexes built from these graphs with shifted Poisson and Beilinson-Drinfeld structures, induced by gluing leafs. Lastly we identify certain Maurer-Cartan elements for these structures.
\begin{df}\label{sgrdf} A \emph{stable graph} is given by the datum of
    \begin{itemize}
    \item a finite set $H$, called the set of half edges,
    \item $V(\Gamma)$ a partition  of $H$, called the set of vertices; we denote the cardinality of a vertex v by $|v|$ and the total number of vertices of a graph by $v(\Gamma)$,
    \item an involution $\sigma_1$ acting on $H$ whose fixed points are called leafs, the set of those denoted $L(\Gamma)$ and whose 2-cycles are called internal edges, the set of those denoted $E(\Gamma)$. We further denote the total number of internal edges by $e(\Gamma)$ and the total number of leafs by $l(\Gamma)$.
    \item a function $o:V(\Gamma)\rightarrow \mathbb{Z}_{\geq 0}$
    such that for $o(v)=0$ we have $|v|\geq 3$ and for $o(v)=1$ we have $|v|\geq1$ called loop defect .
    \end{itemize}
\end{df}
\begin{df}\label{betti}
    The \emph{first Betti number} $bt(\Gamma)$ of a connected stable graph $\Gamma$ is defined by  $$bt(\Gamma)=1-v(\Gamma)+e(\Gamma)+\sum_{v\in V(\Gamma)}o(v).$$
\end{df}

A \emph{graph} is a stable graph for which $o(v)=0$ for all vertices, ie in particular all vertices are at least three-valent.
A connected \emph{tree} is a connected graph $\Gamma$ for which $bt(\Gamma)=0$, that is, it has no loops.
An orientation of a stable graph is an ordering of its edges.
Two oriented stable graphs are isomorphic if there is a bijection between their set of half-edges that respects all the other structures.
\begin{df}[Contracting an Edge]\label{cont_g}
Let $\Gamma$ be an oriented stable graph and let $e\in E(\Gamma)$ be an  edge.
We define the oriented stable graph $\Gamma/e$ to be the graph obtained by contracting this edge according to the following rules:
\begin{enumerate}
    \item Suppose that e is not a loop, that is it connects distinct vertices $v_1,v_2\in V(\Gamma)$. By contracting the edge the vertices $v_1$ and $v_2$ become joined.
    The loop defect for the vertices $v_1$ and $v_2$ are added to give the loop defect for the new vertex created by joining $v_1$ and $v_2$.
    The orientation is defined in the obvious way.
    All the other combinatorial structures elsewhere in the graph are left alone. 
    \item
    Now suppose that $e$ is a loop, that is both its endpoints lie in the same vertex $v$.
    Then $\Gamma/e$ is defined by just deleting this loop, but by that increasing the loop defect of $v$ by one.
\end{enumerate}
\end{df}
Note that the property being a tree is preserved under this differential, whereas being a graph is not.

Denote by $(s\mathcal{G}_{c,\bullet},\partial)$ the complex whose underlying vector space is freely generated by isomorphism classes of connected stable graphs, modulo the relation that reversing the orientation on a stable graph is equivalent to multiplying by (-1).
It is graded by the number of internal edges.
The differential $\partial$ is given by summing over all possible contractions of edges:
    $$\partial(\Gamma):=\sum_{e\in E(\Gamma)}\Gamma/e.$$
    Denote by $(\mathcal{G}_{c,\bullet},\partial)$ the complex given by the subspace (which is not a subcomplex) of connected graphs, which uniquely becomes a chain complex by requiring that the projection map
$$s\mathcal{G}_{c,\bullet}\rightarrow \mathcal{G}_{c,\bullet}$$
is a map of chain complexes.   

Denote by $(\mathcal{T}_{c,\bullet},\partial)$ the subcomplex of $s\mathcal{G}_{c,\bullet}$  given by connected trees.

We could have equivalently defined $(\mathcal{T}_{c,\bullet},\partial)$ as the subcomplex of $(\mathcal{G}_{c,\bullet},\partial)$ given by  trees, which can be verified directly. That is, following diagram of chain complexes commutes
   \begin{equation}\label{inc_proj_2} \begin{tikzcd}
(\mathcal{T}_{c,\bullet},\partial)\arrow{r}\arrow[ bend right=25]{rr}&(\mathcal{G}_{c,\bullet},\partial)&
(s\mathcal{G}_{c,\bullet},\partial)\arrow{l}.
    \end{tikzcd}
    \end{equation}

Let $d$ be an odd integer and $\hbar$ be a formal variable of degree $(3-d)$.
\begin{df}\label{sG}
    The \emph{d-shifted stable graph complex} is defined to be the cochain complex 
$$s\mathcal{G}^d\llbracket\hbar\rrbracket:=\Big(Sym\big( s\mathcal{G}_{c,\bullet}^\vee[d-3]\big)\llbracket \hbar\rrbracket,d\Big).$$ 
Its differential $d$ is given by the one induced from $(s\mathcal{G}_{c,\bullet},\partial)$.  
\end{df}
 \begin{df} The \emph{d-shifted graph complex} is defined to be the cochain complex $$\mathcal{G}^d\llbracket\hbar\rrbracket:=\Big(Sym\big( \mathcal{G}_{c,\bullet}^\vee[d-3]\big)\llbracket \hbar\rrbracket,d\Big),$$
 its differential $d$ induced from  $(\mathcal{G}_{c,\bullet},\partial)$
 \end{df}
 \begin{df}\label{tc} The \emph{d-shifted tree complex} is defined to be the cochain complex 
 $$\mathcal{T}^d:=\Big(Sym\big( \mathcal{T}_{c,\bullet}^\vee[d-3]\big),d\Big),$$
 its differential induced from  $(\mathcal{T}_{c,\bullet},\partial)$. 
 \end{df}
 We refer to an element of eg $s\mathcal{G}^d\llbracket\hbar\rrbracket$ as a not necessarily connected stable graph.
 Indeed, note that above three graph complexes have a basis given by symmetric powers of connected stable graphs, graphs, respectively trees.
 In this basis the differential expands edges, thereby possibly also changing the other properties of the graph, ie changing the loop defect in the case of stable graphs.

 As before we will make extensive use of these 'dual' bases to define further algebraic structures:
  \subsubsection{Shifted Poisson Structure}
\begin{cons}\label{iwjs}
    Define 
    $$\{\_,\_\}:s\mathcal{G}_{c,\bullet}^\vee\otimes s\mathcal{G}_{c,\bullet}^\vee\rightarrow s\mathcal{G}_{c,\bullet}^\vee$$
    by 
    $$\{\Gamma_1,\Gamma_2\}=\sum_{(*)}\Gamma_1\cup_{l_1,l_2} \Gamma_2.$$
    \begin{itemize}

\item Here $(*)$ indicates that we are summing over all isomorphism classes of stable graphs, suggestively denoted ${\Gamma_1\cup_{l_1,l_2} \Gamma_2}$, whose underlying (unoriented) stable graph is obtained by gluing together two leafs of $\Gamma_1$ and $\Gamma_2$, thus creating a new internal edge\footnote{or by gluing together two leafs of $\Gamma_2$ and $\Gamma_1$}.

\item The orientation of $\Gamma_1\cup_{l_1,l_2} \Gamma_2$ is defined by taking first the given ordering of internal edges of $\Gamma_1$,
then the new internal edge, and then the given ordering of internal edges of $\Gamma_2$.
    \end{itemize}
\end{cons} 
\begin{lemma}
    $\big(s\mathcal{G}_{c,\bullet}^\vee,d,\{\_,\_\}\big)$ is a 1-shifted dg Lie algebra.
\end{lemma}
\begin{proof}
This follows analogously to the stable ribbon graph case.  
\end{proof}
\begin{remark}\label{nütz_2}
    In the same way we can define a shifted Lie bracket on  graphs and  trees such that the maps induced by \ref{inc_proj_2} respect these brackets.
\end{remark}
  In the same way as for ribbon graphs (theorem \ref{RS}) we deduce
\begin{theorem}\label{S}
The tree graph complex from definition \ref{tc} 
$$\big(\mathcal{T}^d,-d,\{\_,\_\},\cdot\big)$$
 is a $(d-2)$-shifted Poisson dg algebra. 
 \end{theorem}
 Recall that $\{\_,\_\}$ is given by first shifting the Lie bracket and then extending it as a according to the Leibniz rule to the symmetric algebra.

 We will by abuse of notation also denote this dg $(d-2)$-shifted Poisson algebra by $\mathcal{T}^d$.

Denote by $T_n$ the class of the tree with one vertex, no internal edges and $n$ leaves. We define 
 $$T:=\sum_{n>2}T_n$$
 \begin{theorem}\label{S_mce}
 The corresponding element
$$T \in \mathcal{T}^d$$
is a Maurer-Cartan element.
\end{theorem}
\begin{proof}
   This will follow from theorem \ref{sG_mce} and remark \ref{prak_2}. 
   Again, we will be able to express $T$ as the image of a Maurer-Cartan element under a map of dg-shifted Lie algebras.
\end{proof}

    Denote by $\mathcal{T}^{d,tw}$ the $(d-2)$-shifted Poisson dg algebra obtained by twisting the $(d-2)$-shifted Poisson dg algebra $\mathcal{T}^d$ by the Maurer-Cartan element $T\in\mathcal{T}^d$.

\subsubsection{Beilinson Drinfeld Structure}
\begin{cons}
   Define a map 
   $$\Delta_i:s\mathcal{G}_{c,\bullet}^\vee\rightarrow s\mathcal{G}_{c,\bullet}^\vee$$
by 
$$\Delta_i(\Gamma)=\sum_{(\#)}\circ_{l_1,l_2}\Gamma.$$
    \begin{itemize}

 \item Here $(\#)$ indicates that we sum over all isomorphism classes of stable graphs $\circ_{l_1,l_2}\Gamma$ whose
 underlying (unoriented) stable graph is obtained by gluing two leafs of $\Gamma$ together, thus creating a new internal edge.
    
   \item The orientation of $\circ_{l_1,l_2}\Gamma$ is defined by taking first the new internal edge and then the given ordering of internal edges of $\Gamma$.
    \end{itemize}
\end{cons}
    Note that the betti number changes as $bt(\circ_{l_1,l_2}\Gamma)=bt(\Gamma)+1$.

We easily verify
\begin{lemma}
    We have $[d,\Delta_i]=0$.
\end{lemma}

We extend $\Delta_i$ to a degree 1 map by first shifting it and then as a derivation to
$$\Delta_i: s\mathcal{G}^d\llbracket\hbar\rrbracket\rightarrow s\mathcal{G}^d\llbracket\hbar\rrbracket
.$$
Denote by $\Delta_e$ the Chevalley-Eilenberg differential on $Sym\big( s\mathcal{G}_{c,\bullet}^\vee[d-3]\big)$
associated to the odd Lie bracket $\{\_,\_\}$ on $s\mathcal{G}_{c,\bullet}^\vee[d-3]$.
Denote by the same latter the induced degree (d-2) map  
$$\Delta_e: s\mathcal{G}^d\llbracket\hbar\rrbracket\rightarrow s\mathcal{G}^d\llbracket\hbar\rrbracket.$$
\begin{lemma}
    We have $[d,\Delta_e]=(\Delta_i)^2=[\Delta_i,\Delta_e]=0$ on $s\mathcal{G}^d\llbracket\hbar\rrbracket.$
\end{lemma}
\begin{proof}
 These facts follows analogously to the stable ribbon graph case, see lemma \ref{rbc}.
\end{proof}
Thus we can infer 
\begin{theorem}\label{BD_G}
The stable graph complex from definition \ref{sG} 
$$\big(s\mathcal{G}^d\llbracket\hbar\rrbracket,\cdot, -d+\Delta_i+\hbar\Delta_e,\{\_,\_\}\big)$$
is a $(d-2)$-twisted BD algebra. 
\end{theorem}
\begin{proof}
 Indeed the fact that $-d+\Delta_i+\hbar\Delta_e$ defines a differential of degree 1 that squares to zero follows from the previous lemmas.
 The BD relation follows since $\Delta_i$ and $d$ are derivations and by the definition of $\Delta_e$ as the Chevalley Eilenberg differential associated to $\{\_,\_\}$. 
\end{proof}

    By abuse of notation we also denote by $s\mathcal{G}^d\llbracket\hbar\rrbracket$ this (d-2 twisted) BD algebra. 

\begin{df}\label{deq_sg}
Denote by 
\begin{equation}\label{cgdq}
p:s\mathcal{G}^d_{B}\llbracket\hbar\rrbracket\rightarrow {\mathcal{T}^{d}_{B}}
\end{equation}
the composite 
$$Sym\big( s\mathcal{G}_{c,\bullet}^\vee[d-3]\big)\llbracket \hbar\rrbracket\rightarrow Sym^1\big( s\mathcal{G}_{c,\bullet}^\vee[d-3]\big)\cong s\mathcal{G}_{c,\bullet}^\vee[d-3]\rightarrow \mathcal{T}_{c,\bullet}^\vee[d-3]\rightarrow Sym\big(\mathcal{T}_{c,\bullet}^\vee[d-3]\big).$$
Here the first map is the projection on symmetric words of length one and the $\hbar=0$ part.
The third map in the second line is the one induced by the curved map of diagram \ref{inc_proj_2} and the last map is just the inclusion. 
\end{df}
\begin{lemma}
    For every set $B$ 
$$p:s\mathcal{G}^d_{B}\llbracket\hbar\rrbracket\rightarrow {\mathcal{T}^{d}_{B}}$$
defines a dequantization map in the sense of definition \ref{dq}.
\end{lemma}
\begin{proof}
The fact that $p$ intertwines the shifted Poisson bracket follows from remark \ref{nütz_2}.
The fact that $p$ intertwines the BD differential with the differential of ${\mathcal{T}^{d}}$ can be verified directly:
$\Delta_i$ gets send to zero under the projection to ribbon trees since it produces a graph with betti number bigger or equal to one.
Trivially $\hbar\Delta_e$ gets send to zero since it involved a positive power of $\hbar$.
\end{proof}

 Denote by ${G}_{n,g}$ the isomorphism class of the stable ribbon graph with one vertex of loop defect $g$, no internal edges and $n$ leafs.
 Define
 $${G}:=\sum_{n>0}\sum_{g\geq0}G_{n,g}.$$
\begin{theorem}\label{sG_mce}
The corresponding element 
$${G}\in s\mathcal{G}^d\llbracket\hbar\rrbracket$$
is a Maurer-Cartan element.
\end{theorem}
\begin{proof}
Let us denote by $X$ a representative of a stable graph with two vertices and one internal edge connecting these and
by $Y$ a representative of a stable graph that has one vertices and one internal edge (that is thus a loop).
 We have by definition that $$dG_{n,g}=\sum_{(i)} X+\sum_{(ii)} Y. $$   
 Here $(i)$ denotes the sum over all isomorphism classes of stable  graphs $X$ such that when contracting that internal edge according to rule (\ref{cont_g}-1.) the resulting graph is isomorphic to $G_{n,g}$.
 Further $(ii)$ denotes the sum over all isomorphism classes of stable graphs $Y$ such that when contracting that loop according to rule (\ref{cont_g}-2.) we recover $G_{n,g}$.
 On the other hand we have, also by definition \ref{iwjs}, and if $G_{n_1,g_1}\neq G_{n_2,g_2}$ that $$\{G_{n_1,g_1},G_{n_2,g_2}\}=\sum_{(a)}  X $$ 
and the same if we swap the entries of the bracket.
Here $(a)$ indicates that the sum is over all isomorphism classes of graphs $X$ such that when cutting its one edge the resulting stable graph is the disjoint union of $G_{n_1,g_1}$ and $G_{n_2,g_2}$.
Further we have  
$$\{G_{n_1,g_1},G_{n_1,g_1}\}=2\sum_{(a)}  X $$
where $(a)$ stands for the analogous indexing set from before. Then we have that 
$$\Delta_i(G_{n,g})=\sum_{(b)}Y$$
where $(b)$ indicates that the sum is over all isomorphism classes of stable graphs $Y$ such that when cutting its one loop the resulting stable graph is $G_{n,g}$. Trivially it follows that $$\Delta_e(G_{n,g})=0.$$ Next we observe, with the same notation, that 
$$\sum_{n>0} \sum_{g\geq 0}\ \sum_{(i)}X=\sum_{Iso.Cls(X)}X=\frac{1}{2} \sum_{n_1,n_2>0} \sum_{g_1,g_2\geq0}\ 
\sum_{(a)}  X$$
and that
$$\sum_{n>0}\ \sum_{g\geq0}\ \sum_{(ii)}Y=\sum_{Iso.Cls(Y)}Y= \sum_{n>0} \sum_{g\geq0}\ \sum_{(b)}  Y.$$
Here the sum in the middle is over all isomorphism classes of type $X$ and $Y$ respectively, for all $n>0$ and $g\geq0$.
We can regroup these sums in the middle in two different ways, respectively (namely $(i)$ and $(a)$ respectively $(ii)$ and $(b)$).
Summing over the indexing set of this regrouping of the terms in that group thus equals the original sum.

Taking the previous arguments together it follows that
$$dG=\frac{1}{2}\{G,G\}+\Delta_iG+\hbar\Delta_eG,$$
which is what we wanted to show.
\end{proof}
\begin{remark}\label{prak_2}
Under the dequantization map from definition \ref{deq_sg} we have that $p(G)=T$.   
\end{remark}
    Denote by  $s\mathcal{G}^{d,tw}\llbracket\hbar\rrbracket$ the 
    $(d-2)$-twisted BD algebra obtained by twisting $s\mathcal{G}^{d}\llbracket\hbar\rrbracket$ by the Maurer-Cartan element $G$.

Thus by the previous remark we also get a dequantization map
$$p: s\mathcal{G}^{d,tw}\llbracket\hbar\rrbracket\rightarrow \mathcal{T}^{d,tw}. $$
 Note that we can restrict the maps $\Delta_i$ and $\Delta_e$ to graphs.
 Exactly as in theorem \ref{BD_G} we can deduce that this defines a $(d-2)$-twisted BD algebra on the graph complex.
 Furthermore diagram \ref{inc_proj_2} induces a map of BD algebras
$$v:\mathcal{G}^d\llbracket\hbar\rrbracket\rightarrow s\mathcal{G}^d\llbracket\hbar\rrbracket$$
and a dequantization map, defined exactly as in \ref{cgdq}
$$p:\mathcal{G}^d\llbracket\hbar\rrbracket\rightarrow \mathcal{T}^d.$$
Summarizing, we have 
\begin{theorem}
The graph complex $\big(\mathcal{G}^d\llbracket\hbar\rrbracket,\cdot,-d+\Delta_i+\hbar\Delta_e,\{\_,\_\}\big)$
is a $(d-2)$-twisted BD algebra. Furthermore diagram \ref{inc_proj_2} induces a diagram

 \begin{equation}\label{ag_dq} \begin{tikzcd}
\mathcal{T}^d&\mathcal{G}^d\llbracket\hbar\rrbracket\arrow[swap]{l}{p}\arrow{r}{v}&
s\mathcal{G}^d\llbracket\hbar\rrbracket\arrow[bend left=25]{ll}{p}
    \end{tikzcd},
    \end{equation}
    where the arrows pointing to $\mathcal{T}^d$ denote dequantization maps and the third one denotes a map of BD algebras.
\end{theorem}
\begin{remark}
    Note that the Maurer-Cartan element G does not come from an element in $\mathcal{G}^d\llbracket\hbar\rrbracket$ under the map $v$.
\end{remark}
\subsection{`Comm to Assoc'}\label{vsgzg}
Here we explain how the objects from the previous two subsections are connected.
For this we need that $B$, the set decorating the stable ribbon graphs, is finite.
We denote by
\begin{equation}\label{sRGsG}pr: s\mathcal{RG}^B_{c,\bullet}\rightarrow s\mathcal{G}_{c,\bullet} \end{equation}
the map given by sending an oriented B-colored stable ribbon graph to the underlying oriented graph, which we endow with the structure of a stable graph
by defining the loop defect at a vertex $v$ by
$$o(v)=2g(v)+\sum_{\lambda\in B}b_\lambda(v)+c(v)-1.$$
\begin{lemma}\label{fal}
For every finite set $B$
$$pr: s\mathcal{RG}^B_{c,\bullet}\rightarrow s\mathcal{G}_{c,\bullet}$$
is a map of chain complexes.
\end{lemma}
\begin{proof}
This is a slight generalization of the standard fact that the  map \begin{equation}\label{rib_g}
\mathcal{RG}_{\bullet,c}\rightarrow \mathcal{G}_{\bullet,c}\end{equation}  
sending a ribbon graph to its underlying graph is a map of chain complexes. The lemma is proven in the same way by noting that the differential both on $s\mathcal{RG}^B_{c,\bullet}$ and on $s\mathcal{G}_{c,\bullet} $ is given by all ways of contracting an edge, independent of whether it is considered as part of a stable ribbon graph or a stable graph. The fact that the differential is compatible with the other combinatorial structures can be verified by a straightforward case-by-case analysis.
\end{proof}
Let us consider the induced map, which we denote by$$pr^\vee:s\mathcal{G}_{c,\bullet}^\vee\rightarrow {s\mathcal{RG}^B_{c,\bullet}}^\vee.$$
In the dual basis it is given by
\begin{equation}\label{ebieb}pr^\vee(\Gamma)=\sum_{Q}\Bar{\Gamma}
\end{equation}
where $Q$ indicates that we sum over all isomorphism classes of stable ribbon graphs whose underlying stable graph is $\Gamma$ and $\Bar{\Gamma}$ denotes the corresponding class.
\begin{lemma}\label{kId}
We have that  
$$\{pr^\vee(\_),pr^\vee(\_)\}=pr^\vee(\{\_,\_\})\ \
\text{and}\ \  pr^\vee\circ\Delta_i=(\delta_i+\nabla)\circ pr^\vee.$$
\end{lemma}
\begin{proof}
Given $\Gamma_1$ and $\Gamma_2$ stable graphs we have that $\{pr^\vee(\Gamma_1),pr^\vee(\Gamma_2)\}$ is the sum over all stable ribbon graphs such that
when cutting one edge the result is the disjoint union of stable ribbon graph whose underlying stable graphs (in the sense of lemma \ref{fal}) are $\Gamma_1$ respectively $\Gamma_2$.
But this is the same as the sum over all stable ribbon graphs such that when cutting one internal edge of their underlying stable graphs the result is the disjoint union of $\Gamma_1$ respectively $\Gamma_2$.
In other words `cutting an edge and projecting to the underlying stable graph commutes´.
This is the first identity and the second one follows analogously. 
\end{proof}

\begin{cons}
Let us consider an auxiliary map of degree $(d-3)$ induced by the map $pr^\vee$ and suitable shifts 
$$s\mathcal{G}_{c,\bullet}^\vee[d-3]\rightarrow {s\mathcal{RG}^B_{c,\bullet}}^\vee[2d-6].$$
It induces a map 
$$aux:Sym(s\mathcal{G}_{c,\bullet}^\vee[d-3])\rightarrow Sym({s\mathcal{RG}^B_{c,\bullet}}^\vee[2d-6])$$
which has degree $n(d-3)$ on the subspace $Sym^n(s\mathcal{G}_{c,\bullet}^\vee[d-3])$.
\end{cons}
\begin{df}
We define a map of degree zero
\begin{equation}\label{isjk}
\pi: Sym(s\mathcal{G}_{c,\bullet}^\vee[d-3])\llbracket\hbar\rrbracket\rightarrow Sym({s\mathcal{RG}^B_{c,\bullet}}^\vee[2d-6])\llbracket\gamma\rrbracket[3-d]
\end{equation}
by sending 
$$(\Gamma_1\cdot\ldots \cdot\Gamma_n){\hbar}^{n+2k-1}\mapsto aux(\Gamma_1\cdot\ldots \cdot\Gamma_n)\gamma^k,$$
composed by a shift of $(3-d)$ and zero on elements not of this form.
\end{df}
\begin{theorem}\label{GAL}
For every finite set $B$ the map from definition \ref{isjk} is a map of dg shifted Lie algebras
 $$\pi:\Big(s\mathcal{G}^d\llbracket\hbar\rrbracket,-d+\Delta_i+\hbar\Delta_e,\{\_,\_\}\Big)\rightarrow \Big( s\mathcal{RG}^d_B\llbracket\gamma\rrbracket,-d+\nabla+\delta_i+\gamma\delta_e,\{\_,\_\}\Big).$$   
   \end{theorem}
\begin{proof}
 The fact that $\pi$ is compatible with the graph complex differentials and intertwines $\Delta_i$ with $\nabla+\delta_i$ follows directly from the lemma \ref{fal} respectively \ref{kId} and by definitions.
 The fact that $\pi$ is compatible with the brackets follows as they are defined by extending the Lie brackets as a derivation to the symmetric algebra- further by noting that $n_1+2k_1-1+n_2+2k_2-1=(n_1+n_2-1)+2(k_1+k_2)-1,$
 ie the case distinction for when $\pi$ is or is not zero checks out as well.
 Similarly it follows that $\pi$ intertwines $\hbar\Delta_e$ with $\gamma\delta_e$. 
\end{proof}
Further the map $\pi$ is the predual of a 2-weighted map of algebras, in the following sense.

 \begin{proof}

 Let $\tilde{\gamma}^k$ be a variable of degree $k(2d-6)$ and $\tilde{\hbar}^l$ a variable of degree $l(d-3)$, for all $k,l>0$ which we will understand as the dual of $\gamma^k$ respectively $\hbar^l$.
 Define a map of degree zero
$$\pi^{n,k}:Sym^n\big(s\mathcal{RG}^B_{c,\bullet}[6-2d]\big)\tilde{\gamma}^k[d-3] \rightarrow Sym^n(s\mathcal{G}_{c,\bullet}[3-d])\cdot \tilde{\hbar}^{n+2k-1} $$
by $$\pi^{n,k}(\Gamma_1\cdot\ldots \cdot \Gamma_n\tilde{\gamma}^k)=proj(\Gamma_1\cdot\ldots \cdot\Gamma_n)\widetilde{\hbar}^{n+2k-1}$$
composed by a shift of $(d-3)$.
Here $proj$ is defined as the symmetric algebra functor of the map $pr$ composed by suitable shifts.
Thus we have that
\begin{equation}\label{p2w}
\pi^{n,k}(\Gamma_1\cdot\ldots \cdot \Gamma_n\tilde{\gamma}^k)=\pi^{1,0}(\Gamma_1)\cdot\ldots \cdot \pi^{1,0}(\Gamma_n)\widetilde{\hbar}^{n+2k-1}.
\end{equation}
Considering the dual maps and their direct sum over all $n,k$ this induces a map
$$s\mathcal{G}^d\llbracket\hbar\rrbracket\rightarrow s\mathcal{RG}^d_B\llbracket\gamma\rrbracket,$$
which coincides with the map $\pi$ from theorem \ref{GAL}. We say $\pi$ is predual to a 2-weighted map of algebras (recalling definition \ref{weimap}) in the sense that it is
induced by maps satisfying \ref{p2w} via taking duals and then direct sum.
 \end{proof}
 By using the map \ref{rib_g} $$\mathcal{RG}_{\bullet,c}\rightarrow \mathcal{G}_{\bullet,c}$$ we can prove in the same way as in the previous theorem that it
 induces a map (by abuse of notation also denoted $\pi$) as follows:
\begin{theorem}\label{GALt}
For every finite set $B$
    $$\pi: \Big(\mathcal{G}^d\llbracket\hbar\rrbracket,-d+\Delta_i+\hbar\Delta_e,\{\_,\_\}\Big)\rightarrow \Big(\mathcal{RG}^d_B\llbracket\gamma\rrbracket-d+\nabla+\delta_i+\gamma\delta_e,\{\_,\_\}\Big)$$
    is a map of dg $(d-2)$-shifted Lie algebras and the predual to a map of weighted algebras in the sense of the previous theorem.  
\end{theorem}
Further we note that the map \ref{rib_g} restricts to $$\mathcal{RT}_{\bullet,c}\rightarrow \mathcal{T}_{\bullet,c}.$$
We can prove in a similar way (just omitting some previous shifts) that it induces a map (by abuse of notation also denoted $\pi$) as follows:
\begin{theorem}
For every finite set $B$
    $$\pi: \Big(\mathcal{T}^d,\cdot,-d,\{\_,\_\}\Big)\rightarrow \Big(\mathcal{RT}^d_B,\cdot,-d,\{\_,\_\}\Big)$$
     is a map of $(d-2)$-shifted Poisson dg algebras.
\end{theorem}
The diagrams \ref{rbg_comp} and \ref{ag_dq} fit with the maps provided by the three previous theorems together as follows:
    \begin{equation}\label{comp_g_rg}
\begin{tikzcd}
\mathcal{T}^d\arrow{d}{\pi}&\mathcal{G}^d\llbracket\hbar\rrbracket\arrow{d}{\pi}\arrow[swap]{l}\arrow{r}&
s\mathcal{G}^d\llbracket\hbar\rrbracket\arrow{d}{\pi}\arrow[bend right=20]{ll}\\
\mathcal{RT}^d_{B}&\mathcal{RG}^d_{B}\llbracket\gamma\rrbracket\arrow[swap]{l}\arrow{r}&
s\mathcal{RG}^d_{B}\llbracket\gamma\rrbracket\arrow[bend left=20]{ll}.
    \end{tikzcd}
    \end{equation}
    This can be verified by noting that the diagrams \ref{inc_proj} and \ref{inc_proj_2}
    fit together with the maps induced by \ref{sRGsG} into a commutative diagram and by construction. 

Using the explicit description of the maps $\pi$ from equation \ref{ebieb}
it is immediate to see that for the Maurer-Cartan elements from the previous sections we have that 
 $$\pi(G)=S\ \text{and}\ \pi(T)=D.$$

Since with the previous remark and remarks \ref{nütz} and \ref{nütz_2} we have thus seen that
all the respective Maurer-Cartan elements are compatible we can twist the outer square of \ref{comp_g_rg} by those to obtain following
commutative diagram where the horizontal maps are dequantization maps 
     $$\begin{tikzcd}
     s\mathcal{G}^{d,tw}\llbracket\hbar\rrbracket\arrow{d}{\pi}\arrow{r}{p}&\mathcal{T}^{d,tw}
\arrow{d}{\pi}\\
    s\mathcal{RG}^{d,tw}_{B}\llbracket\gamma\rrbracket\arrow{r}{p} &\mathcal{RT}^{d,tw}_{B}.    
     \end{tikzcd}$$

\section{Generalized Kontsevich Cocycle Map}
In this section we explain how to connect the ribbon graph world, section \ref{rbgs}, with the world of cyclic $A_\infty$-categroies and their quantizations, section \ref{ncw}. We do so using a slight generalization of Kontsevich's cocycle construction \cite{Kon92b}, see also \cite{AmTu15}.
\subsection{Non-Commutative World}\label{ksc}

    Given $V_B$ a collection of cyclic chain complexes of odd degree $d$ together with $I^q\in MCE (\mathcal{F}^{pq}(V_B))$ we construct a map
    $$\rho_{I^q}: s\mathcal{RG}_B^d\llbracket\gamma\rrbracket\rightarrow \mathcal{F}^{pq}(V_B).$$

We explain how to to so in multiple steps, partly inspired by \cite{AmTu15}.
We begin by associating certain tensors, denoted $I^q|_{v_i}$, to each vertex $v_i$ of a connected B-colored stable ribbon graph.
\begin{cons}
  Let $\Gamma$ be a connected B-colored stable ribbon graph and $v_i$ one of its vertices:\begin{itemize}
      \item 
 To such a vertex $v_i$ of $\Gamma$ is associated a genus defect $g_i\in\mathbb{N}_{\geq 0}$, a boundary defect $b_\lambda(v_i)\in\mathbb{N}_{\geq 0}$ for each $\lambda\in B$.
 Further the half-edges of that vertex are partitioned into cycles $c_1,\cdots, c_{n}$.
 Two consecutive half-edges in a cycle together belong to a unique free boundary of $\Gamma$ and thus are associated an element of $B$.
 Considering all the half-edges in a cycle thus defines a cyclic word in elements of $B$ and all the cycles taken together a `symmetric' word in cyclic word in letters from $B$. 
 \item  Given a vertex $v_i$ of $\Gamma$ we define a map (which is not graded), denoted
  \begin{equation}\label{iso_sc}
(\_)|_{v_i}: Sym(Cyc^*(V_B)[d-4])\llbracket\gamma\rrbracket[3-d]\rightarrow Sym(Cyc^*_+(V_B)[-1]).\end{equation}
It is given by first a shift by $(3-d)$, composed with the map  
$$Sym(Cyc^*(V_B)[d-4])\llbracket\gamma\rrbracket\rightarrow Sym(Cyc^*(V_B)[-1])\llbracket\gamma\rrbracket$$
induced by {$Cyc^*(V_B)[d-4]\rightarrow Cyc^*(V_B)[-1]$}.
Lastly we project to the subspace of $Sym(Cyc^*_+(V_B)[-1])$ determined by the symmetric word in cyclic words in letters of $B$ and 
restrict to the power $\gamma^{g_i}$ and the power of $\Pi_{\lambda\in B}\nu_\lambda^{b_\lambda(v_i)}$.
\item Note that for $I^q\in MCE(\mathcal{F}^{pq}(V_B))$, which has degree $(3-d)$, we have that 
$${|{I}^q|_{v_i}|=(3-d)(2-2g_i-c(v_i)-\sum_{\lambda\in B}b_\lambda(v_i))}.$$
\end{itemize}
\end{cons}
Next we explain how the datum of an oriented stable ribbon graph together with some auxiliary choices
defines a permutation of the half-edges of such a stable ribbon graph:
\begin{cons}\label{perm}
 Assume that $\Gamma$  has $n$ vertices, each with $c(v_i)$ cycles.
 \begin{itemize}
     \item 
Chose an ordering of its vertices, at each vertex $v_i$  an ordering of the cycles and for each such cycle an half-edge.
This gives an ordering of all the half edges at that vertex (using the cyclic orderings of the half edges within each cycle) denoted 
$$h_{v_i,1}\cdots h_{v_i,|v_i|}$$
and with the ordering of the vertices together defines an ordering of all half edges of $\Gamma$
$$h_{v_1,1}\cdots h_{v_1,|v_1|}h_{v_2,1}\cdots h_{v_r,|v_r|}.$$
\end{itemize}
Next assume that $\Gamma$ is orientated, ie we have an ordering $(e_1\cdots e_n)$ of the internal edges.
\begin{itemize}
    \item 
Chose an ordering of the half edges belonging to a given internal edge.
Together with the orientation this gives an ordering of the half edges belonging to the internal edges
$$h_{1_x}h_{1_y}\cdots h_{n_x} h_{n_y}.$$ 
\end{itemize}
Given the choices made and given the orientation we can define a permutation of all the half-edges as follows
$$
\sigma: H(\Gamma)\rightarrow H(\Gamma) $$
$$
h_{v_1,1}\cdots h_{v_1,|v_1|}h_{v_2,1}\cdots h_{v_r,|v_r|}\mapsto h_{1_x}h_{1_y}\cdots h_{n_x} h_{n_y}l_{1,1}\cdots l_{1,m_1}l_{2,1}\cdots l_{b,m_{f_{>0}}} $$
\begin{itemize}
    \item 

We begin the ordering of the leaves on the right hand side by the first leave that appears on the left hand side.
This leafs belongs to a face. We  use the cyclic ordering within that face to order the remaining half edges that belong to it.
Then we go to the first half edge that is a leaf and does not lie in that first face.
Again we use the cyclic ordering within that face to order the remaining half edges that belong to it.
And so on till we come to the $f_{>0}$-th face with non-zero open boundaries.
\end{itemize}
\end{cons}
Recall the definition of $\widetilde{(V_B)}$, the cyclic chain complex associated to the collection of cyclic chain complexes $V_B$, see \ref{cyccomm}. Denote $W:=\widetilde{(V_B)}^\vee[-1]$.
We still fix the same oriented stable ribbon graph $\Gamma$ and the choices of orderings made.
We next explain how to define, given these choices, a map $$W^{\otimes h(\Gamma)}\rightarrow Sym(Cyc^*_+({V_B})[d-4])$$ of degree $e(\Gamma)(d-2)+f_{>0}(\Gamma)(d-3)$:

\begin{cons}\label{con_proj}
The permutation $\sigma$ induces a map which we denote by abuse of notaion by the same letter
$${\sigma}:W^{\otimes h(\Gamma)}\rightarrow W^{\otimes h(\Gamma)}\cong W^{\otimes 2e(\Gamma)}\otimes W^{\otimes l(\Gamma)}.$$
We say an element in $W^{\otimes n}$ has `matching boundary' conditions if it is the sum of tensors having matching boundary conditions. 
\begin{itemize}
    \item 
To any edge $e$ of $\Gamma$ are associated two elements of $B$, say $i$ and $j$. Denote by 
$\langle\_,\_\rangle^{-1}_{e}: W^{\otimes 2}\rightarrow k[d-2]$
the map given by sending elements to zero which do not have the matching boundary condition {ij} and
by sending elements which do have the matching boundary condition {ij} to the number given by applying $\langle\_,\_\rangle^{-1}_{ij}$ to them\footnote{The ordering of ${ij}$ does not matter since $\langle\_,\_\rangle^{-1}_{ij}$ is symmetric.}.
\item A stable ribbon graph $\Gamma$ has $f$ faces, $f_{>0}$ of those with non-zero open boundaries lying in that face. Further say that the i-th such face 
(determined by the ordering of all half-edges given made choices, as before) has $m_i$ open boundaries. Denote by 
$$proj: W^{\otimes l(\Gamma)}\rightarrow Sym(Cyc^*_+({V_B})[d-4])$$
the map described as follows:
It is zero on elements not of following matching boundary conditions: We ask that the first $m_1$ elements have the matching boundary condition determined by 
the first face of the graph with nonzero leafs, the next $m_2$ elements have the matching boundary condition determined by the second face of the graph with nonzero leafs etc till the last $m_{f_{>0}}$ elements. 
Then the map is given by
 $$proj=sym^{f_{>0}}(cyc_{m_1}\otimes\cdots\otimes cyc_{m_{f_{>0}}}).$$
Here $cyc_{m_i}$ denotes the projection to cyclic words of length $m_i$ with boundary conditions determined by the i-th face as above,
shifted by $(d-3)$ and $sym$ denoted the projection to symmetric words in cyclic words. Counting together $proj$ has degree $f_{>0}(\Gamma)(3-d)$.
\end{itemize}
Thus we can define a map of degree $e(\Gamma)(d-2)-f_{>0}(\Gamma)(d-3)$
$$\bigotimes_{e\in E(\Gamma)}\langle\_,\_\rangle^{-1}_e\otimes proj:\ W^{\otimes 2e(\Gamma)}\otimes W^{\otimes l(\Gamma)}\rightarrow Sym(Cyc^*_+({V_B})[d-4]).$$
\end{cons}
\begin{df}
Fixing the same oriented stable ribbon graph $\Gamma$ and $I^q\in MCE (\mathcal{F}^{pq}(V_B))$ as in the previous constructions and the choices made we define 
\begin{equation}\label{main_constr_nc}
\rho_{I^q}(\Gamma)=(\bigotimes_{e\in E(\Gamma)}\langle\_,\_\rangle^{-1}_e\otimes proj)\circ{\sigma}\big(I^q|_{v_1}\otimes\ldots\otimes I^q|_{v_n}\big)\nu^{f_{tot}(\Gamma)}\gamma^{g(\Gamma)}\in Sym(Cyc^*({V_B})[d-4])\llbracket\gamma\rrbracket.
\end{equation}
Here 
$$\nu^{f_{tot}(\Gamma)}:=\prod_{\lambda\in B,v\in V(\Gamma)}\nu^{b_\lambda(v)}\prod_{l\in F_0(\Gamma)}\nu_{\lambda(l)},$$
where we recall that $F_0(\Gamma)$ denotes the set of faces without open boundaries,
each of which is thus colored by one element of $B$.
Further we used the standard isomorphism between invariants and coinvariants in characteristic zero to view $I^q|_{v_i}$ as an invariant tensors, which we can thus see as an element of $W^{\otimes r},$
for $r=H(v_i)$, the valency of the vertex $v_i$.
The order of these tensors is determined by the choice of ordering of vertices, the order of the $\nu$ does not matter since they have even degree.
\end{df}
\begin{lemma}
   Given $I^q\in MCE (\mathcal{F}^{pq}(V_B))$ the assignment \ref{main_constr_nc} defines a map
\begin{equation}\label{Kon_gen}
\begin{split}
\rho_{I^q}:{s\mathcal{RG}^B_{c,\bullet}}^\vee[2d-6]&\rightarrow Sym(Cyc^*(V_B)[4-d])\llbracket\gamma\rrbracket,\\
\Gamma&\mapsto \rho_{I^q}(\Gamma).
\end{split}
\end{equation}
In particular, it is independent of the choices made.
\end{lemma}
\begin{proof}
The fact that the degrees work out follows by counting together \begin{equation*} 
\begin{split}&-e+(d-2)e-f_{>0}(d-3)+(d-3)(-2v+c+\sum_{\lambda\in B,v\in V}b_\lambda(v)+2\sum_{v_i}g_i-2g-\sum_{\lambda\in B,v\in V}b_\lambda(v)-f_{0})\\
&=(d-3)(-2g-2v+e+c-f+2\sum_{v_i}g_i)=6-2d
\end{split}\end{equation*}
and definition \ref{armg} of the arithmetic genus. 
This map is independent of the choices made as follows:
\begin{itemize}
    \item Let $\Gamma$ be a stable ribbon graph as in the previous construction. Assume that we chose a different ordering of vertices, of the cycles belonging to each vertex and of a half-edge in each cycle,
    but for now take the same orientation datum and ordering of half-edges belonging to each internal edge.
    Let us call the initial permutation as in construction \ref{perm} $\sigma_1:H(\Gamma)\rightarrow H(\Gamma)$ and the new one $\sigma_2:H(\Gamma)\rightarrow H(\Gamma)$.
    Then we have that 
    $${\sigma_1}\big(I^q|_{v_1}\otimes\ldots\otimes I^q|_{v_n}\big)=\tau\circ{\sigma_2}\big(I^q|_{v_{\Tilde{1}}}\otimes\ldots\otimes I^q|_{v_{\Tilde{n}}}\big),$$
    where on the right-hand-side we use the new ordering of the vertices to order the input tensors.
    Further $\tau$ is the map determined by the following permutation: It is the identity on the half-edges belonging to internal edges.
    On the half-edges belonging to leafs it can be written as a permutation on $f_{> 0}(\Gamma)$ letters applied to $f_{> 0}(\Gamma)$ cyclic permutations (with regard to the cyclic ordering of the leafs in each face of the graph).
    Further there is no sign showing up because all the $|I^q(v_i)|$ are even.
    
    Since in the definition of $\rho_{I^q}$, described in the second bullet point of construction \ref{con_proj},
    we project to the quotient determined by exactly such permutations $\tau$ we have that $\rho_{I^q}$ is independent of 
    a different ordering of vertices, of the cycles belonging to each vertex and of a half-edge in each cycle.
    \item Now assume that we chose a different ordering of the half-edges belonging to an internal edge respectively.
    Since the pairings $\langle\_,\_\rangle^{-1}_{ij}$ are symmetric it follows that this gives the same result $\rho_{I^q}$.
  \end{itemize}  
Lastly this map is compatible with the orientation: Since the pairings have odd degree if
we change the orientation by an odd permutation the result changes by a factor of $(-1)$.
\end{proof}
\begin{lemma}\label{bra_comp}
    Given $I^q\in MCE (\mathcal{F}^{pq}(V_B))$ we have that $\rho_{I^q}(\{\_,\_\})=\{\rho_{I^q}(\_)\rho_{I^q}(\_)\}.$

\end{lemma}
\begin{proof}
    Recall that the bracket on $Sym(Cyc^*({V_B})[d-4])\llbracket\gamma\rrbracket$ is defined by pairing each letter of a symmetric word in cyclic words with a letter of a symmetric word in cyclic words exactly once,
    by that joining two cyclic words and taking into account the Koszul rule. Furthermore, if two cyclic words consists of just one letter respectively we obtain
    a $\nu_\lambda$ weighted by the pairing of the two letters (which have matching boundary condition $(\lambda,\lambda)$, otherwise we get zero).

    On the other hand we can write 
    $$\{\Gamma_1,\Gamma_2\}=\sum_{l_1\in L(\Gamma_1),l_2\in L(\Gamma_1)} \frac{\Gamma_1\cup_{l_1,l_2} \Gamma_2}{c_{l_1}(\Gamma_1)c_{l_2}(\Gamma_2)},$$
    the sum over all leafs of the graphs and gluing the two graphs together along those (if they have matching boundary condition). Here $c_{l_1}(\Gamma_i)=n\in \mathbb{N}$ if there are automorphisms of $\Gamma_i$ that send $l_1$ to $n$ other leafs. 

    By definition \ref{main_constr_nc} we see that $\sum_{l_1\in L(\Gamma_1),l_2\in L(\Gamma_1)} \rho(\Gamma_1\cup_{l_1,l_2} \Gamma_2)$
    is given by pairing the tensor $(\bigotimes_{e\in E(\Gamma_1)}\langle\_,\_\rangle^{-1}_e)\circ{\sigma}\big(I^q|_{v_1}\otimes\ldots\otimes I^q|_{v_n}\big)$ and 
    the tensor $(\bigotimes_{e\in E(\Gamma_2)}\langle\_,\_\rangle^{-1}_e)\circ{\sigma}\big(I^q|_{v_1}\otimes\ldots\otimes I^q|_{v_n}\big)$
    together in all possible ways, using a letter $l_1$ and $l_2$ from each word and then projecting to symmetric words in cyclic words.
    However, compared to $\{\rho(\Gamma_1),\rho(\Gamma_2\}$ we are overcounting exactly $c_{l_1}(\Gamma_1)c_{l_2}(\Gamma_2)$ terms.
    Further, if both leafs that we glue come from a face with only one leaf respectively, we increase the number
    of empty faces by one and thus get one additional $\nu$ according to formula \ref{main_constr_nc}. This proves the result.
\end{proof}
\begin{lemma}\label{diff_comp}
Given $I^q\in MCE (\mathcal{F}^{pq}(V_B))$ we have that 
$\rho_{I^q}\circ \nabla=\nabla\circ\rho_{I^q}$ and
$\rho_{I^q}\circ \delta_i=\gamma\delta_i\rho_{I^q}.$
\end{lemma}
\begin{proof}
This follows similarly to above lemma.
In the second equality the additional $\gamma$ comes from the fact that $\delta_i$ increases the genus by one.
\end{proof}
\begin{lemma}\label{diff_comp2r}
 Given $I^q\in MCE (\mathcal{F}^{pq}(V_B))$ we have that $\rho_{I^q}\circ (-d)=d\circ \rho_{I^q}$.
\end{lemma}
\begin{proof}
    We have 
    $$\rho_{I^q}(-d\Gamma)=-\rho_{I^q}(\sum_{(*)}\Gamma')=\sum_{(*)}(\bigotimes_{e\in E(\Gamma')}\langle\_,\_\rangle^{-1}_e\otimes proj)\circ\sigma\circ'(I^q|_{v_1}\otimes\ldots \otimes X\otimes\ldots I^q|_{v_n})\nu^{f_{tot}(\Gamma)}\gamma^{g(\Gamma)}.$$
    Here $(*)$ indicates that the sum is over isomorphism classes of stable ribbon graphs $\Gamma'$ such that when contracting one of their legs according to rule \ref{cont} we recover $\Gamma$.
    Say $\Gamma$ has $n$ vertices.
    Note that a graph $\Gamma'$ as before has at most n+1 vertices and $n-1$ of those actually `coincide' with $n-1$ ones of $\Gamma$; let us call the vertex of $\Gamma$ which is not matched $v_i$.
    That is how the second equality above follows, where we denote by $X$ the tensor that gets associated to the at most two vertices which are different (and which we assume wlog to be numbered consecutively).
    Note further that $d$ preserves the arithmetic genus and the total number of empty colored faces.
    
    According to the rule how to contract \ref{cont} the sum over $(*)$ splits up into three summands: 
    \begin{enumerate}
        \item 
    Similarly to lemma \ref{bra_comp} we have that for the first summand, corresponding to rule \ref{item_contractedge} of definition \ref{cont}
    $$\sum_{(*.1)}(\bigotimes_{e\in E(\Gamma')}\langle\_,\_\rangle^{-1}_e\otimes proj)\circ\sigma'(\cdots\otimes X\otimes\cdots )=(
    \bigotimes_{e\in E(\Gamma)}\langle\_,\_\rangle^{-1}_e\otimes proj)\circ\sigma(\cdots\otimes  \frac{1}{2}\{I^q,I^q\}|_{v_i}\otimes\cdots).$$
 Note that here it is important that we had redefined the bracket on words of length one, see \ref{ce}, which corresponds to rule \ref{item_contractedge}, second case distinction.
 \item Again similarly to lemma \ref{diff_comp} we have that for the second summand, corresponding to rule \ref{item_contractloop1} of definition \ref{cont} 
 $$\sum_{(*.2)}(\bigotimes_{e\in E(\Gamma')}\langle\_,\_\rangle^{-1}_e\otimes proj)\circ\sigma'(\cdots\otimes X\otimes\cdots )=(
    \bigotimes_{e\in E(\Gamma)}\langle\_,\_\rangle^{-1}_e\otimes proj)\circ\sigma(\cdots\otimes  (\gamma\delta I^q)|_{v_i}\otimes\cdots).$$
Here we need to include a $\gamma$ since for this summand we have decreased the genus defect by one.
\item Further we have that for the third summand, corresponding to rule \ref{item_contractloop2} of definition \ref{cont} that 
$$\sum_{(*.3)}(\bigotimes_{e\in E(\Gamma')}\langle\_,\_\rangle^{-1}_e\otimes proj)\circ\sigma'(\cdots\otimes X\otimes\cdots )=(
    \bigotimes_{e\in E(\Gamma)}\langle\_,\_\rangle^{-1}_e\otimes proj)\circ\sigma(\cdots\otimes  \nabla I^q|_{v_i}\otimes\cdots).$$

 Again it is crucial how we defined $\nabla$ in definition \ref{nabla} with regards to the central extension
 when contracting consecutive letters and when a word consists of just two letters, which matches the second and third case distinction of rule \ref{item_contractloop2}.
 \end{enumerate}
 Since we have that $I^q$ is a Maurer-Cartan element, ie in particular
 $$dI^q|_{v_i}=-(\frac{1}{2}\{I^q,I^q\}|_{v_i}+\nabla I^q|_{v_i}+\gamma\delta I^q|_{v_i})$$
 it follows that above sum is equal to
 $$=\sum_{i=1}^n(\bigotimes_{e\in E(\Gamma)}\langle\_,\_\rangle^{-1}_e\otimes proj)\circ\sigma\circ(I^q|_{v_1}\otimes\ldots dI^q|_{v_i}\cdots I^q|_{v_n})\gamma^{g(\Gamma)},$$
 which we can rewrite, since both $proj$ and $(\langle\_,\_\rangle^{-1}$ are chain maps, as 
 
 $$=d\circ\big(\bigotimes_{e\in E(\Gamma)}\langle\_,\_\rangle^{-1}_e\otimes proj\big)\circ\sigma\circ(I^q|_{v_1}\otimes\cdots \otimes I^q|_{v_n})\gamma^{g(\Gamma)}=d\rho(\Gamma).$$ \end{proof}
\begin{cons}\label{ext_sym}
    We extend the map \ref{Kon_gen} as a map of algebras to 
$$Sym({s\mathcal{RG}^B_{c,\bullet}}^\vee[2d-6])\rightarrow Sym(Cyc^*(V_B)[d-2])\llbracket\gamma\rrbracket,$$
and finally shift both sides to obtain the map
\begin{equation}\label{sdt}
Sym({s\mathcal{RG}^B_{c,\bullet}}^\vee[2d-6])[3-d]\rightarrow Sym(Cyc^*(V_B)[d-2])\llbracket\gamma\rrbracket[3-d].\end{equation}
which by abuse of notation we denote by the same letter.
\end{cons}
\begin{theorem}\label{Kc2}
    For every set $B$ the map \ref{sdt} defines a map of $(d-2)$-twisted BD algebras
    $$\rho_{I^q}: s\mathcal{RG}_B^d\llbracket\gamma\rrbracket\rightarrow \mathcal{F}^{pq}(V_B),$$
    which is functorial with respect to strict cyclic $A_\infty$-morphisms of quantum $A_\infty$-categories. 
\end{theorem}
\begin{proof}  \begin{itemize}
    \item 
  It is a map of algebras (whose multiplication have non-zero even degree) by construction.
\item By the first part of lemma \ref{diff_comp} and construction $\rho_{I^q}$ intertwines
the maps on both sides that we called suggestively $\nabla$, further by lemma \ref{diff_comp2r} it intertwines the internal differentials.
\item We proved that map \ref{Kon_gen} is a map of shifted Lie algebras, in fact it is a map of a shifted Lie algebra to a shifted Poisson algebra.
Any such map determined a map of shifted Poisson algebras from the free algebra on the domain to the target and this is indeed how we defined the shifted Poisson structure on 
$s\mathcal{RG}_B^d\llbracket\gamma\rrbracket$ (up to a total shift).
\item Furthermore this implies that $\rho_{I^q}$ intertwines $\delta_e$ with the associated Chevalley-Eilenberg differential of the 
Lie algebra $\mathcal{F}^{pq}(V_B)$. This, together with the arguments in bullet point two imply
$$\rho_{I^q}\big((-d+\gamma\delta_i+\gamma\delta_e+\nabla)(\Gamma_1\cdot\ldots\cdot \Gamma_n)\big)=(d+\gamma\delta+\nabla)\rho_{I^q}(\Gamma_1\cdot\ldots\cdot \Gamma_n),$$
which finishes the first part of the proof.
\item Given $F:V_B\rightarrow V_C$ a strict cyclic $A_\infty$-morphisms of quantum  $A_\infty$-categories we have in particular a map of sets $B\rightarrow C$, which induces a map 
$$s\mathcal{RG}_C^d\llbracket\gamma\rrbracket\rightarrow s\mathcal{RG}_B^d\llbracket\gamma\rrbracket.$$
Since $F$ is strict and cyclic and we have by definition that $F^*I^q=J^q$ for the respective quantizations one shows that
$$\begin{tikzcd}
 \mathcal{F}^{pq}(V_C)\arrow{r}{F^*}&   \mathcal{F}^{pq}(V_B)\\
 s\mathcal{RG}_C^d\llbracket\gamma\rrbracket\arrow{r}\arrow{u}& s\mathcal{RG}_B^d\llbracket\gamma\rrbracket\arrow{u}
\end{tikzcd}$$
commutes, using the definition \ref{main_constr_nc} of the vertical maps.
\end{itemize}
\end{proof}
Given an $I\in MCE(\mathcal{F}^{fr}(V_B))$, we can construct a map
$$\rho_{I}:{\mathcal{RG}^B_{c,\bullet}}^\vee[2d-6]\rightarrow Sym(Cyc^*(V_B)[d-2)\llbracket\gamma\rrbracket.$$
It is given by almost the exactly the same construction, just that we modify the map \ref{iso_sc} by omitting all shifts. 
Furthermore, in the same way as in lemma \ref{diff_comp} and \ref{diff_comp2r} we can prove that this maps intertwines the algebraic structures.
In the same way as in construction \ref{ext_sym} we get a map, also denoted
\begin{equation}\label{sigs}
\rho_{I}:Sym({\mathcal{RG}^B_{c,\bullet}}^\vee[2d-6])\llbracket\gamma\rrbracket[3-d]\rightarrow Sym(Cyc^*(V_B)[d-4])\llbracket\gamma\rrbracket[3-d]\end{equation}
for which we can argue exactly the same as in theorem \ref{Kc2} that we have: 
\begin{theorem}\label{Kc2.1}
Given $I\in MCE(\mathcal{F}^{fr}(V_B))$ the map \ref{sigs} is a map of $(d-2)$-twisted BD algebras
    $$\rho_{I}: \mathcal{RG}_B^d\llbracket\gamma\rrbracket\rightarrow \mathcal{F}^{pq}(V_B),$$
    functorial with respect to strict cyclic $A_\infty$-morphisms.
\end{theorem}
Again using the same construction, but omitting some shifts, given $I\in MCE(\mathcal{F}^{fr}(V_B))$, we can construct a map
$$\rho_{I}:{\mathcal{RT}^B_{c,\bullet}}^\vee[d-3]\rightarrow (Cyc^*_+(V_B)[-1]).$$
Indeed we land in $Cyc^*_+(V_B)[-1]$ since ribbon trees have by definition genus zero and only one face, which needs to have a non-zero number of open boundaries. 
In the same way as in lemmas \ref{bra_comp} and \ref{diff_comp2r} we can argue that this map is compatible with the algebraic structures. 
As in construction \ref{ext_sym} we get a map, also denoted
\begin{equation}\label{www}
\rho_{I}:Sym({\mathcal{RT}^B_{c,\bullet}}^\vee[3-d])\rightarrow Sym(Cyc^*_+(V_B)[-1])
\end{equation}
for which we can argue analogously to same as in theorem \ref{Kc2} that we have:
\begin{theorem}\label{Kc1}
Given $I\in MCE(\mathcal{F}^{fr}(V_B))$ the map \ref{www} is a map of $(d-2)$-shifted Poisson algebras
    $$\rho_{I}: \mathcal{RT}_B^d\rightarrow \mathcal{F}^{fr}(V_B),$$
    functorial with respect to strict cyclic $A_\infty$-morphisms.
\end{theorem}
Theorems \ref{Kc2.1}, \ref{Kc2} and \ref{Kc1} fit together with 
diagram \ref{rbg_comp} and map \ref{nc_dq} together:
\begin{theorem}\label{comp_rg_al}
Given $I^q\in MCE(\mathcal{F}^{pq}(V_B))$ denote by $I=p(I^q)\in MCE(\mathcal{F}^{fr}(V_B))$ its dequantization. Then following diagram commutes:
$$\begin{tikzcd}
\mathcal{F}^{fr}(V_B)&\mathcal{F}^{pq}(V_B)\arrow{l}{p}&
\mathcal{F}^{pq}(V_B)\arrow[equal]{l}\\
\mathcal{RT}^d_{B}\arrow{u}{\rho_{I}}&\mathcal{RG}^d_{B}\llbracket\gamma\rrbracket\arrow[swap]{l}\arrow{r}\arrow{u}{\rho_{I}}&
s\mathcal{RG}^d_{B}\llbracket\gamma\rrbracket\arrow[bend left=20]{ll}\arrow{u}{\rho_{I^q}}.
    \end{tikzcd}$$
\end{theorem}
\begin{proof}
    This basically follows by construction. Note that it is essential that the dequantization maps
    project to symmetric words of length one, both for (stable) ribbon graphs and for cyclic words. 
\end{proof}
Recall the two Maurer-Cartan elements from \ref{RS_mce} and \ref{sRG_mce} in the respective graph complexes. We state following lemmas, which are straightforward to prove:
\begin{lemma}\label{comp_1}
Let $B$ be a set and given $I^q\in MCE(\mathcal{F}^{pq}(V_B))$ we have that $\rho_{I^q}(S)=I^q$.
\end{lemma}
\begin{lemma}\label{comp_2}
   Let $B$ be a set and given $I\in MCE(\mathcal{F}^{pq}(V_B))$ we have $\rho_I(D)=I.$
\end{lemma}
Thus we can induce the following theorem:
 \begin{theorem}\label{aiaiai}
Given a quantization $I^q$ of a cyclic $A_\infty$-category $\mathcal{C}$ there is a commutative diagram as follows:        
    $$\begin{tikzcd}
      \mathcal{F}^{cl}(\mathcal{C})&  \mathcal{F}^{q}(\mathcal{C})\arrow{l}\\
      \mathcal{RT}^{d,tw}_{Ob\mathcal{C}}\arrow{u}{\rho_{I}}&s\mathcal{RG}^{d,tw}_{Ob\mathcal{C}}\llbracket\gamma\rrbracket\arrow{u}{\rho_{I^q}}\arrow{l}.
    \end{tikzcd}$$
    Here the horizontal maps are dequantization maps, the right vertical map is a map of $(d-2)$-twisted BD algebras and 
    the left vertical maps is a map of $(d-2)$-shifted Poisson algebras
    \end{theorem} 
    \begin{proof}
  Indeed this follows from the previous theorem and definitions and the consecutive two lemmas by twisting the outer diagram by the compatible Maurer-Cartan elements.      
    \end{proof}
Lastly assume we are given an essentially finite (in the sense of definition \ref{essfin}) cyclic $A_\infty$-category $\mathcal{C}$ with two choices of skeleta $(\mathcal{C},Sk\mathcal{C}_1)$ and $(\mathcal{C},Sk\mathcal{C}_2)$.
Further assume that we are given compatible quantizations of these three categories in the sense of definition \ref{qesfq}.
Let us denote by $\widebar{\mathcal{C}}$ the full subcategory on the images of $Sk\mathcal{C}_1$ and $Sk\mathcal{C}_2$.
$\widebar{\mathcal{C}}$ can be endowed with the structure of a quantum $A_\infty$-category, it has finitely many objects.
Its quantization gives another finite presentation of the quantization of $\mathcal{C}$ and it is compatible with the ones by $\mathcal{C}_1$ and $\mathcal{C}_2$ by definition.
We summarize the situation by following diagram, where all the maps are equivalences of quantum $A_\infty$-categories in the sense of \ref{BDqi}.
\begin{equation}\label{dall}\begin{tikzcd}
    &\mathcal{C}&\\
   Sk\mathcal{C}_1\arrow[hookrightarrow]{ur}\arrow[hookrightarrow]{r}&\widebar{\mathcal{C}}\arrow[hookrightarrow]{u}& Sk\mathcal{C}_2\arrow[hookrightarrow]{ul}\arrow[hookrightarrow]{l}
\end{tikzcd}\end{equation}

Then, the vertical right maps of theorem \ref{aiaiai} induced from the quantization of $\mathcal{C}_1$ and from the quantization of $\mathcal{C}_2$ are compatible as follows:
 \begin{equation}\label{unabhai}
\begin{tikzcd}
      \mathcal{F}^{q}(Sk\mathcal{C}_1)&  \mathcal{F}^{q}(\widebar{\mathcal{C}})\arrow{l}[swap]{\simeq}\arrow{r}{\simeq}&\mathcal{F}^{q}(Sk\mathcal{C}_2)\\
      s\mathcal{RG}^{d,tw}_{ObSk\mathcal{C}_1}\llbracket\gamma\rrbracket\arrow{u}&s\mathcal{RG}^{d,tw}_{Ob\widebar{\mathcal{C}}}\llbracket\gamma\rrbracket\arrow{u}\arrow{r}\arrow{l}&s\mathcal{RG}^{d,tw}_{ObSk\mathcal{C}_2}\llbracket\gamma\rrbracket\arrow{u}.
    \end{tikzcd}\end{equation}
Note that here the left and right horizontal map are just given by extending by zero on elements which are not colored by $Sk\mathcal{C}_1$ respectively $Sk\mathcal{C}_2$.
Further $s\mathcal{RG}^{d,tw}_{ObSk\mathcal{C}_1}\llbracket\gamma\rrbracket$ and 
$s\mathcal{RG}^{d,tw}_{ObSk\mathcal{C}_2}\llbracket\gamma\rrbracket$ are in fact (non-canonically) isomorphic, since $ObSk\mathcal{C}_1$ and $ObSk\mathcal{C}_2$ are isomorphic sets, both being the skeleton of the same category.
Thus \ref{unabhai} tells us the left and right vertical map are equivalent in the derived category, up to extending these maps by zero to a bigger domain.

\subsection{Commutative World}\label{cksc}
In this section we explain how to connect the  graph complex world, section \ref{grs}, with the world of cyclic $L_\infty$-algebras and their quantizations, section \ref{cw}.
We do so using a slight generalization of Penkava's cocycle construction \cite{pe96}, generalizing it to graphs with leafs.

 Given $V$ a cyclic chain complex of odd degree $d$ together with $I^q\in MCE (\mathcal{O}^{pq}(V))$ we construct a map 
 $$\theta_{I^q}: s\mathcal{G}^d\llbracket\hbar\rrbracket\rightarrow \mathcal{O}^{pq}(V).$$

We explain how to to so in multiple steps.
We begin by associating certain tensors, denoted $I^q|_{v_i}$, to each vertex $v_i$ of a connected stable graph.
\begin{cons}
  Let $\Gamma$ be a connected stable graph and $v_i$ one of its vertices:\begin{itemize}
      \item 
 To such a vertex $v_i$ of $\Gamma$ is associated a loop defect $o_i\in\mathbb{N}_{\geq0}$ by definition, see \ref{sgrdf}. 
 \item  Given a vertex $v_i$ of $\Gamma$ we define a map of degree $(d-3)o_i$, denoted
  \begin{equation}\label{iso_scg}
(\_)|_{v_i}:\mathcal{O}^{pq}(V)\rightarrow C^*(V)\end{equation}
by restricting to the power $\hbar^{o_i}$.
\item Thus for $I^q\in MCE(\mathcal{O}^{pq}(V))$, which has degree $(3-d)$, we have that $${|{I}^q|_{v_i}|=(3-d)(1-o_i)}.$$
\end{itemize}
\end{cons}
Next we explain how the datum of an oriented stable graph together with some auxiliary choices defines a permutation of the half-edges of such a stable  graph:
\begin{cons}\label{perm_g}
 Assume that $\Gamma$  has $n$ vertices.
 \begin{itemize}
     \item 
Chose an ordering of its vertices and at each vertex $v_i$  an ordering of the half-edges, denoted, 
$$h_{v_i,1}\cdots h_{v_i,|v_1|}.$$
With the ordering of the vertices together this defines an ordering of all half edges of $\Gamma$
$$h_{v_1,1}\cdots h_{v_1,|v_1|}h_{v_2,1}\cdots h_{v_r,|v_r|}.$$
\end{itemize}
Next assume that $\Gamma$ is oriented, ie we have an ordering $(e_1\cdots e_n)$ of the internal edges.
\begin{itemize}
    \item 
Chose an ordering of the half edges belonging to a given internal edge.
Together with the orientation this gives an ordering of the half edges belonging to the internal edges 
$$h_{1_x}h_{1_y}\cdots h_{n_x} h_{n_y}.$$ 
\end{itemize}
Given the choices made and given the orientation we can define a permutation of all the half-edges as follows
$$
\sigma: H(\Gamma)\rightarrow H(\Gamma) $$
$$
h_{v_1,1}\cdots h_{v_1,|v_1|}h_{v_2,1}\cdots h_{v_r,|v_r|}\mapsto h_{1_x}h_{1_y}\cdots h_{n_x} h_{n_y}l_{1,1}\cdots l_{1,m_1}l_{2,1}\cdots l_{b,m_b} $$
\begin{itemize}
    \item 

We begin the ordering of the leaves on the right hand side by the first leave that appears on the left hand side, then continue with the next one and so on.
\end{itemize}
\end{cons}
 We still fix the same oriented stable graph $\Gamma$ and the choices of orderings made.
 Denote $U:=V^\vee[-1]$. We next explain how to define, given these choices, a map
$$U^{\otimes h(\Gamma)}\rightarrow C^*(V)$$
 of degree $e(\Gamma)(d-2)$.
\begin{cons}\label{con_proj_g}
The permutation $\sigma$ induces a map which we denote by abuse of notation by the same letter
$${\sigma}:U^{\otimes h(\Gamma)}\rightarrow U^{\otimes h(\Gamma)}\cong U^{\otimes 2e(\Gamma)}\otimes U^{\otimes l(\Gamma)}.$$ 
\begin{itemize}
    \item 
Denote by 
$$proj: U^{\otimes l(\Gamma)}\rightarrow C^*(V)$$
the projection to symmetric words, which has degree 0.
\end{itemize}
Thus we can define a map of degree $e(\Gamma)(d-2)$
$$\bigotimes_{e\in E(\Gamma)}\langle\_,\_\rangle^{-1}\otimes proj:\ U^{\otimes 2e(\Gamma)}\otimes U^{\otimes l(\Gamma)}\rightarrow C^*(V).$$
\end{cons}
\begin{df}
Fixing the same oriented stable graph $\Gamma$ as in the previous constructions and the choices made we define 
\begin{equation}\label{main_constr_c}
\theta_{I^q}(\Gamma)=(\bigotimes_{e\in E(\Gamma)}\langle\_,\_\rangle^{-1}\otimes proj)\circ{\sigma}\big(I^q|_{v_1}\otimes\ldots\otimes I^q|_{v_n}\big)\hbar^{bt(\Gamma)}\in C^*(V)\llbracket\hbar\rrbracket.
\end{equation}
Here we use the the standard isomorphism between invariants and coinvariants in characteristic zero to view $I^q|_{v_i}$
as an invariant tensors, which we can thus see as an element of $U^{\otimes r},$ for $r=|v_i|$.
Further the order of the tensors is determined by the choice of ordering of vertices. 
\end{df}
\begin{lemma}
   Given $I^q\in MCE (\mathcal{O}^{pq}(V))$ the assignment  \ref{main_constr_c} defines a degree zero map of vector spaces
\begin{equation}\label{Kon_gen_2}
\begin{split}
\theta_{I^q}:{s\mathcal{G}_{c,\bullet}}^\vee[d-3]&\rightarrow C^*(V)\llbracket\hbar\rrbracket,\\
\Gamma&\mapsto \rho_{I^q}(\Gamma).
\end{split}
\end{equation}
In particular, it is independent of the choices made.
\end{lemma}
\begin{proof}
The fact that the degrees work out follows by counting together 
$$3-d=-e+(d-2)e+(3-d)(bt+v-\sum_{v_i}o_i)$$
and definition \ref{betti} of the betti number. 
This map is independent of the choices made as follows:
\begin{itemize}
    \item Let $\Gamma$ be a stable graph as in the previous construction.
    Assume that we chose a different ordering of vertices and of the half-edges at each vertex, but for now take the same orientation datum and ordering of half-edges belonging to each internal edge.
    Let us call the initial permutation as in construction \ref{perm_g} $\sigma_1:H(\Gamma)\rightarrow H(\Gamma)$ and the new one $\sigma_2:H(\Gamma)\rightarrow H(\Gamma)$.
    Then we have that 
    $${\sigma_1}\big(I^q|_{v_1}\otimes\ldots\otimes I^q|_{v_n}\big)=\tau\circ{\sigma_2}\big(I^q|_{v_{\Tilde{1}}}\otimes\ldots\otimes I^q|_{v_{\Tilde{n}}}\big),$$
    where on the right-hand-side we use the new ordering of the vertices to order the input tensors.
    Here $\tau$ is a general permutation.
    There is no sign showing up because all the $|I^q(v_i)|$ are even.
    
    Since in the definition of $\theta_{I^q}$, described in the second bullet point of construction \ref{con_proj_g}, we project to the quotient determined by all permutations $\tau$, that is symmetric words,
    we have that $\theta_{I^q}$ is independent of a different ordering of vertices and of an ordering of the half-edge at each vertex.
    \item Now assume that we chose a different ordering of the half-edges belonging to an internal edge respectively.
    Since the pairings $\langle\_,\_\rangle^{-1}$ are symmetric it follows that this gives the same result $\theta_{I^q}$.
  \end{itemize}  
Lastly this map is compatible with the orientation: Since the pairings have odd degree if
we change the orientation by an odd permutation the result changes by a factor of $(-1)$.
    
\end{proof}
\begin{lemma}\label{bra_comp_2}
     Given $I^q\in MCE (\mathcal{O}^{pq}(V))$ we have that 
     $\theta_{I^q}(\{\_,\_\})=\{\theta_{I^q}(\_)\theta_{I^q}(\_)\}.$

\end{lemma}
\begin{proof}
    Recall that the bracket on 
    $C^*(V)\llbracket\hbar\rrbracket$
    is defined by pairing each letter of a symmetric word with a letter of a symmetric word exactly once,
    by that joining the two symmetric words and taking into account the Koszul rule. 

    On the other hand we can write 
    $$\{\Gamma_1,\Gamma_2\}=\sum_{l_1\in L(\Gamma_1),l_2\in L(\Gamma_1)} \frac{\Gamma_1\cup_{l_1,l_2} \Gamma_2}{c_{l_1}(\Gamma_1)c_{l_2}(\Gamma_2)},$$
    the sum over all leafs of the graphs and gluing the two graphs together along those.
    Here $c_{l_1}(\Gamma)=n\in \mathbb{N}$ if there are automorphisms of $\Gamma$ that send $l_1$ to $n$ other leafs. 

    By definition \ref{main_constr_c} we see that $\sum_{l_1\in L(\Gamma_1),l_2\in L(\Gamma_1)} \theta(\Gamma_1\cup_{l_1,l_2} \Gamma_2)$
    is given by pairing the tensor 
    $(\bigotimes_{e\in E(\Gamma_1)}\langle\_,\_\rangle^{-1})\circ{\sigma}\big(I^q|_{v_1}\otimes\ldots\otimes I^q|_{v_n}\big)$
    and the tensor $(\bigotimes_{e\in E(\Gamma_2)}\langle\_,\_\rangle^{-1})\circ{\sigma}\big(I^q|_{v_1}\otimes\ldots\otimes I^q|_{v_n}\big)$
    together in all possible ways, using a letter $l_1$ and $l_2$ from each word and then projecting to symmetric words.
    However, compared to $\{\theta(\Gamma_1),\theta(\Gamma_2\}$ we are overcounting exactly $c_{l_1}(\Gamma_1)c_{l_2}(\Gamma_2)$ terms.
    This proves the result. 
\end{proof}
\begin{lemma}\label{diff_comp_2}
Given $I^q\in MCE (\mathcal{O}^{pq}(V))$ we have $\theta_{I^q}\circ \Delta_i=\hbar\Delta\circ\theta_{I^q}$.
\end{lemma}
\begin{proof}
This follows similarly to above lemma.
In the second equality the additional $\hbar$ comes from the fact that $\Delta_i$ increases the loop defect by one.
\end{proof}
\begin{lemma}\label{diff_comp2}
   Given $I^q\in MCE (\mathcal{O}^{pq}(V))$ we have that 
   $\theta_{I^q}\circ (-d)=d\circ \theta_{I^q}$.
\end{lemma}
\begin{proof}
    We have 
    $$\theta_{I^q}(-d\Gamma)=-\theta_{I^q}(\sum_{(*)}\Gamma')=-\sum_{(*)}(\bigotimes_{e\in E(\Gamma')}\langle\_,\_\rangle^{-1}\otimes proj)\circ\sigma'\circ(I^q|_{v_1}\otimes\ldots \otimes X\otimes\ldots I^q|_{v_n})\hbar^{bt(\Gamma)}.$$
    Here $(*)$ indicates that the sum is over isomorphism classes of stable graphs $\Gamma'$
    such that when contracting one of their legs according to rule \ref{cont_g} we recover $\Gamma$.
    Say $\Gamma$ has $n$ vertices.
    Note that a graph $\Gamma'$ as before has at most n+1 vertices and $n-1$ of those actually `coincide' with $n-1$ ones of $\Gamma$;
    let us call the vertex of $\Gamma$ which is not matched $v_i$.
    That is how the second equality above follows, where we denote by $X$ the tensor that gets associated to the 
    at most two vertices which are different (and which we assume wlog to be numbered consecutively).
    Note further that $d$ preserves the betti number.
    
    According to the rule how to contract \ref{cont_g} the sum over $(*)$ splits up into two summands: 
    \begin{enumerate}
        \item 
    Similarly to lemma \ref{bra_comp_2} we have that for the first summand, corresponding to rule \ref{cont_g}.1
    $$-\sum_{(*.1)}(\bigotimes_{e\in E(\Gamma')}\langle\_,\_\rangle^{-1}\otimes proj)\circ\sigma'(\cdots\otimes X\otimes\cdots )=-(
    \bigotimes_{e\in E(\Gamma)}\langle\_,\_\rangle^{-1}\otimes proj)\circ\sigma(\cdots\otimes  \frac{1}{2}\{I^q,I^q\}|_{v_i}\otimes\cdots).$$
 \item Again similarly to lemma \ref{diff_comp_2} we have that for the second summand, corresponding to rule \ref{cont_g}.2
 $$-\sum_{(*.2)}(\bigotimes_{e\in E(\Gamma')}\langle\_,\_\rangle^{-1}\otimes proj)\circ\sigma'(\cdots\otimes X\otimes\cdots )=-(
    \bigotimes_{e\in E(\Gamma)}\langle\_,\_\rangle^{-1}\otimes proj)\circ\sigma(\cdots\otimes  (\hbar\Delta I^q)|_{v_i}\otimes\cdots).$$
Here we need to include a $\hbar$ since for this summand we have decreased the betti number by one.
 \end{enumerate}
 Since we have that $I^q$ is a Maurer-Cartan element, ie in particular
 $$dI^q|_{v_i}=-(\frac{1}{2}\{I^q,I^q\}|_{v_i}+\hbar\Delta I^q|_{v_i})$$
 it follows that above sum is equal to
 $$=\sum_{i=1}^n(\bigotimes_{e\in E(\Gamma)}\langle\_,\_\rangle^{-1}\otimes proj)\circ\sigma\circ(I^q|_{v_1}\otimes\ldots dI^q|_{v_i}\cdots I^q|_{v_n})\hbar^{bt(\Gamma)},$$
 which we can rewrite, since both $proj$ and $\langle\_,\_\rangle^{-1}$ are chain maps, as 
 
 $$=d\circ\big(\bigotimes_{e\in E(\Gamma)}\langle\_,\_\rangle^{-1}\otimes proj\big)\circ\sigma\circ(I^q|_{v_1}\otimes\cdots \otimes I^q|_{v_n})\hbar^{g(\Gamma)}=d\theta_{I^q}(\Gamma).$$ \end{proof}
\begin{cons}\label{ext_sym_g}
    We extend the map \ref{Kon_gen_2} as a map of algebras to 
\begin{equation}\label{sdnit}
Sym({s\mathcal{G}_{c,\bullet}}^\vee[d-3])\rightarrow C^*(V)\llbracket\hbar\rrbracket,\end{equation}
which by abuse of notation we denote by the same letter.
\end{cons}
\begin{theorem}\label{Kc3}
    The map \ref{sdnit} is a map of $(d-2)$-twisted BD algebras
    $$\theta_{I^q}: s\mathcal{G}^d\llbracket\hbar\rrbracket\rightarrow \mathcal{O}^{pq}(V),$$
    which is functorial with respect to strict cyclic $L_\infty$-morphisms.
\end{theorem}
\begin{proof}  \begin{itemize}
    \item 
  It is a map of algebras by construction.
\item By the first part of lemma \ref{diff_comp_2} and construction $\theta_{I^q}$ intertwines $\Delta_i$ with $\hbar\Delta$, further by lemma \ref{diff_comp2} it intertwines the internal differentials.
\item We proved that map \ref{Kon_gen_2} is a map of shifted Lie algebras,
in fact it is a map of a shifted Lie algebra to a shifted Poisson algebra.
Any such map determines a map of shifted Poisson algebras from the free algebra on the domain to the target and this is indeed
how we defined the shifted Poisson structure on 
$s\mathcal{G}_B^d\llbracket\hbar\rrbracket$.
\item Furthermore this implies that $\theta_{I^q}$ intertwines $\Delta_e$ with the associated Chevalley-Eilenberg differential of the Lie algebra $\mathcal{O}^{pq}(V)$.
This, together with the arguments in bullet point two imply
$$\theta_{I^q}\big((-d+\hbar\Delta_e+\Delta_i)(\Gamma_1\cdot\ldots\cdot \Gamma_n)\big)=(d+\hbar\Delta)\rho_{I^q}(\Gamma_1\cdot\ldots\cdot \Gamma_n),$$
which finishes the first part of the proof.
\item Given $F:V\rightarrow W$ a strict cyclic $L_\infty$-morphism such that $F^*I^q=J^q$ one can check that 
\begin{equation}\label{hdlm}\begin{tikzcd}
 \mathcal{F}^{pq}(W)\arrow{rr}{F^*}&&   \mathcal{F}^{pq}(V)\\
&s\mathcal{G}^d\llbracket\hbar\rrbracket\arrow[swap]{ur}{\theta_{I^q}}\arrow{ul}{\theta_{J^q}}&
\end{tikzcd}\end{equation}
commutes, using the definition \ref{main_constr_c}.
\end{itemize}\end{proof}

Given an $I\in MCE(\mathcal{O}^{fr}(V))$, we can construct a map
$$\theta_{I}:{\mathcal{G}_{c,\bullet}}^\vee[d-3]\rightarrow C^*(V)\llbracket\hbar\rrbracket$$
given by exactly the same construction. Furthermore, in the same way as in lemma \ref{bra_comp_2} and \ref{diff_comp2} we can prove that this maps intertwines the algebraic structures.
In the same way as in construction \ref{ext_sym_g} we get a map, also denoted
\begin{equation}\label{hhd}
    \theta_{I}:Sym({\mathcal{G}_{c,\bullet}}^\vee[d-3])\llbracket\hbar\rrbracket\rightarrow C^*(V)\llbracket\hbar\rrbracket
    \end{equation}
for which we can argue exactly the same as in theorem \ref{Kc3} that we have: 
\begin{theorem}\label{Kc3.1}
Given $I\in MCE(\mathcal{O}^{fr}(V))$ the map \ref{hhd} is a map of $(d-2)$-twisted BD algebras
    $$\theta_{I}: \mathcal{G}^d\llbracket\hbar\rrbracket\rightarrow \mathcal{O}^{pq}(V),$$
    functorial with respect to strict cyclic $L_\infty$-morphisms.
\end{theorem}
Again, using the same construction, given $I\in MCE(\mathcal{O}^{fr}(V))$, we can construct a map
$$\theta_{I}:{\mathcal{T}_{c,\bullet}}^\vee[d-3]\rightarrow C^*_+(V).$$
Indeed we land in $C^*_+(V)$ since trees needs to have a non-zero number of leafs.
In the same way as in lemmas \ref{bra_comp_2} and \ref{diff_comp2} we can argue that this map is compatible with the algebraic structures.
As in construction \ref{ext_sym_g} we get a map, also denoted
\begin{equation}\label{ibmd}
\theta_{I}:Sym({\mathcal{T}_{c,\bullet}}^\vee[d-3])\rightarrow C^*_+(V)\end{equation}
for which we can argue analogously to same as in theorem \ref{Kc3} that we have:
\begin{theorem}\label{Kc4}
Given $I\in MCE(\mathcal{O}^{fr}(V))$ the map \ref{ibmd} is a map of $(d-2)$-shifted Poisson algebras
    $$\theta_{I}: \mathcal{T}^d\rightarrow \mathcal{O}^{fr}(V),$$
    functorial with respect to strict cyclic $L_\infty$-morphisms.
\end{theorem}
Theorems \ref{Kc3.1}, \ref{Kc4} and \ref{Kc3} fit together with 
diagram \ref{ag_dq} and the dequantization map \ref{fc_dq} together:
\begin{theorem}\label{comp_gr_al}
Given $I^q\in MCE(\mathcal{O}^{pq}(V))$ denote by $I=p(I^q)\in MCE(\mathcal{O}^{fr}(V))$ its dequantization, see \ref{fc_dq}. Then following diagram commutes:
$$\begin{tikzcd}
\mathcal{O}^{fr}(V)&\mathcal{O}^{pq}(V)\arrow{l}{p}&
\mathcal{O}^{pq}(V)\arrow[equal]{l}\\
\mathcal{T}^d\arrow{u}{\theta_{I}}&\mathcal{G}^d\llbracket\hbar\rrbracket\arrow[swap]{l}\arrow{r}\arrow{u}{\theta_{I}}&
s\mathcal{G}^d\llbracket\hbar\rrbracket\arrow[bend left=20]{ll}\arrow{u}{\theta_{I^q}}.
    \end{tikzcd}$$
\end{theorem}
\begin{proof}
    This follows by unraveling the constructions. 
\end{proof}
We recall the singled-out Maurer-Cartan elements from \ref{S_mce} and \ref{sG_mce} and state following lemmas, which are straightforward to prove:
\begin{lemma}\label{comp_3}
Given $I^q\in MCE(\mathcal{O}^{pq}(V))$ we have that $\theta_{I^q}(G)=I^q$.
\end{lemma}
\begin{lemma}\label{comp_4}
    Given $I\in MCE(\mathcal{O}^{fr}(V))$ we have $\theta_I(T)=I.$
\end{lemma}
Thus we can induce the following theorem:
 \begin{theorem}\label{fde}
Given a quantization $I^q$ of a cyclic $L_\infty$-algebra $L$ there is a commutative diagram as follows:        
    $$\begin{tikzcd}
      \mathcal{O}^{cl}(L)&  \mathcal{O}^{q}(L)\arrow{l}\\
      \mathcal{T}^{d,tw}\arrow{u}{\theta_{I}}&s\mathcal{G}^{d,tw}\llbracket\hbar\rrbracket\arrow{u}{\theta_{I^q}}\arrow{l}.
    \end{tikzcd}$$
    Here the horizontal maps are dequantization maps, the right vertical map is a map of $(d-2)$-twisted BD algebras and the left vertical maps is a map of $(d-2)$-shifted Poisson algebras
    \end{theorem} 
    \begin{proof}
  Indeed this follows from the previous theorem and the consecutive two lemmas by twisting the outer diagram by the compatible Maurer-Cartan elements.      
    \end{proof}

\section{Main Compatibility Results}
In this section we explain how cyclic $A_\infty$-categories and their quantizations (section \ref{ncw}), cyclic $L_\infty$-algebras and their quantizations (section \ref{cw}),
(stable) ribbon graphs (section \ref{rbgs}) and (stable) graphs (section \ref{grs}) fit together, using the (quantized) LQT map (section \ref{slqtq}), Kontsevich's cocycle construction (section \ref{cksc}), its commututive analogue (section \ref{ksc}) and the map from section \ref{vsgzg}.
\begin{theorem}\label{Main}
Assume we are given $I^q\in MCE(\mathcal{F}^{pq}(V_{Ob\mathcal{C}}))$,
a quantization of an essentially finite dimension $d$ cyclic $A_\infty$-category $\mathcal{C}$ with its associated hamiltonian 
$I\in MCE(\mathcal{F}^{fr}(V_{Ob\mathcal{C}}))$.
For $N\in \mathbb{N}$ denote by $\mathfrak{gl}_NA_\mathcal{C}$ the associated cyclic commutator $L_\infty$-algebra with values in $N\times N$-matrices
and by $I^q_N$ its quantization determined by theorem \ref{LQT_q}.
Then following diagram commutes.
Further it is functorial with respect to strict cyclic morphisms of quantizations of essentially finite cyclic $A_\infty$-categories
and independent up to isomorphism from the choice of a finite presentation.   
$$ \begin{tikzcd}[row sep=scriptsize, column sep=scriptsize]
& \mathcal{F}^{cl}(\mathcal{C})  \arrow[rr,"LQT^{cl}"] & & \mathcal{O}^{cl}(\mathfrak{gl}_NA_\mathcal{C})   \\ \mathcal{F}^{q}(\mathcal{C})\arrow[rr, "LQT^q"{xshift=-7pt}]  \arrow[ur] & & \mathcal{O}^{q}(\mathfrak{gl}_NA_\mathcal{C})\arrow[ur] \\
& \mathcal{RT}^{d,tw}_{Ob(Sk\mathcal{C})}\arrow[uu,dashed, "\rho_{I}^{tw}" {yshift=-10pt}]   & & \mathcal{T}^{d,tw}\arrow[ll,dashed]\arrow[ uu,"\theta_{I_N}^{tw}" {yshift=-8pt}]  \\
 s\mathcal{RG}^{d,tw}_{Ob(Sk\mathcal{C})}\llbracket\gamma\rrbracket \arrow[uu,"\rho_{I^q}^{tw}" {yshift=5pt}]\arrow[ur] & & s\mathcal{G}^{d,tw}\llbracket\hbar\rrbracket\arrow[ll]\arrow[ur]\arrow[uu,"\theta_{I^q_N}^{tw}" {yshift=-10pt}]\\
\end{tikzcd}
$$
The objects on the front face are $(d-2)$-twisted BD algebras, the vertical arrows are maps of $(d-2)$-twisted BD algebras, the upper horizontal arrow is a 2-weighted map of  BD algebras and the lower one is the predual to such a map. The back face commutes as a diagram $(d-2)$-shifted Poisson algebras. The maps in between the two faces are dequantization maps.   
\end{theorem}
\begin{proof}
Commutativity will follow from the next theorem \ref{main_compa}. More precisely, consider the outer commuting cube featuring in the next theorem. Rotating that cube around the z-axis by 90 degrees and then twisting that diagram by the compatible Maurer-Cartan in each corner (see lemma \ref{comp_1}, lemma \ref{comp_2} and lemma \ref{comp_3}, lemma \ref{comp_4}) gives the commutative cube of this theorem.

We have explained in theorem \ref{LQT_q} and \ref{LQT_cl} how the top face is functorial as claimed in the theorem. This also explains functoriality of the right face by diagram \ref{hdlm}. In the same cited theorems the independence of the top face on the choice of a finite presentation is explained. 

In theorem \ref{Kc2} and theorem \ref{Kc1} functoriality for the left face is explained. Further we have elaborated in \ref{unabhai} how the left face is independent of the choice of finite presentation. The corresponding statements for the other faces can be deduced from these results.
\end{proof}

\begin{theorem}\label{main_compa}
Assume we are given $I^q\in MCE(\mathcal{F}^{pq}(V_{Ob\mathcal{C}}))$, a quantization of an essentially finite dimension $d$ cyclic $A_\infty$-category $\mathcal{C}$ with its associated hamiltonian $I\in MCE(\mathcal{F}^{fr}(V_{Ob\mathcal{C}}))$. Fix the same notations as in the previous theorem. Then following diagram, where $B=ObSk\mathcal{C}$ finite, commutes.
$$\begin{tikzcd}
&\mathcal{O}^{fr}(\mathfrak{gl}_NV)&&\mathcal{O}^{pq}(\mathfrak{gl}_NV)\arrow[swap]{ll}{p}&&
\mathcal{O}^{pq}(\mathfrak{gl}_NV)\arrow[equal]{ll}\\
\mathcal{F}^{fr}(V_B)\arrow{ur}{LQT^{fr}_N}&&\mathcal{F}^{pq}(V_B)\arrow[swap,ll, "p" {xshift=10pt}]\arrow{ur}{LQT^{pq}_N}&&
\mathcal{F}^{pq}(V_B)\arrow[equal]{ll}\arrow{ur}{LQT^{pq}_N}&\\ 
&\mathcal{T}^d\arrow[dashed, uu,"\theta_{I_N}" {yshift=-8pt}]\arrow{dl}{\pi}&&\mathcal{G}^d\llbracket\hbar\rrbracket\arrow[swap,dashed]{ll}\arrow[dashed]{rr}\arrow[dashed, uu,"\theta_{I_N}" {yshift=-10pt}]\arrow{dl}&&
s\mathcal{G}^d\llbracket\hbar\rrbracket\arrow[bend left=12,dashed]{llll}\arrow[dashed, uu,"\theta_{I^q_N}" {yshift=-8pt}]\arrow{dl}{\pi}\\
\mathcal{RT}^d_{B}\arrow[uu,"\rho_{I}" {yshift=5pt}]&&\mathcal{RG}^d_{B}\llbracket\gamma\rrbracket\arrow[swap]{ll}\arrow{rr}\arrow[uu,"\rho_{I}" {yshift=5pt}]&&
s\mathcal{RG}^d_{B}\llbracket\gamma\rrbracket\arrow[bend left=12]{llll}\arrow[uu,"\rho_{I^q}" {yshift=5pt}]&
\end{tikzcd}$$
The corners of the cube on the right are $(d-2)$-twisted BD algebras, the upper arrows pointing inward are 2-weighted maps of BD algebras, whereas the lower arrows pointing outward are predual to such.
The horizontal arrows are maps of $(d-2)$-twisted BD algebras. The horizontal arrows in the left cube are dequantization maps, as are the curved arrows.
Lastly the left face commutes as a diagram $(d-2)$-shifted Poisson algebras.  
\end{theorem}
\begin{proof}
We have already seen that the upper two faces are commuting by remark \ref{pqdfr}, the lower two faces, including the curved arrows, are commuting by remark \ref{comp_g_rg}.
Further the front two faces, including the curved arrows, are commuting by theorem \ref{comp_rg_al} as are the two faces in the back, including the curved arrows, by theorem \ref{comp_gr_al}.
Thus it remains to show that the three parallel squares commute. We show that the rightmost one commutes, for the other two it follows similarly.
By definition of the prequantum LQT map this outermost diagram is equal to the following one.
\begin{equation}\label{comp_LQT_}
\begin{tikzcd}
\mathcal{F}^{pq}(V_B)\arrow{r}{\iota_B^*}&\mathcal{F}^{pq}(\widetilde{V_B})\arrow{r}{LQT_N^{pq}}&\mathcal{O}^{pq}(\mathfrak{gl}_NV)\\ \\
s\mathcal{RG}^d_{B}\llbracket\gamma\rrbracket\arrow{uu}{\rho_{I^q}}&s\mathcal{RG}^d_{\{*\}}\llbracket\gamma\rrbracket\arrow{uu}{\rho_{\widetilde{I}^q}}\arrow{l}&s\mathcal{G}^d\llbracket\hbar\rrbracket\arrow{l}{\pi}\arrow{uu}{\theta_{I^q_N}}
\end{tikzcd}
\end{equation}
Here we denoted by $\widetilde{V_B}$ the cyclic chain complex associated to $V_B$ under the cyclic commutator functor \ref{cyccomm} and  $\widetilde{I}^q:=\iota_B^*I^q\in MCE(\mathcal{F}^{pq}(\widetilde{V_B}).$
Further the lower left horizontal map is induced by the map of sets $B\rightarrow \{*\}$.
We prove commutativity of diagram \ref{comp_LQT_} in two steps:
    \begin{lemma}
The left square of diagram \ref{comp_LQT_} commutes. 
\end{lemma}
\begin{proof}
For $\Gamma$ a stable ribbon graph colored just by one element $\{*\}$ we can write
$$\rho_{\widetilde{I}^q}(\Gamma)=\sum_C \rho_{I^q}(\Gamma_C),$$
where the sum is over all isomorphism classes of $B$-colored stable ribbon graphs $\Gamma_C$ whose underlying $\{*\}$-colored stable ribbon graph is $\Gamma$.
Here we denoted by $C$ such a coloring.
Indeed, this follows by recalling the definition \ref{cyccomm} of the cyclic pairing on the cyclic chain complex associated to a collection of cyclic chain complexes and the construction \ref{main_constr_nc} of the map $\rho$.
But this equality expresses exactly commutativity of the diagram since $\iota_B^*$ is just an inclusion.
\end{proof}
\begin{lemma}
The right square of diagram \ref{comp_LQT_} commutes.
\end{lemma}
\begin{proof}
By recalling the definition of the prequantum LQT map for algebras the right square of the diagram is given by
$$
\begin{tikzcd}
\mathcal{F}^{pq}(\widetilde{V_B})\arrow{r}{tr^{pq}}&\mathcal{F}^{pq}(M_N\widetilde{V_B})\arrow{r}{m^{pq}\circ f^{pq}}&\mathcal{O}^{pq}(\mathfrak{gl}_NV)\\ \\
s\mathcal{RG}^d_{\{*\}}\llbracket\gamma\rrbracket\arrow{uu}{\rho_{\widetilde{I}^q}}\arrow{uur}{\rho_{M_N\widetilde{I}^q}}&&s\mathcal{G}^d\llbracket\hbar\rrbracket\arrow{ll}{\pi}\arrow{uu}{\theta_{I^q_N}}.
\end{tikzcd}
$$
Here $M_N\widetilde{I}^q=tr^{pq}\widetilde{I}^q$ and $I^q_N=m^{pq}\circ f^{pq}(M_N\widetilde{I}^q).$
We prove commutativity of the triangle first.

We recall from around lemma \ref{tr} the definition of $tr^{pq}$.
As explained in \cite{GGHZ21} theorem 4.9 and theorem 4.15, where $tr^{pq}$ is denoted $\mathcal{M}_{\gamma,\nu}$, one way to understand this map is by tensoring with certain tensors
$$\mu^{g,b}_{n_1\cdots n_m}\in M_N^{-\otimes n_1}\otimes\cdots \otimes M_N^{-\otimes n_m},$$
which exist for all integers $g,b\geq 0$ and $m,n_i>0$.
The properties of these tensors is described in lemma 3.10 of \cite{Ham11}, where they are denoted $\alpha^{g,b}$.
To be precise we are interested in these tensors applied to the Frobenius algebra $M_N$ of $N\times N$ matrices, in the set-up of \cite{Ham11}.

In the following we use the ambiguous notation $\mu^{g,b}_{m}$ for such a tensor above. 
We verify commutativity of the triangle for connected stable ribbon graphs $\Gamma$ from which general commutativity follows. We have by definition that
$$\rho_{M_N\widetilde{I}^q}(\Gamma)=proj \circ \bigotimes_{e\in E(\Gamma)}(\langle\_,\_\rangle^{-1}_e\langle\_,\_\rangle^{-1}_{M_N})\sigma \Big((\mu^{g_1, b_1}_{h_1}\otimes \widetilde{I}^q)|_{v_1}\otimes\cdots\otimes (\mu^{g_n, b_n}_{h_n}\otimes \widetilde{I}^q)|_{v_n}\Big)\nu^{{tot(\Gamma)}}\gamma^{g(\Gamma)}$$
Using implicitly the isomorphism $M_N^{\otimes 2}\cong {M_N^{\otimes -2}}^\vee$ and picking a good basis for $M_N$ we may write $$\langle\_,\_\rangle^{-1}_{M_N}=x_i\otimes y^i$$
where $x_i,y^i\in M_N$ and an implicit sum over $i$ is omitted. Then consecutively using properties (1)-(5) of the tensors $\mu^{g,b}_{m}$ as described in lemma 3.10 of \cite{Ham11} we have that above is equal to

$$=proj \circ \bigotimes_{e\in E(\Gamma)}(\langle\_,\_\rangle^{-1}_e\otimes\mu^{g(\Gamma), f_{tot}(\Gamma)}_{l(\Gamma)})\sigma(\widetilde{I}^q|_{v_1}\otimes\cdots\otimes \widetilde{I}^q|_{v_n})\nu^{f_{tot}}\gamma^{g(\Gamma)} =tr^{pq}\rho_{\widetilde{I}^q}$$
Indeed, since the graph is connected we end up with just one tensor $\mu^{g(\Gamma), f_{tot}(\Gamma)}_{l(\Gamma)}$.
While consecutively applying (3)-(5) of these properties (and possibly using (1)-(2) to bring elements in the good positions - no signs arise since $M_N$ is concentrated in degree zero) we pick up exactly the genus and the total number of faces of $\Gamma$, recalling definitions \ref{armg} and \ref{faces}.
For that latter part we have used the (not written `extreme' terms) of lemma 3.10 (3)-(5) in loc. cit.
$$(3)\ \ \mu^{g,b}_{n+1}((x_iy^i)\otimes (a_{11}\cdots a_{1k_1})\otimes\cdots\otimes (a_{n1}\cdots a_{nk_n} )=\mu^{g,b+2}_{n}( (a_{11}\cdots a_{1k_1})\otimes\cdots \otimes(a_{n1}\cdots a_{nk_n} ))$$
$$(4)\ \ \mu^{g,b}_{n+2}((x_i)\otimes(y^i)\otimes (a_{11}\cdots a_{1k_1})\otimes\cdots\otimes (a_{n1}\cdots a_{nk_n} )=\mu^{g+1,b+1}_{n}( (a_{11}\cdots a_{1k_1})\otimes\cdots \otimes(a_{n1}\cdots a_{nk_n} ))$$\begin{equation*}
\begin{split}
(5)\ \
& \mu^{g_1,b_1}_{n_1+1}\big( (a_{11}\cdots a_{1k_1})\otimes\cdots\otimes (a_{n_11}\cdots a_{n_1k_{n_1}})\otimes(x_i) \big)\mu^{g_2, b_2}_{n_2+1}\big((y^i)\otimes (b_{11}\cdots b_{1k_1})\otimes\cdots\otimes (b_{n_21}\cdots b_{n_2k_{n_2}}) \big)\\
&=\mu^{g_1+g_2,b_1+b_2+1}_{n_1+n_2}\big( (a_{11}\cdots a_{1k_1})\otimes\cdots\otimes (a_{n_11}\cdots a_{n_1k_{n_1}} )\otimes(b_{11}\cdots b_{1k_1})\otimes\cdots\otimes (b_{n_21}\cdots b_{n_2k_{n_2}})\big)
\end{split}
\end{equation*}
Now we show commutativity of the odd square. This is expressed by the following equality of first and third line:
\begin{equation*}
\begin{split}
&m^{pq}f^{pq}\rho_{M_N\widetilde{I}^q}(\pi(\Gamma))=m^{pq}f^{pq}\Big(\sum_{\bar{\Gamma}}(\bigotimes_{e\in E(\bar{\Gamma})}\langle\_,\_\rangle^{-1}\otimes proj)\circ{\sigma}\big(M_N\widetilde{I}^q|_{v_1}\otimes\ldots\otimes M_N\widetilde{I}^q|_{v_n}\big)\nu^{tot}\gamma^{g(\Gamma)}\Big)\\
&=(\bigotimes_{e\in E(\Gamma)}\langle\_,\_\rangle^{-1}\otimes proj)\circ{\sigma}\big({I}^q_N|_{v_1}\otimes\ldots\otimes {I}^q_N|_{v_n}\big)\hbar^{2g(\bar{\Gamma})+f_{tot}(\bar{\Gamma})-1}\\
&=(\bigotimes_{e\in E(\Gamma)}\langle\_,\_\rangle^{-1}\otimes proj)\circ{\sigma}\big({I}^q_N|_{v_1}\otimes\ldots\otimes {I}^q_N|_{v_n}\big)\hbar^{bt(\Gamma)}=\theta_{I^q_N}(\Gamma)
\end{split}
\end{equation*}
Here $\Gamma$ denotes a connected stable graph and the sum in the first line is over all stable ribbon graphs $\bar{\Gamma}$ whose underlying connected stable graph is $\Gamma$.
The first equality follows since  pairing the tensors $\{I^q_N|_{v_i}\}$ according to the graph $\Gamma$ is the same as pairing the tensors $\{M_NI^q|_{v_i}\}$ in all possible ways, but respecting the rule imposed by $\Gamma$.
However we can reorder all these possible ways according to each way of putting the structure of a stable ribbon graph on $\Gamma$ and then pairing the tensors $M_NI^q|_{v_i}$ respectively,
which is expressed by the sum in the first line. 

The second equality follows since given a connected stable ribbon graph $\bar{\Gamma}$ with underlying stable graph $\Gamma$ (recall definition \ref{sRGsG}) we have
$$bt((\Gamma))=2g(\bar{\Gamma})+f_{tot}(\bar{\Gamma})-1,$$
as can be verified by a simple computation further; recalling the definitions of the arithmetic genus \ref{armg}, of the total number of empty faces \ref{faces} as well as of the betti number \ref{betti}.
\end{proof}

These two lemmas finish the proof of theorem \ref{main_compa}.
\end{proof}
\section{Declaration on Funding, Competing interests and Data availability}
This work was done during my PhD which was funded by the European Union’s Horizon 2020 research and innovation programme under the Marie Skłodowska-Curie grant agreement No 945322 through the Fondation Sciences Mathématiques de Paris. The author has no competing interests to declare that are relevant to the content of this article. We do not analyse or generate any datasets, because our work proceeds within a theoretical and mathematical approach.

\bibliography{refs}
\bibliographystyle{alpha}
\end{document}

%% file: commands.tex
\usepackage[left=2.5cm, right=2.5cm, top=0.5cm, bottom=2.5cm]{geometry}
\usepackage[colorlinks,
pdfpagelabels,
pdfstartview = FitH,
bookmarksopen = true,
bookmarksnumbered = true,
linkcolor = black,
plainpages = false,
hypertexnames = false,
citecolor = black] {hyperref}
\usepackage{adjustbox}
\usepackage{amsmath,amssymb,mathrsfs,amsfonts}

\usepackage{bm}
\newcommand*{\B}[1]{\ifmmode\bm{#1}\else\textbf{#1}\fi}
\usepackage[english]{babel}
\usepackage{tikz-cd}
\usepackage{amstext}
\usepackage{parskip}
\usepackage{pdfpages}
\usepackage{mathtools}
\usepackage{stmaryrd}
\usepackage{xcolor}
\usepackage{xfrac} 
\usepackage[utf8]{inputenc}
\usepackage{url} 
\usepackage{pgfplots}
\usepackage[autooneside=false,headsepline,markcase=noupper]{scrlayer-scrpage}
\usepackage{blindtext}
\usepackage{enumitem}
\usepackage{verbatim}
\usepackage{graphicx}
\usepackage{mathtools}
\usepackage{makeidx}
\usepackage{tikz}
\usetikzlibrary{babel}

\addto\captionsenglish{}

\usetikzlibrary{%
	matrix,%
	calc,%
	arrows%
}

\definecolor{mycolor}{RGB}{196,19,47}
\definecolor{mygray}{gray}{0.4}

\setkomafont{section}{\normalfont\LARGE\bfseries}
\setkomafont{subsection}{\normalfont\large\bfseries}
\setkomafont{pagehead}{\normalfont\bfseries}
\setkomafont{title}{\normalfont\bfseries\Large}

\usepackage{chngcntr}
\counterwithin{equation}{subsection}
\usepackage{amsthm}

\theoremstyle{remark}
\newtheorem{remark}[equation]{Remark}
\newtheorem*{remark_nn}{Remark}

\theoremstyle{definition}
\newtheorem{df}[equation]{Definition}
\newtheorem{cons}[equation]{Construction}

\newtheorem{ex}[equation]{Example}

\theoremstyle{plain}

\newtheorem*{thm_nn}{Theorem}
\newtheorem{theorem}[equation]{Theorem}
\newtheorem*{lemma_nn}{Lemma}
\newtheorem{corollary}[equation]{Corollary}
\newtheorem*{cor_nn}{Corollary}
\newtheorem{lemma}[equation]{Lemma}

\newtheorem*{thm_nn_a0}{Theorem A.0}
\newtheorem*{thm_nn_a1}{Theorem A.1}
\newtheorem*{thm_nn_b}{Theorem B}

\newtheorem*{thm_nn_c}{Theorem C}
\newtheorem*{thm_nn_d}{Theorem D}
\newtheorem*{thm_nn_e}{Theorem E}
\newtheorem*{thm_nn_f}{Theorem F}

\newtheorem*{Claim}{Claim}
\newtheorem*{Conj_nn}{Conjecture}

\clearpairofpagestyles 
\ohead{  \ifstr{\rightbotmark}{\leftmark}{}{\rightbotmark}}
\ihead{\leftmark}
\rohead*{\pagemark}
\rehead*{\pagemark}
\lehead{\headmark}
\lohead{\headmark}
\clearscrheadfoot
\cofoot[\pagemark]{\pagemark}

\makeatletter

\providecommand*{\rightbotmark}{\expandafter\@rightmark\botmark\@empty\@empty}
\makeatother
\makeindex

\usepackage{mathabx}

%% file: intro.tex
In this paper we explain a connection between seemingly unrelated mathematical works, which are \begin{itemize}
\item  the \textbf{Loday-Quillen-Tsygan} (LQT) \textbf{theorem} \cite{LQ84, Tsy83}, 
\item \textbf{Kontsevich's cocycle construction} \cite{Kon92b}, \item  and various \textbf{graph complexes}.
 \end{itemize}
To be more precise we connect these three different works when taking as input datum a dimension $d$ Calabi-Yau\footnote{In fact here we work with something stricter than Calabi-Yau categories. We consider (smooth) cyclic $A_\infty$-categories.
See also remark \ref{cy_cyc}.} category and we are relating these objects equipped with \textbf{$(d-2)$-shifted Poisson algebra} structures.

We provide a quantized version of such a relationship in the category of \textbf{Batalin-Vilkovisky algebras}\footnote{To be more precise we use Beilinson-Drinfeld, which are a slight variation of Batalin-Vilkovisky algebras.}, building upon work \cite{GGHZ21} on a quantized LQT theorem. Here the input is a `quantization' of a Calabi-Yau category, described by a Maurer-Cartan element, also known as a solution to the open Quantum Master Equation associated to the Calabi-Yau category\footnote{Equivalently such an element may be described as a possibly curved (involutive bi-Lie$)_\infty$-algebra structure on a central extension of cyclic cochains of this category.
Compare work by Campos-Merkulov-Willwacher \cite{Ca16}, section 5 and also work by Cieliebak-Fukaya-Latschev \cite{CiFuLa20}, section 12 and work by Naef-Willwacher \cite{NaWi19}, section 7.
However, these authors don't consider the central extension, which is crucial for us.}.

In section \ref{MM from CY} we provide an independent motivation for the relevance of such Maurer-Cartan elements, coming from enumerative geometry. As an upshot, a quantization of a Calabi-Yau category provides an answer to the following question:  
\begin{equation}\label{Q}
\text{\textbf{Can we start with a Calabi-Yau category and produce a large $N$ gauge theory?}}
\end{equation}
Indeed, the perspective advocated in \cite{GGHZ21,CL15} is that a solution to the open Quantum Master Equation describes, within the Batalin-Vilkovisky formalism from theoretical physics and under the quantized LQT map, the partition function of a so called large $N$ matrix model, the finite dimensional cousins of large $N$ gauge theories,\footnote{Compare section \ref{gap}.
To tackle proper gauge theories an important task is to  understand whether the obtained partition function is sufficiently smooth. See also the exciting recent work \cite{ha25a, ha25b}.} thus offering an answer to \eqref{Q}.

Work in preparation \cite{AmTu25}, \cite{Ul25}, together with fundamental results \cite{HVZ08}  \cite{Zw98} will give a criterion when and how to obtain a quantization of a Calabi-Yau category.

We explain the content of this paper in more detail: Let $A$ be a a cyclic $A_\infty$-algebra (for instance a graded Frobenius algebra) whose pairing is of odd degree $d$.
The first result is that the original Loday-Quillen-Tsygan map
 $$LQT:\ Sym\big(Cyc^*_+(A)[-1]\big)\rightarrow C^{*}_+(\mathfrak{gl}_NA),$$
 fits into a commuting diagram (we recall that it connects $Cyc^*_+(A)$, the cyclic cochains of $A$ computing its \textbf{cyclic cohomology}, and Lie algebra cochains of the Lie algebra $\mathfrak{gl}_NA$ of $N\times N$ matrices with values in $A$). 
 The other corners are given by two \textbf{graph complexes}, denoted 
$\mathcal{RT}^{tw}$ respectively $\mathcal{T}^{tw}$, built from (ribbon) tree graphs, with a natural map $\pi$ between them.
Those objects are in fact \textbf{(d-2)-shifted Poisson dg algebras}, induced by gluing leafs.
The differential is given by the standard differential on graph complexes, twisted by a certain Maurer-Cartan element, respectively.
The vertical maps are generalization of Kontsevich cocycle construction, respectively its 'commutative' analogue \cite{pe96}.
Summarizing, we have:
\begin{thm_nn_a0}
Given a cyclic $A_\infty$-algebra of odd degree $d$ there is a commutative diagram of \textbf{(d-2)-shifted Poisson dg algebras}    
$$
\begin{tikzcd}
    Sym\big(Cyc^*_+(A)[-1]\big)\arrow{r}{LQT}&C^{*}_+(\mathfrak{gl}_NA)\\
    \mathcal{RT}^{tw}\arrow{u}{}&\mathcal{T}^{tw}\arrow{l}{}\arrow{u}{},
\end{tikzcd}$$
functorial with respect to strict cyclic $A_\infty$-morphisms.
\end{thm_nn_a0}
If $A$ is $\mathbb{Z}$-graded, then so are the corners of this diagram.
For readability we only show the underlying $\mathbb{Z}/2$-graded objects in this first overview and omit some even shifts.
\par
Another contribution of this paper is \textbf{generalizing the LQT theorem to $A_\infty$-categories}:
\begin{thm_nn_b}[see theorems \ref{LQT_c} and \ref{LQT_conj}]
Given a small $A_\infty$-category $\mathcal{C}$ there is a map of dg-algebras 
$$\begin{tikzcd}
LQT_{\mathcal{C}}:\ Sym\big(Cyc^*_+(\mathcal{C})[-1]\big)\arrow{r}& C^*(\mathfrak{gl}_NA_\mathcal{C}),
\end{tikzcd}$$
which is functorial in $\mathcal{C}$ with respect to $A_\infty$-functors.

If $\mathcal{C}$ is unital then for $N\rightarrow \infty$ this map becomes a quasi-isomorphism
   $$ \begin{tikzcd}
        LQT_{\mathcal{C}}:\ Sym\big(Cyc^*_+(\mathcal{C})[-1]\big)\arrow{r}{\sim}& \underset{N}{\varprojlim}\ C^*(\mathfrak{gl}_NA_\mathcal{C}).\end{tikzcd}$$
\end{thm_nn_b}
Here $\mathfrak{gl}_NA_\mathcal{C}$ is what we call the commutator $L_\infty$-algebra with values in $N\times N$-matrices associated to $\mathcal{C}$.\footnote{After the writing of this paper was completed the author became aware of \cite{zeng23}, which deals with closely related phenomena.}

The domain of the LQT map can be equipped with a (d-2)-shifted Poisson dg algebra in case the category is \textit{cyclic of degree d}.
However, to do so also with the codomain we need to restrict to essentially finite cyclic $A_\infty$-categories.\footnote{That is cyclic $A_\infty$-categories which admit a finite skeleton (see definition \ref{essfin}). For the sake of this introduction we identify an essentially finite category with such a skeleton.} This allows us to generalize theorem A.0:  
\begin{thm_nn_a1}[see theorem \ref{Main}]
Given an essentially finite cyclic $A_\infty$-category $\mathcal{C}$ of odd degree $d$  there is a commutative diagram of $(d-2)$-shifted Poisson dg algebras
$$\begin{tikzcd}
    Sym\big(Cyc^*_+(\mathcal{C})[-1]\big)\arrow{r}{LQT_{\mathcal{C}}}&C_+^*(\mathfrak{gl}_NA_\mathcal{C})\\
    \mathcal{RT}^{tw}_{\mathcal{C}}\arrow{u}{}&\mathcal{T}^{tw}\arrow{l}{}\arrow{u}{},
\end{tikzcd}$$
functorial with respect to strict cyclic $A_\infty$-morphisms.
Here the ribbon tree graph complex from before is now decorated, see definition \ref{Dec}, by the set of objects of $\mathcal{C}$.
\end{thm_nn_a1}
\begin{remark_nn}
It is a natural question how far this construction can be made homotopically coherent.
One may try to generalize theorem A.1 to proper Calabi-Yau categories, but allowing shifted Poisson $\infty$-algebras and maps of those.
It is to be expected that proper Calabi-Yau categories can be strictified to a cyclic one.\footnote{See theorem 10.2.2 of \cite{KoSo24} and \cite{ChLe10} for a proof in the algebra case and theorem 2.34 of \cite{AmTu22} for the category case, but where additionally smoothness is required.} 
\end{remark_nn}

Next, we explain a \textbf{quantized} version of this story.
It applies to quantizations of cyclic $A_\infty$-categories $\mathcal{C}$ of degree $d$; cyclic categories whose Hamiltonian (see definition \ref{Hamint}) lifts to a solution to above mentioned Quantum Master Equation (see definition \ref{uzwq}); those are sometimes called quantum $A_\infty$-categories.
Algebraically, we find a slight variant of so called \textbf{Beilinson-Drinfeld} ($BD$)-\textbf{algebras}, which we refer to as twisted, see definition \ref{BD}.
A closely related notion are Batalin-Vilkovisky algebras.

We can deform the domain and target of the LQT theorem by adjoining a formal variable $\hbar$, adding a free BV differential and twisting by the quantized Hamiltonian. We denote the obtained BD algebras by  $\mathcal{F}^q(\mathcal{C})$ respectively $\mathcal{O}^q(\mathfrak{gl}_NA_\mathcal{C}).$  
\begin{thm_nn_c}[see theorem \ref{LQT_q} and theorem \ref{Main}] Given a quantization of an essentially finite cyclic $A_\infty$-category of odd degree $d$, there is a commutative diagram of $(d-2)$ twisted $BD$ algebras
$$
\begin{tikzcd}
    \mathcal{F}^q(\mathcal{C})\arrow{r}{LQT^q_N}&\mathcal{O}^q(\mathfrak{gl}_NA_\mathcal{C})\\
s\mathcal{RG}^{tw}_{\mathcal{C}}\llbracket\hbar^2\rrbracket\arrow{u}&s\mathcal{G}^{tw}\llbracket\hbar\rrbracket\arrow{l}\arrow{u}{},
\end{tikzcd}$$
functorial with respect to strict cyclic $A_\infty$-morphisms of quantizations of cyclic $A_\infty$-categories.
Specifically, the upper horizontal map is a weighted map of twisted BD algebras, a notion we introduce see definition \ref{weimap}.
The lower horizontal map is dual to such a map.  

 For $N\rightarrow\infty$ the map $LQT^q_N$ becomes a quasi-isomorphism.
\end{thm_nn_c} 
The symbol $s\mathcal{RG}$ denotes what is called the stable ribbon graph complex (definition \ref{sRG}) and has a close relationship with the \textbf{Deligne-Mumford compactification} of the \textbf{moduli space of Riemann surfaces} \cite{Ha07}, \cite{Ba10}. $s\mathcal{G}$ denotes the \textbf{stable graph complex}.
In both these complexes the BD structure is induced by gluing leaves; further the graph differential also creates loops, dual to loop contraction and is twisted by a Maurer-Cartan element to the BD algebra structure, respectively.

Lastly theorem A.1 and theorem C fit together into a commutative diagram (theorem \ref{Main}) by what we call \textbf{dequantization maps} (definition \ref{dq}).
\par
The vertical map(s) on the right hand in theorem A (and C) are also defined for a cyclic $L_\infty$-algebra (and a quantization thereof) even if this structure is not induced by a cyclic $A_\infty$-category (or a quantization thereof) and are functorial with respect to strict cyclic $L_\infty$-morphisms, see theorem \ref{fde}.
Thus these maps express the tree complex (respectively the stable graph complex) as a \textbf{universal source} of the Lie algebra cochains of cyclic $L_\infty$-algebras (respectively their quantizations). \par
\textbf{Acknowledgements.} I would like to thank my advisors Grégory Ginot and Owen Gwilliam for continuous feedback and support.
A conversation with Junwu Tu in fall 2023 lead me to think about theorem C.
My PhD is founded by the European Union’s Horizon 2020 research and innovation programme under the Marie Skłodowska-Curie grant agreement No 945332.
\subsection{Overview}
 We begin by expanding upon the introduction: \begin{itemize}
\item In 1.2 we elaborate on Calabi-Yau categories, enumerative invariants and matrix models. 
\item In 1.3 we give some more details on the constructions. This allows us to comment on the relation to other works, notably Costello's work on TCFT's \cite{Cos07a}, Barannikov's work on modular operads \cite{Ba10}.
Lastly we comment on the appearance of these structures in geometry and physics, eg. Cieliebak-Fukaya-Latschev \cite{CiFuLa20}, Costello-Li \cite{CL15}.
\end{itemize}
The rest of the paper is organized as follows:
\begin{itemize}
    \item In section 2 we give the definitions of d twisted BD algebras, weighted maps of BD algebras and dequantization maps, next to fixing some notation and recalling other definitions.
    \item Section 3.1 is devoted to the algebraic structures on cyclic cochains of a cyclic $A_\infty$-category.
    \item Section 3.2 is devoted to algebraic structures on Lie algebra cochains of a cyclic $L_\infty$-algebra. 
    \item In section 4 we set up the LQT map for $A_\infty$-categories and prove the large $N$ statement. We explain how to incorporate the shifted Poisson and BD algebra structures.
    \item Section 5 is about various graph complexes and their algebraic properties induced by gluing leafs. Further we identify certain Maurer-Cartan elements, which play a crucial role later.
    \item In section 6 we explain the generalized Kontsevich's cocycle construction, both in the commutative and non-commutative setting.
    \item In section 7 we show how classically and quantized LQT map, graph complexes and generalized Kontsevich cocycle construction fit together, and how quantum and classical setting are connected by dequantization maps. 

\end{itemize}

\subsection{Enumerative Invariants from Calabi-Yau Categories}\label{MM from CY}
 To provide a motivation for question (\ref{Q}), we make some historical remarks:
 \textit{Enumerative} mirror symmetry postulates \cite{can91} a close relationship between numbers on the one hand obtained from a given symplectic manifold and on the other hand obtained from a complex manifold, which are additionally both Calabi-Yau manifolds and dubbed mirror partner.
 In the prior case these numbers go under the name of Gromov-Witten invariants.
 Kontsevich's visionary approach \cite{Kon94} to mirror symmetry was:
 Let us not try to directly compare numbers obtained from these two geometric set-ups.
 Let us associate categories to the geometric context, concretely the so called Fukaya category associated to the symplectic manifold
 and the derived category of coherent sheafs associated to the complex manifold and then compare these categories.
 The statement of \textit{homological} mirror symmetry is then that these two categories should be equivalent in a suitable sense if the two geometric contexts are 'mirror'.
 Further, the categories that we find are Calabi-Yau, which roughly means that there are non-degenerate pairings on their Hom-spaces.
$$\begin{tikzcd}
    Numbers\arrow[bend left=60]{r}{Kontsevich}& Categories\arrow[bend left=60, swap, dotted]{l}{?}
\end{tikzcd}$$
However, now one may ask the question \textbf{how it is possible to pass back from these categories to the numbers} that we were originally interested in.

A proposal has been given in the works of Costello \cite{Cos05}, Caldararu-Tu \cite{CaTu24}.
They tell us that given a smooth and proper Calabi-Yau $A_\infty$-category, together with an additional datum `s', a splitting of its noncommutative Hodge filtration, which is canonically determined in the geometric contexts of mirror symmetry (see sections 6.2 and 6.3 of \cite{AmTu22}), we recover an element \begin{equation}\label{ccei}
\mathcal{D}^{\mathcal{C},s}\in Sym\left(HH^*(\mathcal{C})\llbracket u\rrbracket \right).\end{equation}

We can understand $\mathcal{D}^{\mathcal{C},s}$ just as a polynomial in variables given by powers of $u$ and the Hochschild cohomology of $\mathcal{C}$.
Roughly speaking\footnote{Indeed, one expects to match the geometric invariants only when considering a family of categories and taking a certain limit. We ignore this point here.} the coefficients of these variables of $\mathcal{D}^{\mathcal{C},s}$ should be the enumerative invariants that we are interested in.
The following conjecture, likely going back to Kontsevich, see also Costello \cite{Cos05}, phrases this expectations: 
\begin{Conj_nn}
Let $X$ be a compact symplectic manifold\footnote{We write  $\Lambda$ for its Novikov ring.} and denote $\mathcal{C}=Fuk(X)$ its Fukaya category.
Then there exist a canonical splitting s and an isomorphism $HH^*(\mathcal{C})\cong H^*(X,\Lambda)$ and for $\gamma_i\in H^*(X,\Lambda)$ 
the coefficient of  $\gamma_1u^{k_1}\cdots\gamma_nu^{k_n}$ of 
$\mathcal{D}^{\mathcal{C},s}\in Sym\left(HH^*(\mathcal{C})\llbracket u\rrbracket\right)$
is the GW-invariant
$$\int_{\bar{\mathcal{M}}_{g,n}(X)}\psi^{k_1}ev_*(\gamma_1)\dots \psi^{k_n}ev_*(\gamma_n).$$
\end{Conj_nn}
In fact for the simple case of $X=pt$ it is true:
\begin{thm_nn}[\cite{Tu21}]
For $\mathcal{C}=\mathbb{Q}$, the category determined by the algebra of rational numbers, we have $HH^*(\mathbb{Q})\cong \mathbb{Q}$,
there is a unique splitting $s$ and the coefficient of  $u^{k_1}\cdots u^{k_n}$ of $\mathcal{D}^{\mathcal{C},s}\in Sym\left(\mathbb{Q}\llbracket u\rrbracket\right)$ is
$$\int_{\bar{\mathcal{M}}_{g,n}}\psi^{k_1}\dots \psi^{k_n},$$
\end{thm_nn}
ie exactly \textbf{the intersection numbers on the moduli space of Riemann surfaces}.
\begin{remark_nn}
 The theory of Costello, Caldararu-Tu tries to generalize (closed) Gromov-Witten invariants to the case of arbitrary Calabi-Yau categories. A crucial part of this theory is the representing map, induced from the closed TCFT \cite{Cos05} associated to a Calabi-Yau category by a sort of totalization; see section 6 of \cite{CaTu24}.
  Geometrically, there exists also open Gromov-Witten invariants. Our constructions, especially the middle vertical map from theorem E are relevant to the analogous theory of \emph{open} categorical invariants. Indeed, the middle vertical map from theorem E should be induced from the \emph{open} TCFT associated to a Calabi-Yau category by totalization, see section \ref{2dop}, making clear the analogy to the work of Costello, Caldararu-Tu. 
\end{remark_nn}
\subsubsection*{Matrix Models}
 Are there other tools to package these enumerative invariants?
 Indeed, both in mathematics and physics a different way to encode enumerative information is known, namely through \textbf{large $N$ matrix models}.
 One advantage is that they are amenable by different means as well, for instance formal PDE methods.
\begin{itemize}
    \item 

The prime example for such a matrix model comes from Kontsevich \cite{Kon92a}.
Here he  considered the following  integral over hermitian $N\times N$ matrices and proved that when $N \rightarrow \infty$ it can be written, suitably normalized, as a rational polynomial
in variables $tr(\Lambda^{-2i+1}), i\in\mathbb{N} $, where $\Lambda$ is just some parameter matrix.
$$\lim_{N\rightarrow \infty}\int_{\mathfrak{h}_N}e^{itr(X^3)}e^{tr(X\Lambda X)}dX\ \in\ \mathbb{Q}[tr(\Lambda^{-1}),tr(\Lambda^{-3}),tr(\Lambda^{-5}),..].$$
Finally he related the coefficients of this polynomial to the mentioned \textbf{intersection numbers on the moduli space of Riemann surfaces}.
\item
Another instance of a large $N$ matrix model appears in the r-spin intersection theory case \cite{BCEGF23}.
Conjecturally,  under mirror symmetry the partition function of this matrix model should be related, using the machinery \eqref{ccei}, to categories of matrix factorizations \cite{CaLiTu18}, which also have a Calabi-Yau structure.
\item
We should also mention the case of large $N$ Chern-Simons on $S^3$, which is an infinite dimensional matrix model, aka. a gauge theory, that has been conjectured (\cite{GoVa99}) to describe the Gromov-Witten invariants of the resolved conifold.
See \cite{EkSh25} for recent work.
\item In a similar vein Costello-Gaiotto \cite{cg21} describe a duality between the so called large $N$ $\beta\gamma$-system and the B-model of the deformed conifold, described by the Calabi-Yau category of coherent sheafs.
They conjecture that this duality describes a subsector of the celebrated $AdS_5/CFT_4$ correspondence in physics \cite{Ma99}. 
\end{itemize}
Summarizing, there sometimes seem to be two different ways to encode the enumerative invariants of a Calabi-Yau category, extracted from a geometric context.\begin{enumerate}[label=\roman*)]
    \item
 On the one hand one can pass through the works of Costello, Caldararu, Tu. 
 \item On the other hand sometimes there is a  large $N$ matrix model, also encoding these invariants.
\end{enumerate}
A natural question is whether there is an abstract theory that tells us how and when to obtain large $N$ matrix models from Calabi-Yau categories.
Furthermore one may ask for such an overarching framework within which one can relate i) and ii).
Work in progress \cite{AmTu25}\cite{Ul25} is in particular about proposing an answer to this question.

The BV-formulation, quantized version of the Loday-Quillen-Tsyagan theorem is an integral part of that proposed answer, eg that is how the `large $N$ part' comes into play \cite{GGHZ21}, see also \cite{CL15}.

\subsection{Some more Details and Relation to Other Work}
We elaborate on how to obtain the vertical maps in our main theorems, at the example of the ribbon graph side.
By doing so we explain precisely the connection to Kontsevich's original construction. 
Further this allows us to comment on the relation  to previous work. 
\subsubsection*{Introduction Continued}
Given a set $B$, not necessarily finite, next to the ribbon tree complex $\mathcal{RT}_{B}$ and the stable ribbon graph complex 
$s\mathcal{RG}_{B}\llbracket\gamma\rrbracket$ previously mentioned, there is also the ribbon graph complex
$\mathcal{RG}_{B}\llbracket\gamma\rrbracket$, which can also be equipped with a twisted BD algebra structure.
This result and following theorem are explained in theorem \ref{rbg_comp}:
\begin{thm_nn_d}
There is a commutative diagram
    \begin{equation}\label{icnam} \begin{tikzcd}
\mathcal{RT}_{B}&\mathcal{RG}_{B}\llbracket\gamma\rrbracket\arrow[swap]{l}{p}\arrow{r}{v}&
s\mathcal{RG}_{B}\llbracket\gamma\rrbracket\arrow[bend left=25]{ll}{p}.
    \end{tikzcd},    
    \end{equation}
    where the object on the left is a shifted Poisson algebra, the other two are twisted BD algebras.
    The $p$ maps are dequantization maps and $v$ is a map of twisted BD algebras.
\end{thm_nn_d}
Let 
$$V_B:=\big(V_{ij}, \langle\_,\_\rangle_{ij}\big)_{i,j\in B}$$
be a collection of cyclic chain complexes of degree $d$, that is chain complexes $V_{ij}$ indexed by $i,j\in B$ together with symmetric non-degenerate pairings $\langle\_,\_\rangle_{ij}:V_{ij}\otimes V_{ji}\rightarrow k[-d]$ .
Note that Kontsevich-Soibelman \cite{KoSo24} refer to such an object as a quiver.
We can equip the cochain complex, the \textbf{`observables of the free theory'}
\begin{equation}
\mathcal{F}^{fr}(V_B):=\big(Sym(Cyc^*_+(V_B)[-1]),\cdot,d,\{\_,\_\}\big)
\end{equation}
with a (d-2) shifted Poisson dg algebra, see theorem \ref{fP}. 
\begin{lemma_nn}[See \ref{Ham1}]
A cyclic $A_\infty$-category $\mathcal{C}$ with underlying collection of cyclic chain complexes $V_B$, ie in particular $B=Ob\mathcal{C}$, induces a Maurer-Cartan element \begin{equation}\label{Hamint}
I\in MCE(\mathcal{F}^{fr}(V_B)).
\end{equation}  
\end{lemma_nn}
Similarly given $V_B$ a collection of cyclic chain complexes of degree d the cochain complex,\footnote{Again we only show the $\mathbb{Z}/2$ graded object of a possibly $\mathbb{Z}$ graded object.}
the \textbf{`observables of the prequantum theory'}
 \begin{equation}\label{pqobs}
\mathcal{F}^{pq}(V_B):=\big(Sym\big(Cyc^*(V_B)[-1]\big) \llbracket\gamma\rrbracket ,\cdot,d+\nabla+\gamma\delta,\{\_,\_\}\big)
\end{equation}
 can be equipped, see theorem \ref{nc_pq} with a (d-2) twisted BD algebra structure. 
 Further there is a dequantization map 
 $$p:\mathcal{F}^{pq}(V_B)\rightarrow \mathcal{F}^{fr}(V_B).$$
 A \textbf{quantization} of $\mathcal{C}$ above is by definition a Maurer-Cartan element \begin{equation}\label{uzwq}
    I^q \in\mathcal{F}^{pq}(V_B)
    \end{equation}
    such that $p(I^q)=I$.
\begin{thm_nn_e}[See theorem \ref{comp_rg_al}]
Let $\mathcal{C}$ be a cyclic $A_\infty$-category with underlying collection of cyclic chain complexes $V_B$.
Denote by $I\in MCE(F^{fr}(V_B))$ its associated hamiltonian.
This datum induces the two vertical maps of shifted Poisson algebras on the left of the diagram below.
Given a quantization of $\mathcal{C}$, denoted $I^q\in MCE(F^{pq}(V_B))$, we further get the right vertical map of BD algebras, compatible as follows:
\begin{equation}\label{tot_nc}
\begin{tikzcd}
\mathcal{F}^{fr}(V_B)&\mathcal{F}^{pq}(V_B)\arrow{l}{p}&
\mathcal{F}^{pq}(V_B)\arrow[equal]{l}\\
\mathcal{RT}_{B}\arrow{u}{\rho_{I}}&\mathcal{RG}_{B}\llbracket\gamma\rrbracket\arrow[swap]{l}\arrow{r}\arrow{u}{\rho_{I}}&
s\mathcal{RG}_{B}\llbracket\gamma\rrbracket\arrow[bend left=20]{ll}{p}\arrow{u}{\rho_{I^q}}.
    \end{tikzcd}
    \end{equation}
\end{thm_nn_e}
\begin{thm_nn_f}[See theorems \ref{RS_mce}, \ref{sRG_mce} and lemmas \ref{comp_1} and \ref{comp_2}]
There is 
$$D\in MCE(\mathcal{RT}_{B})\ \text{and}\ S\in MCE(s\mathcal{RG}_{B}\llbracket\gamma\rrbracket)$$
such that $p(S)=D$ and further $\rho_I(D)=I$ and $\rho_{I^q}(S)=I^q$.
\end{thm_nn_f}
\begin{cor_nn}
Twisting the outer diagram of theorem E by the compatible Maurer-Cartan elements results in another commutative diagram, which is how we obtain the left vertical maps of theorem A and C and their compatibility by dequantization maps.
\end{cor_nn}
\begin{remark_nn}[Kontsevich's Original Construction]
Let us denote by $\mathcal{M}_{g,b}$ the moduli space of connected Riemann surfaces of genus $g$ with $b$ marked points. It is a standard result (for us particularly relevant Costello's \cite{Cos07b} proof) that 
$$\mathcal{RG}_{m=0}^{con.}\simeq C_*(\amalg_{stable} \mathcal{M}_{g,b},\mathcal{L}),$$
where $\mathcal{RG}_{m=0}^{con.}$ denotes connected ribbon graphs without leafs and $\mathcal{L}$ is a certain line bundle.
Let $A$ be a cyclic $A_\infty$-algebra and $I\in MCE(\mathcal{F}(V))$ its associated hamiltonian.
Denote by $\rho|_{m=0}$ the restriction of the middle vertical map $\rho_I$ from diagram \ref{tot_nc} to finite sums of connected graphs without leaves and to the $\gamma=0$ part.
It then gives a map 
$$\begin{tikzcd}
    \mathcal{RG}_{m=0}^{con.}\arrow{r}{\rho|_{m=0}}&k[\nu,\gamma]\arrow{r}{\nu=\gamma=1}&k
\end{tikzcd}$$
which coincides with Kontsevich's construction \cite{Kon92b}, ie produces a cocycle on $\amalg\mathcal{M}_{g,b}$.
Indeed, as a corollary of above theorem E this composition is a chain map, ie determines a cochain.
In fact in comparison to Kontsevich's construction we are additionally weighting his map by $\gamma$ to the power of the genus of the respective ribbon graph.
\end{remark_nn}

\begin{remark_nn}
All the results have analogues for cyclic $L_\infty$-algebras and (ordinary) graph complexes, which we explain in the bulk of this paper.
Further we can recover Penkava's construction \cite{pe96}, the commutative analog of Kontsevich's construction. 
\end{remark_nn}

\subsubsection{2d Topological Field Theories}\label{2dop}
We recall the category $\mathcal{O}_B:=C_*(\mathcal{M}^{o}_B,\mathcal{L})$ from \cite{Cos07a}, definition 3.0.2.
Its objects are tuples of elements of B and its Hom spaces are built from chains on the moduli space of stable Riemann surfaces with boundary
and intervals embedded into the boundary, the free boundaries decorated by objects of a set $B$.

\begin{figure}[h]
\centering
\includegraphics[]{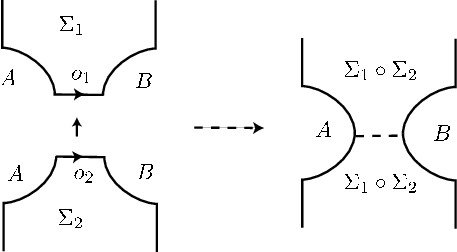}
\caption{An example of composition in $\mathcal{O}_B$, from \cite{Cos07a}}
\end{figure}

In loc.cit. Costello proved that a cyclic $A_\infty$-category $\mathcal{C}$ of odd dimension with set of objects $B$ is the same as an open TCFT, that is a symmetric monoidal functor 
$$F:(\mathcal{O}_B,\cup)\rightarrow (Ch,\otimes).$$
Further we have that \begin{equation}\label{tensf}
F\big((\lambda_1^a,\lambda_1^b)\cup\cdots\cup (\lambda_n^a,\lambda_n^b)\big)=Hom_\mathcal{C}(\lambda_1^a,\lambda_1^b)\otimes\cdots \otimes Hom_\mathcal{C}(\lambda_n^a,\lambda_n^b) \end{equation}

Let us denote by 
$$Tot^+(\mathcal{O}_B):=\bigoplus_{n\in\mathbb{N},(x_1,\cdots, x_n)\in Ob\mathcal{O}_B^{\times n}} Hom_{\mathcal{O}_B}(x_1\cup\cdots\cup x_n,\varnothing)/S_n$$
the direct sum over all Hom-spaces with empty output, quotient out by the symmetric group acting on the ordering of the inputs.
In the modular operad context this is known as totalization, see \cite{DoShVaVa24}.
Totalization applied to the closed TCFT obtained from $\mathcal{C}$ is crucial in the work of Costello, Caldararu-Tu on categorical enumerative invariants \cite{Cos05} \cite{CaTu24}.
Let us denote by $V_B$ the collection of cyclic chain complexes underlying $\mathcal{C}$.
\begin{Claim}[`Conjecture']\label{ibssm}
 We claim that $F$ induces a map 
 $$F_{tot}:Tot^+(\mathcal{O}_B)\rightarrow \mathcal{F}^{pq}(V_B).$$
We further claim that
$$  Tot^+(\mathcal{O}_B)\simeq \mathcal{RG}_{B}$$
and under this conjectural equivalence $\rho_I=F_{tot}$ when extended $\gamma$-linearly (see diagram \ref{tot_nc}).
Further this quasi-isomorphism is compatible with the operation of gluing open boundaries together and corresponds to gluing leafs together. 
\end{Claim}
Up to signs and possible gradings, this claim should follow from Costello's work, eg Proposition 6.2.1 of \cite{Cos07a} and generally section 6.2 of \cite{Cos07a}, see also \cite{Cos07b} and may be straightforward to experts.\footnote{The idea is that we can reorder the tensor factors as in \ref{tensf}, quotiented by the symmetric group, by remembering their preimage under $F$.
For each element of $Tot^+(\mathcal{O}_B)$ we obtain a symmetric ordering of cyclic orderings of elements of $B$, provided by the coloring of free boundaries in the boundary components of a surface.} 

We did not check the details on signs and gradings and thus formulated the above as a conjecture.

\subsubsection{Cyclic and Modular Operads}
Theorem 1 of Barannikov's paper \cite{Ba10} (see also the very good exposition \cite{bra12}, section 8.2.1.) describes, given a modular operad $\mathcal{P}$, a bijection between: \begin{itemize}
    \item 
 algebras on a chain complex $V$ over an operad built from the Feynman transform of $\mathcal{P}$, 
 \item  Maurer-Cartan elements in a dg shifted Lie algebra built out of $V$ and $\mathcal{P}$.
 \end{itemize}
In section 10 of loc.cit. Barannikov constructs an explicit modular operad $\mathbb{S}[t]$.
In theorem 2 of loc.cit. he shows that the Feynman transform of $\mathbb{S}[t]$ is, after applying a totalization functor similar as in the previous subsection, the same as $s\mathcal{RG}$, at least when restricting to the part without leaves.
Furthermore, restricting to the cyclic operad part of  $\mathbb{S}[t]$ recovers in the similar way $\mathcal{RT}$, see eg \cite{geKa96} 5.9 and section 9 of \cite{Ba10}.
The dg shifted Lie algebra he describes, associated to $\mathbb{S}[t]$ and a chain complex $V$, is basically $\mathcal{F}^{pq}(V)$ considered here and its cyclic part reduces to $\mathcal{F}^{fr}(V)$, see around equation 6.1 of loc.cit.
\begin{remark_nn}
One could expect that our diagram \ref{tot_nc} can be obtained as a corollary of Barannikov's result after applying a totalization functor, analogous to the previous subsection.
It is interesting to note the existence of Maurer-Cartan elements in the induced dg-Lie algebras.
\end{remark_nn}
A similar construction may be available for the commutative side:
That is one may ponder the existence of a modular operad, analogous to $\mathbb{S}[t]$, which is a modular extension of $Com$ and whose Feynman transform gives $s\mathcal{G}$.
The LQT map may then be derived as induced from a Morita invariance map and a map from $\mathbb{S}[t]$ to this  modular generalization of $Com$, which would naturally explain our compatibility result.
Compare also other works of Barannikov, eg \cite{Ba10b}, who worked on these ideas before.

Furthermore in our examples we allow not necesarilly connected graphs, which gives rise to a commutative product.
Form an operadic point of view this is studied in \cite{djpp22}.
\begin{remark_nn}
 Summarizing, the outer two objects of \ref{icnam} can be understood as the totalization of the Feynam transform of $Ass$ and $\mathbb{S}[t]$.
 In \cite{co04} Costello characterized the middle object of \ref{icnam} as the (totalization of the) modular envelope of the cyclic $A_\infty$-operad.
 Is there an operad theoretic way to understand diagram \ref{icnam}?
 Similarly for the graph complex analog, see diagram \ref{ag_dq}.  
\end{remark_nn}
\begin{remark_nn}
    On the other hand, from \cite{Cos07b} the geometric meaning of the two left objects of \ref{icnam} are clear,
    describing chains on moduli spaces of disks with marked points on the boundary respectively arbitrary Riemann surfaces with marked points on the boundary.
    On the other hand the geometric meaning of the complex on the right of diagram \ref{icnam} is not definitely clear to the author. 
\end{remark_nn}
\subsubsection{Geometry and Physics}\label{gap}
The algebraic structures studied here, that is the non-commutative part, and notably their generalization to certain infinite dimensional examples as
$\Omega^*(M)$ de-Rham forms on a manifold $M$ are also studied in \cite{CiVo23}, \cite{CiFuLa20} and \cite{Che10}, with applications to string topology.

In \cite{sla24} Slawinski has explicitly analyzed the lift of the $A_\infty$-structure on the Fukaya category of an elliptic curve to a quantum $A_\infty$-structure by using a certain combinatorial model.

There are general analytic approaches to deal with $A_\infty$- and $L_\infty$- algebras which are cyclic only on homology,
see Costello's work on renormalizaton \cite{Co11} and notably \cite{CL15} by Costello-Li.
In this work the authors construct a shifted Poisson structure for the
`cyclic' algebra of holomorphic de-Rham forms on a compact Calabi-Yau manifold, leading to so called large $N$ holomorphic Chern Simons under the LQT map.
The techniques developed by Cieliebak and Volkov \cite{CiVo23} to deal with
the `cyclic' algebra of de Rham forms on a compact 3-manifold should similarly lead to large $N$ Chern Simons theory, at least classically.

Naef-Willwacher \cite{NaWi19} study related phenomena and especially section 7.4 has many formal similarities to our work, which we would like to study further in the future.

%% file: algebra.tex
We fix $k$ a field of characteristic 0, unless remarked otherwise.
Furthermore we work in the category of chain complexes over $k$, denoted $Ch$, and stick to the homological convention.
The suspension $[\_]$ of a chain complex $V$ is defined in such a way that if $V$ is concentrated in degree zero, then V[1] is concentrated in degree $-1$.
The results of this paper also apply to the $\mathbb{Z}/{2}$-graded context.

See e.g 1.4.2 of \cite{CPTTV17} for the following definition.
\begin{df}
An \emph{$r$-shifted Poisson differential graded algebra} is a differential graded, not necessarily unital, k-algebra such that 
\begin{itemize}
    \item 
 the differential has degree 1
 \item it has a Poisson bracket of degree $r$, that is it is a Lie bracket of degree $r$ that fulfills the Leibniz rule with respect to the product and is compatible with the differential.
 \end{itemize}

A \emph{map of shifted Poisson dg algebras} is both a map of the underlying dg-algebras and the shifted Lie algebras.    

We denote by $(Poiss_{r})$ the category whose objects are $r$-shifted Poisson algebras and whose morphisms are maps of $r$-shifted Poisson algebras.
\end{df}

 \begin{remark} 
 Note that by forgetting the multiplication a shifted Poisson dg algebra forgets to a dg shifted Lie algebra.
 When we say Maurer-Cartan element of a shifted Poisson dg algebra we refer to a MCE of this dg shifted Lie algebra.
 We can \textit{twist} both this dg shifted Lie algebra as well as the overlying shifted Poisson algebra by such a MCE.
\end{remark}

We restrict now to $r$ an odd integer.

\begin{df}\label{BD}
An \emph{$r$-twisted Beilinson-Drinfeld ($BD$) algebra} is a graded commutative unital algebra over the ring $k\llbracket \mu\rrbracket$ of formal power series in a variable $\mu$, which has cohomological degree $1-r$ such that 
\begin{itemize}
    \item it is endowed with a Poisson bracket $\{\_,\_\}$ of degree $r$, which is $R\llbracket \mu\rrbracket$-linear
    \item 
    it is endowed with an $R\llbracket\mu\rrbracket$-linear map $d$ of degree 1 that squares to zero 
    \item
    and such that
$$d(a\cdot b)=d(a)\cdot b+(-1)^{|a|}a\cdot d(b)+\mu \{a,b\}.$$
\end{itemize}
\end{df}

\begin{remark}
We will also encounter $r$-twisted $BD$ algebras whose multiplication has even degree $m$.
Here we require the formal variable to have degree $1-r+m$.
Note that we could always reduce to the previous definition by performing a shift.\footnote{We do not do so since then the quantized LQT map would have non zero degree.}
We do not want to introduce additional notation for this case. 
\end{remark}
     
\begin{remark}
In \cite{CPTTV17}, section 3.5.1, the notion of a $BD_d$ algebra is introduced.
It is not clear to the author what is the relationship of this notion and ours, in general.
However, for $r=1$ an $r$-twisted BD algebra is the same as a $BD_0$ algebra.
\end{remark}

\begin{remark}
Since in our paper we will only deal with $r$-twisted $BD$-algebras we will sometimes refer to them as simply BD algebras as well. 
\end{remark}

\begin{remark}
Note that by forgetting the multiplication a BD algebra forgets to a dg shifted Lie algebra.
When we say Maurer-Cartan element of a BD algebra we refer to a MCE of this dg shifted Lie algebra.
Note that we can \textit{twist} not only this dg shifted Lie algebra, but also the overlying BD algebra by such a MCE.
\end{remark}
We will only see maps of BD algebras which have the same dimension.
However, we will encounter maps of BD algebras which do not directly respect the algebra structure, but do so up to a certain weighting and live over different base rings:
%\begin{df}[Weak Map of BD algebra]\label{wBD}
 %   Let A be a dimension s $BD$ algebra over a  $k$-algebra $R$ and formal variable $\gamma$ and B be a dimension $s$ $BD$ algebra over a  k-algebra $S$ and formal variable $\hbar$. A weak map of BD algebras is given by a degree 0 k-linear map $$f:A\rightarrow B$$ and a k-linear map $$g:R[\gamma]\rightarrow S[\hbar]$$ where $f$ is a map of the underlying shifted dg k-Lie algebras (but not necessarily a map of algebras over the bigger base rings!) and satisfies additionally : If $a,b\in A$ then \begin{equation}\label{wBD}g(\gamma)\cdot f\{a,b\}=f(\gamma\cdot \{a,b\}).\end{equation} 
%\end{df}
%\begin{remark}
 %   Equivalent condition, given that A and B are BD algebras, is that
 %   $$df(a\cdot b)=f(da\cdot b)+(-1)^{|a|}f(a\cdot db)+g(\gamma)\cdot\{f(a),f(b)\},$$
  %  that is, it intertwines the BD relation in an interesting way.
 %   Note that this equation is trivially satisfied if g is additionally a map of rings and $f: A\rightarrow g^*B$ is a map of $R[\gamma]$-algebras. We will encounter examples where these conditions are not met, but that are still weak maps of BD algebras.
%\end{remark}}
\begin{df}\label{weimap}
Let A be an $s$-twisted $BD$ algebra over $k\llbracket\gamma\rrbracket$ and B be an $s$-twisted $BD$ algebra over $k\llbracket\hbar\rrbracket$.
Given $t\in \mathbb{N}_{}$ we say that a \emph{$t$-weighted map of BD algebras} is given by a map of algebras
$$g:k\llbracket\gamma\rrbracket\rightarrow k\llbracket\hbar\rrbracket$$
such that 
\begin{equation}\label{pow}
    g(\gamma)=\hbar^t 
    \end{equation}
and a map of $k\llbracket\gamma\rrbracket$ dg shifted Lie algebras $$f:A\rightarrow g^*B$$ such that 
\begin{equation}\label{mult}
f(a_1\cdot \ldots \cdot a_k)=\hbar^{(t-1)(k-1)}f(a_1)\cdot \ldots \cdot f(a_k).\end{equation}
\end{df}
\begin{remark}
    For $t=1$ we recover the standard definition of a map of BD algebras.
    We indicate this by simply dropping the adjective weighted. 
\end{remark}
\begin{remark}
    Composition of a $t$-weighted map of BD algebras and an $s$-weighted map of BD algebras is an $s\cdot t$-weighted map of BD algebras. 
\end{remark}
\begin{df}
  Denote by $(\text{BD}^{r-tw})$ the category whose objects are $r$-twisted BD algebras and whose morphisms are maps of BD algebras.
\end{df}
\begin{remark}
Note that \eqref{pow} is a necessary condition allowing us to extend g to a map
$$\tilde{g}:k(\!(\gamma)\!)\rightarrow k(\!(\hbar)\!)$$
where $\tilde{g}(r\cdot \gamma^{-k})=g(r)g(\gamma)^{-k}.$
We further get an induced map
$$\tilde{f}:A(\!(\gamma)\!)\rightarrow \tilde{g}^*B(\!(\hbar)\!).$$
Conidition \eqref{mult} implies that for any $x\in A$ we have 
  \begin{equation}\label{int_rel}
 \tilde{f}e^{x/\gamma}=\hbar^{1-t} e^{f(x)/\hbar}. 
 \end{equation} 
This guarantees commutativity of 
$$\begin{tikzcd}
MCE(A)\arrow{r}{f}\arrow{d}{exp(-/\gamma)}&MCE(B)\arrow{d}{exp(-/\hbar)}\\
Z(A(\!(\gamma)\!))\arrow{r}{\hbar^{t-1}\tilde{f}}&Z(B(\!(\hbar)\!))
\end{tikzcd},$$
where MCE refers to the Maurer-Cartan set of the respective dg shifted Lie algebra and Z refers to the cycles of the underlying chain complex.\footnote{As is well known, we can replace $Z$ by homology if we replace MCE by its quotient by the gauge group action. } 
For $t=1$ commutativity of the diagram is a standard fact about BD algebras, for general $t$ it follows from equation~\eqref{int_rel}.
\end{remark}
\begin{df}\label{dq}
Let $P$ be an r-shifted Poisson algebra and $R$ be an $r$-twisted BD algebra. 
We call 
$$p:R\rightarrow P$$
 a \emph{dequantization map} if it is a  degree zero map of the underlying chain complexes and of shifted Lie algebras.
\end{df}
\begin{remark}
    Note that we do not ask that this map respects the products, which is what we will see in our examples.
\end{remark}

\section{Algebraic Constructions}
\subsection{Non-Commutative World}\label{ncw}

Let $B$ be a set and $V_B:=\{V_{ij}\}_{i,j\in B}$ be a set of chain complexes.
Then its bar-construction\footnote{In the sense that it is the Bar complex of a dg-category with zero compositions.} is defined by 
\begin{equation}\label{bar}BarV_B:=\bigoplus_{n=1}^\infty \bigoplus_{\lambda_0,\cdots,\lambda_n\in B}V_{\lambda_{0},\lambda_{1}}[1]\otimes V_{\lambda_{1},\lambda_2}[1]\otimes \cdots \otimes V_{\lambda_{n-1},\lambda_n} [1]
\end{equation}
and it is naturally a dg coalgebra $(BarV_B,d)$ whose comultiplication, see e.g. \cite{KoSo24}, example 2.1.6, we don't give an extra symbol.
\begin{df}
A \emph{small $A_\infty$-category} $\mathcal{C}$ is given by a set of objects $Ob(\mathcal{C})$, a set of chain complexes  $\{V_{\lambda_i,\lambda_j}\}_{\lambda_i,\lambda_j\in Ob(\mathcal{C})}$ and a collection of maps
\begin{equation*}
m_{n}\colon V_{\lambda_0,\lambda_{1}}[1]\otimes V_{\lambda_{1},\lambda_{2}}[1]\otimes\ldots\otimes V_{\lambda_{n-1},\lambda_n}[1] \to V_{\lambda_0,\lambda_{n}}[1]
\end{equation*} 
for all $n\geq 1$ and all tupels of objects $(\lambda_0,\lambda_1,\cdots,\lambda_n).$
We can uniquely extend those to a coderivation 
$$m_\mathcal{C}:BarV_{Ob(\mathcal{C})}\rightarrow BarV_{Ob(\mathcal{C})}$$
and we require that $$m_\mathcal{C}^2=0$$ and that it commutes with the internal differential $d$ of the chain complex $(BarV_{Ob(\mathcal{C})},d)$. 
\end{df}
In this case we denote $V_{\lambda_i,\lambda_j}=:Hom_\mathcal{C}(\lambda_i,\lambda_j)$ and 
$(BarV_{Ob(\mathcal{C})},d+m_\mathcal{C})=:(Bar\mathcal{C},d+m_\mathcal{C}).$ 
\begin{remark}
We will only consider small $A_\infty$-categories in this paper and thus drop this adjective from now on.    
\end{remark}
\begin{df}
We say that an $A_\infty$-category $\mathcal{C}$ is \emph{unital} if there is an element 
$1_{i}\in Hom_\mathcal{C}(\lambda_i,\lambda_i)$ for all $\lambda_i\in Ob(\mathcal{C})$ such that 
$m_2(g,1_i)=g$ for all $g\in Hom_\mathcal{C}(\lambda_k,\lambda_i)$, 
$\lambda_k\in Ob(\mathcal{C})$ and $m_2(1_i,f)=f$ for all $f\in Hom_\mathcal{C}(\lambda_i,\lambda_l)$, $\lambda_l\in Ob(\mathcal{C})$
and if in $m_n(a_1,a_2,\cdots)$ any entry is one of the $1_{i}$ then $m_n(a_1,a_2,\cdots)=0$ for $n\geq 3$.
\end{df}
\begin{remark} 
We can somewhat unpack the definition of an $A_\infty$-category:
\begin{itemize}
\item   
The fact that $m_\mathcal{C}^2=0$ is  equivalent to the infinite list of relations, one for each $n \geq 1$, 
\begin{equation}\label{A_inf_re}
    \sum_{\substack{r,t\geq 0, s\geq 1\\r+s+t=n}} m_{r+1+t}(id^{\otimes r} \otimes m_{s} \otimes id^{\otimes t}) = 0.
\end{equation}
 Here $id$ denotes ambiguously the identity on whichever Hom-space it is applied to.
 \item Using this it is easy to see that in the case that $m_\mathcal{C}$ is only non-zero on $Bar^2V_B$,
 i.e. on the summand $n=2$, we recover the definition of a (not necessarily unital) differential graded category; where $f\circ g=m_2(g,f)$ for composable $f,g.$
     
\item In fact we assume that $m_1=0$ for the rest of this article since we can just absorb it into the differential of the underlying chain complexes. 
\end{itemize}
\end{remark}

\begin{df}
A \emph{functor between two $A_\infty$-categories} $F:\mathcal{C}\rightarrow \mathcal{D}$ is given by a map of sets
$$F: Ob(\mathcal{C})\rightarrow Ob(\mathcal{D})$$
and maps
    $$F_n: Hom_\mathcal{C}({\lambda_0,\lambda_1})[1]\otimes Hom_\mathcal{C}({\lambda_1,\lambda_2})[1]\otimes \cdots \otimes Hom_\mathcal{C}({\lambda_{n-1},\lambda_n})[1]\rightarrow Hom_\mathcal{D}(F(\lambda_0),F(\lambda_n))[1],$$
for all $n\geq 1$ and all tupels of objects $(\lambda_0,\lambda_1,\cdots,\lambda_n).$ 
We can uniquely extend this to a map of coalgebras, which we require to determine a chain map $$F_*:\  (Bar(\mathcal{C}),d+m_\mathcal{C})\rightarrow (Bar\mathcal{D}),d+m_\mathcal{D}).$$ 
\end{df}
We denote by $A_\infty cat$ the category whose objects are $A_\infty$-categories and with morphisms given by $A_\infty$-functors.

\begin{remark}
    Again we can unpack the requirement that $F_*$ intertwines the differentials. 
    Equivalently it means that
    \begin{equation}\label{F_inf_rel}
        \sum_{r+s+t=n} F_{r+1+t}(id^{\otimes r} \otimes m_{\mathcal{C},s} \otimes id^{\otimes t}) = \sum_{i_1+\ldots+i_l=n}  m_{\mathcal{D},l}(F_{i_1} \otimes F_{i_2} \otimes \ldots \otimes F_{i_l})
    \end{equation}
for $n\geq 1$ and where $m_1=d$. 
Notably $F_1$ induces maps of chain complexes on the Hom spaces. 
\end{remark}
\begin{df}
 We say that an $A_\infty$-functor $F:\mathcal{C}\rightarrow \mathcal{D}$ of unital $A_\infty$-categories is \emph{unital} if $F_1(1_{i})=1_{F(i)}$ and $F_n(\cdots,1_{i},\cdots)=0$ for all $n\geq 2$. \end{df}
\begin{df}
   We call an $A_\infty$-functor 
   $$F:\mathcal{C}\rightarrow \mathcal{D}$$
   \emph{strict} if $F_n=0$ for all $n\geq 2$. 
\end{df}

 Given an unital $A_\infty$-functor $F: \mathcal{C}\rightarrow \mathcal{D}$ taking homology induces a functor of ordinary categories $H_*F: H_*\mathcal{C}\rightarrow H_*\mathcal{D}.$  

\begin{df}\label{equiv}
We call an unital $A_\infty$-functor $F:\mathcal{C}\rightarrow \mathcal{D}$ an \emph{equivalence} if the induced functor 
$$H_*F_1: H_*\mathcal{C}\rightarrow H_*\mathcal{D}$$
of ordinary categories is an equivalence.
\end{df}

Let us recall some further notions some of which we, strictly speaking, won't need.
Since we alluded to them in the introduction we wish to record them, but are relatively fast in doing so:
\par
Given an $A_\infty$-category $\mathcal{C}$ there is dg-category of $\mathcal{C}$-bimodules as described in \cite{Ga19}, section 3.1 and references therein.
Further in loc. cit. it is explained that there are certain objects called the diagonal bimodule $\mathcal{C}_\Delta$, respectively the linear dual diagonal bimodule $\mathcal{C}_\Delta^\vee$ in this category.
Then one can define 
\begin{df}
$$CH_*(\mathcal{C}):=\mathcal{C}_\Delta\otimes^{\mathbb{L}}_{\mathcal{C}^e}\mathcal{C}_\Delta\in Ch,$$
 given by the derived tensor product in the category of $\mathcal{C}$-bimodules (see e.g. \cite{Ga19}, section 3.4).
 Its homology is called \emph{Hochschild homology}.
\end{df}
\begin{remark}
There are explicit complexes that compute this derived tensor product, which is one way to see that there are maps of chain complexes, for all objects $x$ of $\mathcal{C}$
$$\iota_x: Hom_\mathcal{C}(x,x)\rightarrow CH_*(\mathcal{C}).$$
\end{remark}
Next we introduce important properties of $A_\infty$-categories:

\begin{df}
    We say that an $A_\infty$-category $\mathcal{C}$ is \emph{smooth} if $\mathcal{C}_\Delta$ is a compact, i.e. a perfect object in the dg category of $\mathcal{C}$-bimodules.

    We say that an $A_\infty$-category $\mathcal{C}$ is \emph{proper} if $Hom_\mathcal{C}(x,y)\in Ch$ is perfect for all $x,y\in Ob(\mathcal{C}).$ 
\end{df}
\begin{remark}
Compare also \cite{BrDy19} ex. 2.8 and ex. 2.9 for a more abstract phrasing, in the setting of dg categories, which shows that the notions smooth and proper are complementary in a certain sense.
\end{remark}
Let us come back to the notion of Hochschild homology.
One of its important properties is that there is an $S^1$-action on Hochschild chains (\cite{co85}, in our context e.g \cite{Ga23}, section 3.2).
By taking its homotopy $S^1$ orbits we arrive at
\begin{df}
   Given an $A_\infty$-category we call the homology of $$CC_*(\mathcal{C}):=CH_*(\mathcal{C})_{hS^1}\in Ch$$ its \emph{cyclic homology}. 
\end{df}
\begin{remark}\label{cph}
Thus just by abstract reasons we  get a map $\pi:CH_*(\mathcal{C})\rightarrow CC_*(\mathcal{C}).$ \end{remark}
\begin{df}
Given $\mathcal{C}$ a proper $A_\infty$-category we 
call a choice of quasi-isomorphism of $\mathcal{C}$-bimodules
$$\mathcal{C}_\Delta \simeq\mathcal{C}_\Delta^\vee[-d]$$
a \emph{weak proper Calabi-Yau} structure of dimension $d$.
\end{df}
\begin{remark}[See rmk. 49 of \cite{Ga23}]
A weak proper Calabi-Yau structure of dimension $d$, also called weak right Calabi-Yau structure, may be equivalently described by saying that there is a map of complexes 
$$tr: CH_*(\mathcal{C})\rightarrow k[-d]$$
such that the induced map 
$$\begin{tikzcd}
    Hom_\mathcal{C}(x,y)\otimes Hom_\mathcal{C}(y,x)\arrow{r}{m^2_{x,y}}& Hom_\mathcal{C}(x,x)\arrow{r}{\iota_x}& CH_*(\mathcal{C})\arrow{r}{tr}& k[-d]\end{tikzcd}$$
is non-degenerate on homology, for all objects $x,y\in Ob\mathcal{C}$.
If we additionally ask that the trace map comes from a map from cyclic homology along the map $\pi$ from remark \ref{cph}, this is called a strong proper Calabi-Yau structure.
\end{remark}
\begin{remark}
There is (are) also the notion(s) of a (weak and strong) smooth Calabi-Yau, also called (weak and strong) left Calabi-Yau structure on a smooth $A_\infty$-category.
See definitions 12 and 13 of \cite{Ga23}.
We won't need this here.   
\end{remark}
Finally we want to consider a concrete model computing the cyclic (co)homology of an $A_\infty$-category:
\begin{cons}
    Let $\mathcal{C}$ be an $A_\infty$-category and consider the map given by cyclic permutation:
    $$\begin{tikzcd}
    \hspace*{-0.5cm}
     Hom(\lambda_0,\lambda_1)[1]\otimes Hom(\lambda_1,\lambda_2)[1]\otimes \cdots \otimes Hom(\lambda_{n},\lambda_0)[1]\arrow{d}{t}&\ni&a_{01}\otimes a_{12}\otimes a_{23}\otimes \cdots\otimes a_{no}\arrow[mapsto]{d} \\ 
     \hspace*{-0.5cm} Hom(\lambda_1,\lambda_2)[1]\otimes Hom(\lambda_2,\lambda_3)[1]\otimes \cdots \otimes Hom(\lambda_{0},\lambda_1)[1]&\ni&(-1)^* a_{12}\otimes a_{23}\otimes\cdots\otimes a_{no}\otimes a_{01}
    \end{tikzcd}$$
    Here $(-1)^*$ comes from the Koszul sign rule.
    We extend $t$ to  $Bar\mathcal{C}$ by zero on the subspaces that are not of the above form, i,e, where $\lambda_n\neq \lambda_0$ and denote it by the same letter.
    \end{cons}
    \begin{lemma}
Given an $A_\infty$-category the differential $m_\mathcal{C}+d$ on $Bar\mathcal{C}$ restricts to ${ker(1-t)}$.
    \end{lemma}
\begin{proof}
        This is a standard result, which follows in the same way as in lemma 2.1.1 of \cite{Lo13}.
\end{proof}
    Finally we have to include a degree shift by 1: 
\begin{theorem}[`Cyclic Words']
The cyclic chain complex computes cyclic homology $$(Cyc_*^+(\mathcal{C}),m_\mathcal{C}+d):=\big(ker(1-t)[-1],m_\mathcal{C}+d\big)\simeq CC_*(\mathcal{C})$$ 
\end{theorem}
\begin{proof}
    This is well known.
    For the algebra case see \cite{Lo13}, theorem 2.1.5, in our generality e.g around definition 2.5 and remark 2.6 of \cite{AmTu22}.
\end{proof}
% If confused again for grading convention, check Loday, page 177. We follow that convention.
\begin{df}
  The cochain complex $$(Cyc^*_+(\mathcal{C}),m_\mathcal{C}+d):=Hom(Cyc_*^+(\mathcal{C}),k)$$ with the induced differential is called \emph{cyclic cochain complex} and its homology is referred to as the \emph{cyclic cohomology} of $\mathcal{C}$. 
\end{df}

\begin{remark}
    We think of a cyclic cochain as a cyclic word in the dual of the Hom-spaces with matching `boundary condition'.
    We adopt the notation 
    $$(a_1a_2\cdots a_n)\in Cyc^*_+(\mathcal{C})[-1],$$
where by definition $a_i\in Hom(\lambda,\mu)^\vee[-1]$ and $a_{i+1}\in Hom(\mu,\nu)^\vee[-1]$ for $\lambda,\mu,\nu\in Ob(\mathcal{C})$.
\end{remark}
\begin{df}
    A \emph{cyclic $A_\infty$-category} of degree $d$, denoted $(\mathcal{C},\langle\_,\_\rangle_\mathcal{C})$, is an
    $A_\infty$-category $\mathcal{C}$ together with a  symmetric non-degenerate chain map
    $$\langle\_,\_\rangle_{\mu\lambda}:\  Hom_\mathcal{C}(\mu,\lambda)\otimes Hom_\mathcal{C}(\lambda,\mu)\rightarrow k[-d]$$
    for all $\mu,\lambda\in Ob(\mathcal{C})$ such that 
    \begin{equation}\label{cycl}S_\mathcal{C}:=\langle (m+d)\_,\_\rangle\in Cyc^*_+(\mathcal{C})[-1],
    \end{equation}
    i.e. is cyclically symmetric.
    \end{df}
    Note that in this case $S_\mathcal{C}$ is of degree $3-d$.
    We further denote by 
    $$I_\mathcal{C}:=\langle m\_,\_\rangle\in Cyc^*_+(\mathcal{C})[-1],$$
    which has the same degree.

\begin{remark}\label{cy_cyc}
    See remark 9 of \cite{Ga23} for the fact that any cyclic $A_\infty$-category is canonically a strong (and thus weak) proper Calabi-Yau.
    Further it is explained that in characteristic zero one can strictify a strong proper Calabi-Yau category to a cyclic one.
    However, the notion of cyclic $A_\infty$-category is not preserved under general $A_\infty$-functors, whereas the one of proper Calabi-Yau categories is.
\end{remark}
Without giving details we state the result from definition 2.15 of \cite{AmTu22} that there is a functor
$$\Omega^{2,cl}: (A_\infty Cat)^{op} \rightarrow Ch,$$
which associates to an $A_\infty$-category a chain complex called its non-commutative 2-forms.
As described in definition 2.23 of loc. cit. a cyclic structure $\langle\_,\_\rangle_\mathcal{C}$ on an 
$A_\infty$-category $\mathcal{C}$ induces an element $\rho_\mathcal{C}\in\Omega^{2,cl}(\mathcal{C}).$

\begin{df}\label{cycAi}
   Given two cyclic $A_\infty$-categories $(\mathcal{C},\langle\_,\_\rangle_\mathcal{C})$ and $(\mathcal{D},\langle\_,\_\rangle_\mathcal{D})$
   a  \emph{cyclic $A_\infty$-functor} is an $A_\infty$-functor $F:\mathcal{C}\rightarrow \mathcal{D}$ such that

   $$F^*\rho_{\mathcal{D}}=\rho_{\mathcal{C}},$$
   where $\rho_{\mathcal{D}}$ respectively $\rho_{\mathcal{C}}$ refers to the non-commutative 2-form associated to the cyclic structure by the previous remark.
\end{df}
\begin{remark}\label{caf}
    This definition should reduce to definition 5.3 (2) of \cite{HaLa08} in the case of a category with one object.
    Note that the authors use the term symplectic where we used the term cyclic.
    See equation 13 of \cite{AmTu22} for an explicit  equivalent description of definition \ref{cycAi}.
    In particular one has 
    \begin{equation}\label{nfn}
        \langle F_1(x),F_1(y)\rangle_\mathcal{D}=\langle x,y\rangle_\mathcal{C}.
        \end{equation}
\end{remark}
 Denote by  $(cyc_dA_\infty cat)_{st}$ the (ordinary) category whose objects are cyclic $A_\infty$-categories of degree $d$ and whose morphisms are cyclic strict $A_\infty$-functors. 
\begin{df}\label{ColV}
    We say 
    $$V_B:=(V_{ij},\langle\_,\_\rangle_{ij})_{i,j\in B}$$
    is a \emph{collection of cyclic chain complexes} if it is a cyclic $A_\infty$-category $\mathcal{C}$ with $m_\mathcal{C}=0$ and $B=Ob(\mathcal{C})$. 
\end{df}

\subsubsection{Shifted Poisson Structure} 
As before, given a collection of cyclic chain complexes $V_B$ we have that $\big(Cyc^*_+(V_B)[-1],d\big)$ is a cochain complex.
Recall that we adopted the notation 
$$(a_1a_2\cdots a_n)\in Cyc^*_+({V_B})[-1],$$
where by definition $a_i\in Hom(\lambda,\mu)^\vee[-1]$ and $a_{i+1}\in Hom(\mu,\nu)^\vee[-1]$ for $\lambda,\mu,\nu\in B$.
\begin{cons}
   
 We define
 $$\{\_,\_\}:Cyc^*_+({V_B})[-1]\otimes Cyc^*_+({V_B})[-1]\rightarrow Cyc^*_+({V_B})[-1]$$
 by 
 \begin{equation}\label{bracket}
 \{(a_1a_2\cdots a_n),(b_1b_2\cdots b_m\}= \sum_{i=1}^m\sum_{j=1}^n\langle a_j,b_i\rangle^{-1}(-1)^*(a_{i+1}\cdots a_na_1\cdots a_{i-1}b_{j+1}\cdots b_mb_1\cdots b_{j-1})
 \end{equation}
for $n+m>2$ and zero otherwise, $(-1)^*$ being determined by the Koszul rule. 
Here 
\begin{equation*}
   \langle a_j,b_i\rangle^{-1}=
    \begin{cases}
      \langle a_j,b_i\rangle_{\lambda\mu}^{-1}, & \text{if}\ a_j\in Hom(\lambda,\mu)^\vee[-1]\ \text{and}\ b_i\in Hom(\mu,\lambda)^\vee[-1]\ \text{for some $\mu,\lambda\in Ob(\mathcal{C})$} \\
      0, & \text{otherwise}
    \end{cases}
      \end{equation*}
      \end{cons}
\begin{lemma}
   Let $V_B$ be a collection of cyclic chain complexes of degree $d$. Then 
   $$(Cyc^*_+(V_B)[-1],d,\{\_,\_\})$$
   is a $(d-2)$-shifted dg Lie algebra. 
   \end{lemma}
\begin{proof}
The fact that this is a shifted Lie algebra is proven in \cite{Che10}, Definition-Lemma 2.
For the algebra case see \cite{Ha07}, proposition 2.8 where a notation closer to here is used.
The fact that the differential is compatible with the Lie structure follows from the fact that the cyclic structure is by definition closed with respect to this differential.
\end{proof}

\begin{lemma}\label{Ham1}

   A cyclic $A_\infty$-category structures on $V_B$ induces a Maurer-Cartan elements of $(Cyc^*_+(V_B)[-1],d,\{\_,\_\})$ of degree $(3-d)$.
    \end{lemma}
   
\begin{proof}
This is a standard result, see e.g. 4.4.3 of \cite{da24}.
Recall that we can describe a cyclic $A_\infty$-category structure on $V_B$ as a degree $-1$ 
coderivation on $BarV_B$ that squares to zero and commutes with $d$.
We get an induced differential $m_\mathcal{C}$ on $Cyc^*_+(V_B)$ that squares to zero and commutes with the internal differential $d$.
This fact then implies that $I_\mathcal{C}=\langle m_\mathcal{C}\_,\_\rangle$ satisfies 
$$dI_\mathcal{C}+\frac{1}{2}\{I_\mathcal{C},I_\mathcal{C}\}=0,$$
i.e. it is a Maurer-Cartan element.
\end{proof}
 \begin{remark}
We call such Maurer Cartan elements also classical interaction term or $A_\infty$-hamiltonian.
    \end{remark}
 We recall that for $V$ a chain complex we use the notation $SymV:=\bigoplus_{n=0}V^{\otimes n}/S_n.$ 
 It is a standard fact that for $\mathfrak{g}$ a dg n-shifted Lie algebra $Sym(\mathfrak{g})$ 
 becomes an n-shifted Poisson dg algebra by extending the bracket according to the Leibniz rule and the differential as a derivation.
 Thus from the previous lemma it follows:
\begin{theorem}\label{fP}
    Let $V_B$ be a collection of cyclic chain complexes of degree $d$.
    Then `the observables of the free\footnote{a name motivated by physics, where a theory coming from a quadratic action functional is referred to as free.} theory' 
    $$\mathcal{F}^{fr}(V_B):=Sym(Cyc^*_+(V_B)[-1])$$
    are a $(d-2)$-shifted Poisson dg-algebra.

\end{theorem}

\begin{df}\label{F_cl} 
Given a cyclic $A_\infty$ category $\mathcal{C}$ of degree $d$  denote by
    $$\mathcal{F}^{cl}(\mathcal{C}):=\mathcal{F}^{fr}(V_{Ob\mathcal{C}})^{tw}$$
    the $(d-2)$-shifted Poisson dg-algebra obtained by twisting $\mathcal{F}^{fr}(V_{Ob\mathcal{C}})$ by the Maurer-Cartan element
    determined by cyclic $A_\infty$-structure through lemma \ref{Ham1}.
    
\end{df}
Here $cl$ stands for classical.\footnote{ a natural name from the physics perspective would also be interactive.}
\begin{lemma}\label{cov}
We find a functor
$$(cyc_dA_\infty cat)_{st}\rightarrow   (Poiss_{d-2})^{op},$$
$$(F:(\mathcal{C},\langle\_,\_\rangle_\mathcal{C})\rightarrow (\mathcal{D},\langle\_,\_\rangle_\mathcal{D}))\mapsto(F^*:\mathcal{F}^{cl}(\mathcal{D})\rightarrow \mathcal{F}^{cl}(\mathcal{C}))$$
\end{lemma}

\begin{proof}
Since we assume that $F$ is cyclic we have commutativity of following diagram.
$$\begin{tikzcd}
\hspace*{-1.5cm}
    Hom_\mathcal{D}(F(i),F(j))^\vee\otimes Hom_\mathcal{D}(F(j),F(i))^\vee[-2d]\arrow{d}{F^*_{i,j}\otimes F^*_{j,i}}&Hom_\mathcal{D}(F(i),F(j))\otimes Hom_\mathcal{D}(F(j),F(i))\arrow{l}\arrow{r}{\langle\_,\_\rangle_\mathcal{D}}&k[-d]\\
    Hom_\mathcal{C}(i,j)^\vee\otimes Hom_\mathcal{C}(j,i)^\vee[-2d]&Hom_\mathcal{C}(i,j)\otimes Hom_\mathcal{C}(j,i)\arrow{l}\arrow[swap]{ru}{\langle\_,\_\rangle_\mathcal{C}}\arrow{u}{F_*^{ij}\otimes F_*^{ji}}
\end{tikzcd}$$
Here the horizontal two maps are defined in the standard way using the respective pairings, see section 3.1 of \cite{GGHZ21}.
Since we assume that these are non-degenerate we can invert the vertical maps and ponder commutativity of following diagram.
$$\begin{tikzcd}
\hspace*{-1.5cm}
    Hom_\mathcal{D}(F(i),F(j))^\vee\otimes Hom_\mathcal{D}(F(j),F(i))^\vee[-2d]\arrow{d}{F^*_{i,j}\otimes F^*_{j,i}}\arrow{r}&Hom_\mathcal{D}(F(i),F(j))\otimes Hom_\mathcal{D}(F(j),F(i))\arrow{r}{\langle\_,\_\rangle_\mathcal{D}}&k[-d]\\
    Hom_\mathcal{C}(i,j)^\vee\otimes Hom_\mathcal{C}(j,i)^\vee[-2d]\arrow{r}&Hom_\mathcal{C}(i,j)\otimes Hom_\mathcal{C}(j,i)\arrow[swap]{ru}{\langle\_,\_\rangle_\mathcal{C}}\arrow{u}{F_*^{ij}\otimes F_*^{ji}},
\end{tikzcd}$$
where now the horizontal maps are the inverses of the previous vertical maps.
Indeed, this diagram is also commutative, which follows by injectivity of the maps $F_*^{ji}$ and commutativity of the previous diagram.
This implies that for all $F(i),F(j)\in Ob(\mathcal{D})$ and $a\in Hom_\mathcal{D}(F(i),F(j))^\vee$ and $b\in Hom_\mathcal{D}(F(j),F(i))^\vee$ we have 
$\langle F^*_{ij}(a),F^*_{ji}(b)\rangle_{\mathcal{C}}^{-1}=\langle a,b\rangle_{\mathcal{D}}^{-1}.$
As we further assume that  $F$ is strict it is now straightforward to verify that the induced map 
$F^*:\mathcal{F}^{cl}(\mathcal{D})\rightarrow \mathcal{F}^{cl}(\mathcal{C})$ respects the shifted Poisson algebra structures, using defintion \eqref{bracket}.   
\end{proof}

\begin{lemma}
Given a cyclic strict $A_\infty$-morphism $F:\mathcal{C}\rightarrow \mathcal{D}$
we have that 
$$F^*(I_\mathcal{D})=I_\mathcal{C}$$
for the associated cyclic $A_\infty$ hamiltonians.
\end{lemma}
\begin{proof}
This is well known, see e.g. proposition 4.6 of \cite{da24}.
Concretely this can be verified using the explicit characterizations of cyclic $A_\infty$-functors,
equation \eqref{nfn} and equation \eqref{F_inf_rel}.
\end{proof}
\begin{remark}\label{befu}
    In the case of a cyclic $A_\infty$-algebra definition 2.7 and equation 2.5 of \cite{Ha07}
    provide an alternative description of the shifted Poisson bracket. Using this alternative 
    description lemma 5.5 of \cite{HaLa08} implies that the assignment of the shifted Poisson
    algebra $F^{cl}(A)$ to a cyclic $A_\infty$ algebra $A$ is also functorial with respect to 
    cyclic $A_\infty$-isomorphisms, not just strict ones. The authors use the so called non-commutative calculus for doing so.
    See section 4.5 of \cite{da24} for a sketch of such a non-commutative calculus for the category case. 
    We imagine that using this one should be able to prove a similar functoriality result also for categories.
    See also our remark \ref{caf} for a comparison of nomenclature used. 
    \end{remark}
    \begin{remark}\label{dar}
Similarly lemma 5.5 of \cite{HaLa08} implies that a cyclic $A_\infty$-isomorphism sends associated
$A_\infty$-hamiltonian to associated $A_\infty$-hamiltonian. We imagine that analogous comments to
above remark apply.
\end{remark}

Note that this discussion carries over to the full subcategory of $(cyc_dA_\infty cat)_{st}$ on
collections of cyclic chain complexes, i.e. explains functoriality of theorem \ref{fP}.
  
\subsubsection{Beilinson-Drinfeld Structure}
\begin{df}
    Given a collection of cyclic chain complexes $V_B$ we extend the previous dg shifted Lie
    algebra structure to
    \begin{equation}
    \label{ce}Cyc^*(V_B):=Cyc^*_+(V_B)\Pi (\prod_{\lambda\in B}\nu_\lambda k)[1], 
    \end{equation}
    where the vector space $\prod_{\lambda\in B}\nu_\lambda k$ spanned by the $\nu_\lambda$ is central and we further redefine the bracket on words of length one by 
    \begin{equation*}
    \{(a),(b)\}=
    \begin{cases}
      \langle a,b\rangle^{-1}\nu_\lambda, & \text{if}\ a\in Hom(\lambda,\lambda)^\vee[-1]\ \text{and}\ b_i\in Hom(\lambda,\lambda)^\vee[-1]\ \text{for some $\lambda\in Ob(\mathcal{C})$} \\
      0, & \text{otherwise.}
    \end{cases}
 \end{equation*}

\end{df}
\begin{remark}
   The image of the bracket of words of length one can contain infinite sums of $\nu_\lambda$ over
   different $\lambda$. Our definition is well defined since we are using the direct product over
   $B$ in the central extension. 
\end{remark}
\begin{cons}\label{nabla}
Given $V_B$ we define 
$$\nabla:Cyc^*(V_B)[-1]\rightarrow \big(Cyc^*(V_B)[-1]\otimes Cyc^*(V_B)[-1]\big)_{S_2}$$
by 
$$\nabla(a_1a_2\cdots a_n)= \sum_{1\leq i<j\leq n}\langle a_i,a_j\rangle^{-1}(-1)^*(a_{i+1}\cdots  a_{j-1})\otimes(a_{j+1}\cdots a_na_1\cdots a_{i-1}).$$
Here 
$$\langle a_i,a_j\rangle^{-1}\in k$$
is defined as before.
We define the value of the summand where $j=i+1$ and if $a_i\in Hom(\beta,\lambda)^\vee[-1]$ and 
$a_j=a_{i+1}\in Hom(\lambda,\beta)^\vee[-1]$ to be 
$$\langle a_i,a_{i+1}\rangle^{-1}_{\beta\lambda}\ \nu_\lambda\otimes(a_{i+2}\cdots a_na_1\cdots a_{i-1}),$$
otherwise zero. If additionally $n=2$ we define the value by 
$$\langle a_i,a_j\rangle^{-1}_{\mu\lambda}\ \nu_\lambda\otimes\nu_\mu$$
if $a_i\in Hom(\mu,\lambda)^\vee[-1]$ and $a_j=a_{i+1}\in Hom(\lambda,\mu)^\vee[-1]$, again
otherwise zero. 
Additionally $\nabla$ is zero on the vector space spanned by the $\nu$'s.
\end{cons}
We refer to definition 2.11 of \cite{Ha07} for the notion of an involutive bi Lie algebra.
\begin{lemma}\label{ibl}
Given $V_B$ a collection of cyclic chain complexes of degree $d$ we have that
    $(Cyc^*(V_B),d,\{\_,\_\},\nabla)$ is an involutive bi Lie dg-algebra of degree $(d-2)$.
\end{lemma}
\begin{proof}
    This is proved in theorem 7 of \cite{Che10}. Note however that the author does not consider
    the central extension. One can check directly that the result remains true. For the one object
    case this is also explained in section 3.2 of \cite{GGHZ21}, where this central extension is included.
\end{proof}

From now on let us \textbf{restrict to odd $d$}: 
Introducing $\gamma$, a formal variable of cohomological degree $6-2d$, we denote 
$$\mathcal{F}^{pq}(V_B):=Sym\big(Cyc^*(V_B)[d-4]\big) \llbracket\gamma\rrbracket[3-d]$$
where $V_B$ is a collection of cyclic chain complexes of degree $d$. 
Denote by $\delta$ the Chevalley-Eilenberg differential extended to $\mathcal{F}^{pq}(V_B)$ from the Lie algebra chain
complex of the odd Lie algebra $Cyc^*(V_B)$. 
Denote by $d$ and $\nabla$ the differentials extended
from $Cyc^*(V_B)$ as derivations to $\mathcal{F}^{pq}(V_B)$ and by $\{\_,\_\}$
the shifted Poisson bracket extended from the odd Lie bracket on $Cyc^*(V_B)$.
\begin{theorem}\label{nc_pq} 
Let $V_B$ be a collection of cyclic chain complexes of degree $d$. Then 
    $$\Big(\mathcal{F}^{pq}(V_B) ,d+\nabla+\gamma\delta,\{\_,\_\}\Big)$$
    is a $(d-2)$ twisted Beilinson-Drinfeld algebra over $k\llbracket\gamma\rrbracket$. $\mathcal{F}^{pq}(V_B)$ has multiplication of even degree $(d-3)$.
\end{theorem}
The superscript $pq$ stands for prequantum.\footnote{A better name should have been derived from
`quantization of free theory', which is what $\mathcal{F}^{pq}(V)$ is.}
\begin{proof}
     The result follows as in lemma 3.4 of \cite{Ha07} from our lemma \ref{ibl} and accounting for
     the shifts made. 
     Note that these shifts and the degree of $\gamma$ precisely give the correct
     degrees required in our definition of an $r$-twisted BD algebra, here for $s=(d-2)$.
\end{proof}
Given a strict cyclic $A_\infty$-morphism $F: V_B\rightarrow V_C$ we define a map 
$$F^*: Cyc^*(V_C)\rightarrow Cyc^*(V_B)$$
by setting 
$$F^*(\nu_\lambda)=\sum_{\mu\in F^{-1}(\lambda)}\nu_{\mu}.$$
One can check directly that this is compatible with the extension of the odd Lie bracket. 
Furthermore, by arguments analogous to the discussion of lemma \ref{cov}, it follows that $F^*$ intertwines the co Lie brackets.
We denote by the same letter the extension of this map to 
    $$F^*: \mathcal{F}^{pq}(V_C)\rightarrow \mathcal{F}^{pq}(V_B),$$
   which is a map of BD algebras over $k\llbracket\gamma\rrbracket$. As an upshot we find
\begin{lemma}
There is a functor from the full subcategory of collections of cyclic chain complexes within the category  
$(cyc_dA_\infty Cat)_{st}$ to the category of $(d-2)$ twisted BD algebras by the assignment 
$$\mathcal{F}^{pq}:\ (F: V_B\rightarrow V_C)\mapsto \big(F^*: \mathcal{F}^{pq}(V_C)\rightarrow \mathcal{F}^{pq}(V_B)\big).$$
 
\end{lemma}
\begin{df}\label{nc_dq}
Define  $$p:  \mathcal{F}^{pq}(V_B)\rightarrow Sym^1(Cyc^*_+(V_B)[d-4])[3-d]\cong Cyc^*_+(V_B)[-1]\rightarrow \mathcal{F}^{fr}(V_B),$$
where the first map is the projection on the `$Sym^1$, $\gamma=0$ and $\nu=0$ part'.
The last map is just the inclusion.
\end{df}
 It is straightforward to check that $p$ is a dequantization map in the sense of definition \ref{dq}. 
 For example we have $p(\gamma\delta)=0=p(\nabla).$ 
\begin{df}\label{qcat}
Given a collection of cyclic chain complexes $V_B$ and $I\in MCE (\mathcal{F}^{fr}(V_B))$, we say that 
$$I^q\in MCE (\mathcal{F}^{pq}(V_B))$$
of degree $(3-d)$ is a \emph{quantization of $I$} if $p(I^q)=I.$ 
\end{df}
\begin{remark}
Note that more generally one may ask a quantization to be a differential $d^q$ on the underlying
shifted Poisson algebra of $\mathcal{F}^{pq}(V)$ satisfying the BD relation and such $p\circ d^q=d^{cl}\circ p.$
Thus we are more restrictive, imposing $d^q=d+\nabla+\gamma\delta+\{I^q,\_\}$. 
We ask that a quantization $I^q$ of the interacting theory $\mathcal{F}^{cl}$ is compatible with the standard quanization of the free theory, i.e. with $\mathcal{F}^{pq}(V)$.
\end{remark}
\begin{remark}
Further we restrict to the case where $I^q$ has zero constant terms, that is its components in the ground field and the vector space spanned by the $\nu$'s are zero.
\end{remark}
Given $F:V_B\rightarrow V_C$ a cyclic strict $A_\infty$-morphism between collections of cyclic chain complexes and an $I^q\in MCE (\mathcal{F}^{pq}(V_C))$ then 
$$F^*I^q\in MCE (\mathcal{F}^{pq}(V_B)).$$

Indeed this follows as we saw that $F^*$ is compatible with the odd Lie bracket,
thus also with the induced Chevalley Eilenberg differential, as well as the odd co Lie bracket.

Thus we can define a category which we denote 
$$(\text{q-cyc}_dA_\infty cat)_{st},$$
whose objects are pairs of a cyclic $A_\infty$-category $\mathcal{C}$ together with an
$I^q\in MCE (\mathcal{F}^{pq}(V_{Ob\mathcal{C}}))$ such that $p(I^q)=I_\mathcal{C}$.
Its morphisms are cyclic strict $A_\infty$-morphisms $F:\mathcal{C}\rightarrow \mathcal{D}$ such that $F^*I^q_\mathcal{D}=I^q_\mathcal{C}$. 
Note that there is a forgetful functor to $(cyc_dA_\infty Cat)_{st}$ induced by the dequantization map $p$.

\begin{remark}
 In the literature such elements $I^q$ have also been referred to as quantum $A_\infty$-categories. 
\end{remark}
\begin{df}

Given an object of $(\text{q-cyc}_dA_\infty cat)_{st}$ define 
$$\begin{tikzcd}\label{F_q}
\mathcal{F}^q(\mathcal{C}):=\mathcal{F}^{pq}(V_B)^{tw}\end{tikzcd}$$
as the $(d-2)$ twisted BD algebra obtained by twisting $\mathcal{F}^{pq}(V_B)$ by the MC element $I^q$.
The assignment \ref{F_q} defines a functor 
$$\mathcal{F}^q:\ (\text{q-cyc}_dA_\infty cat)_{st}\rightarrow (\text{BD}^{(d-2)tw})^{op},$$
where on morphisms we use $F^*$.
\end{df}
\begin{remark}
Note that this notation for $\mathcal{F}^q(\mathcal{C})$ is ambiguous and we need to have the choice of an $I^q$ in mind.
\end{remark}

In this work we did not examine how the discussion so far can be refined when considering homotopy equivalent solutions to the quantum master equation.
Compare section 6 of \cite{GGHZ21}.

\subsubsection{Essentially finite cyclic $A_\infty$-categories}
Recall that given a cyclic $A_\infty$-category $\mathcal{C}$ we defined the central extension of our previous shifted Lie algebra by 
$$Cyc^*(V_B):=Cyc^*_+(V_B)\times (\prod_{\lambda\in B}\nu_\lambda k[1]).$$
As we will see in section \ref{LQT_r}, under the `quantized' LQT map at level $N$, each $\nu_\lambda$ will be send to the actual number $N$.
However, if we allow infinite sums of $\nu_\lambda$'s this would be ill-defined.
Thus we restrict ourselves as follows:
\begin{df}\label{essfin}
We call a cyclic unital $A_\infty$-category $\mathcal{C}$ that admits a finite skeleton $Sk\mathcal{C}$, that is a subcategory with finitely many objects whose inclusion defines an equivalence, 
an \emph{essentially finite cyclic $A_\infty$-category}.
\end{df} 
\begin{remark}
Note that for an essentially finite cyclic $A_\infty$-category the inclusion 
$$Sk\mathcal{C}\rightarrow\mathcal{C}$$
is strict and cyclic.
\end{remark}

We define a category, denoted
$$(\text{e.f cyc}_dA_\infty cat)_{st},$$
as the full subcategory of $(cyc_dA_\infty cat)_{st}$ given by essentially finite cyclic $A_\infty$-categories.

In practise we will work with a slightly more complicated, though equivalent category:
\begin{df}
We denote by 
 $$(\text{e.f}^{+} \text{cyc}_dA_\infty cat)_{st},$$
 the category  whose objects are essentially finite cyclic unital $A_\infty$-categories together with the choice of a finite skeleton and together with the inclusion 
$$Sk\mathcal{C}\rightarrow\mathcal{C},$$
and whose morphisms are diagrams as follows 
    \begin{equation}\label{fwde}
    \begin{tikzcd}
   & \mathcal{C}\arrow{r}{F} &\mathcal{D}&\\
   Sk\mathcal{C}\arrow{ur}{\simeq}&  &&Sk\mathcal{D}\arrow{ul}[swap]{\simeq},
\end{tikzcd}
\end{equation}
where $F$ is a strict cyclic $A_\infty$-functor.
\end{df}
Note that the obvious functor 
\begin{equation}\label{obvFun}
  (\text{e.f}^{+} \text{cyc}_dA_\infty cat)_{st}\rightarrow (\text{e.f cyc}_dA_\infty cat)_{st}  
\end{equation}
is an equivalence, which is why we will treat the two appearing categories as interchangeable. We will later see \eqref{undf} that different choices of an inverse of \eqref{obvFun}, that is different choices of a finite skeleta for each essentially finite category, are related by natural quasi-isomorphism, once we localized appropriately.

Let us denote by $\mathcal{D}Poiss_{(d-2)}$ the localization of the category of 
$(d-2)$ shifted Poisson algebras with respect to quasi-ismomorphisms.
We (re)define previous functors:
\begin{lemma}\label{nattr1}
    Sending $(\mathcal{C},Sk\mathcal{C})$ an essentially finite cyclic $A_\infty$ category of degree $d$ to $\mathcal{F}^{cl}(Sk\mathcal{C})$ defines a functor 
    $$\mathcal{F}^{cl}_{fin}:\ (\text{e.f cyc}_dA_\infty cat)_{st}\rightarrow (\mathcal{D}Poiss_{(d-2)})^{op}.$$ 
    \end{lemma}
This follows directly from definition \ref{fwde} and \eqref{fff} below.
Further by \eqref{undf} we will see that this functor is independent of the choice of skeleton, up to isomorphism.
To see these facts we first remark the standard result 
\begin{lemma}
    
Cyclic (co)homology is invariant under $A_\infty$-equivalence.
\end{lemma}
\begin{proof} 
   As a sketch, by Proposition 7.3.6 and Lemma 7.4.1 of \cite{Cos07a} it follows that Hochschild (co)homology is invariant under $A_\infty$-equivalence.
   By arguments as in Loday, corollary 2.2.4, it follows that also cyclic (co)homology is invariant.
\end{proof}
 Thus given $(\mathcal{C},Sk\mathcal{C})$ an essentially finite cyclic $A_\infty$-category we get a map of shifted Poisson algebras
    \begin{equation}\label{fff}
        \mathcal{F}^{cl}(\mathcal{C})\rightarrow \mathcal{F}^{cl}(Sk\mathcal{C}),
\end{equation}
    whose underlying map of chain complexes is a quasi-isomorphism.
    This follows from the lemma before, as the inclusion $Sk\mathcal{C}\rightarrow\mathcal{C}$ is in particular an $A_\infty$-equivalence.
    
    Furthermore \eqref{fff} defines a natural quasi-isomorphism as follows 

$$\begin{tikzcd}
(\text{e.f cyc}
    _dA_\infty cat)_{st} \arrow[rr, "\mathcal{F}^{cl}_{fin}", bend left=25, ""{name=U, below}]
\arrow[rd," forget ", bend right=10]
&&\text{$(\mathcal{D}Poiss_{(d-2)})^{op}$}
\\
&(cyc_dA_\infty cat)_{st}\arrow[ur, "\mathcal{F}^{cl}", bend right=10]\arrow[Rightarrow, to=U, ""]&
\end{tikzcd}$$
 
    which we remark to justify our redefinition.
Regarding the choice of finite skeleton of a cyclic $A_\infty$-category we remark that
given a cyclic $A_\infty$-category $\mathcal{C}$ with two choices of finite skeleta $(\mathcal{C},Sk\mathcal{C}_1)$, $(\mathcal{C},Sk\mathcal{C}_2)$ we get a zig-zag of shifted Poisson algebras
   \begin{equation}\label{undf}
        \begin{tikzcd}
        &\mathcal{F}^{cl}(\mathcal{C})\arrow{dl}[swap]{\simeq}\arrow{dr}{\simeq}&\\
        \mathcal{F}^{cl}(Sk\mathcal{C}_1)&&\mathcal{F}^{cl}(Sk\mathcal{C}_2),
    \end{tikzcd}
    \end{equation}
    where the underlying maps of chain complexes are quasi-isomorphisms. 
    
    We make similar (re)definitions for the quantized context: 
 \begin{df}\label{qesfq}
 Let $\mathcal{C}$ be an essentially finite cyclic $A_\infty$-category.
 Denote by $I$ the  $A_\infty$-hamiltonian associated to $\mathcal{C}$.
 We say that  
 $$ I^q\in MCE(\mathcal{F}^{pq}(V_{Ob{\mathcal{C}}}))$$
 is \emph{a quantization of $\mathcal{C}$}
if $$ p(I^q)=I,$$ recalling the dequantization maps p from definition \ref{nc_dq}.
\end{df}
\begin{df}
   We define a category, denoted  
   $$(\text{q.e.f. cyc}_dA_\infty cat)_{st}.$$
    Its objects are quantizations of essentially finite cyclic $A_\infty$-categories.
    The morphisms are given by morphisms of the underlying  cyclic $A_\infty$-categories  
    \begin{equation}
        F:\ \mathcal{C}\rightarrow \mathcal{D}
\end{equation}
such that additionally  $F^*I^q=J^q$ for the respective quantizations.
\end{df}
In practise we will work with a slightly more complicated, though equivalent category:

\begin{df}
   We define a category, denoted  
   $$(\text{q.e.f}^+ \text{cyc}_dA_\infty cat)_{st}.$$
    An objects of this category is given by \begin{itemize}
        \item an essentially finite cyclic $A_\infty$-category
        \item together with the choice of a finite skeleton and with the data of inclusion 
$$\Phi: Sk\mathcal{C}\rightarrow\mathcal{C},$$
\item
denoting by $(I_f,I)$ the pair of $A_\infty$-hamiltonian associated to $Sk\mathcal{C}$ and the one associated to $\mathcal{C}$, we further ask for a pair of 
 $$I^q_f\in MCE(\mathcal{F}^{pq}(V_{Ob{Sk\mathcal{C}}}))\ \text{and} \ I^q\in MCE(\mathcal{F}^{pq}(V_{Ob{\mathcal{C}}}))$$
such that $$p(I^q_f)=I_f\ \text{and}\ \ p(I^q)=I,$$ (recalling the dequantization maps p from definition \ref{nc_dq}) and we have that 
$$\Phi^*I^q=I^q_f.$$
\end{itemize}
The morphisms are given by morphisms in $(\text{e.f}^+ \text{cyc}_dA_\infty cat)_{st}$, that is diagrams like 
    \begin{equation}\label{tnc}
        \begin{tikzcd}
   & \mathcal{C}\arrow{r}{F} &\mathcal{D}&\\
   Sk\mathcal{C}\arrow{ur}{\simeq}&  &&Sk\mathcal{D}\arrow{ul}[swap]{\simeq},
\end{tikzcd}
\end{equation}
such that additionally  $F^*I^q=J^q$ for the respective quantizations.
\end{df}
Note that the obvious functor 
\begin{equation}\label{qobvFun}
  (\text{q.e.f}^{+} \text{cyc}_dA_\infty cat)_{st}\rightarrow (\text{q.e.f cyc}_dA_\infty cat)_{st}  
\end{equation}
is an equivalence, which is why we will treat the two appearing categories as interchangeable.

Note that given an element of $(\text{q.e.f}^{+} \text{cyc}_dA_\infty cat)_{st}$ the induced map 
\begin{equation}\label{BDqi}
\Phi^*:\mathcal{F}^q(\mathcal{C})\rightarrow \mathcal{F}^q(Sk\mathcal{C})
\end{equation}
of BD algebras is a quasi-isomorphism.
Indeed, consider the complete and exhaustive filtration of $\mathcal{F}^q(\mathcal{C})$ described in example 6.3 of \cite{GGHZ21}.
The map $\Phi^*$ is by definition a filtered map.
Then the claim follows from the from the Eilenberg-Moore comparison theorem (theorem 5.5.11 of \cite{wei94}) as on the first page of the spectral sequence associated to the filtrations the map $\Phi^*$ induces a quasi-isomorphism, as remarked in \eqref{fff}.

We denote by $\mathcal{D}\text{BD}^{(d-2)tw}$ the category obtained by localizing $\text{BD}^{(d-2)tw}$ with respect to quasi-isomorphisms.
Since \eqref{BDqi} is an quasi-isomorphism and by definition \eqref{tnc} we observe that
 sending a quantization of an essentially finite cyclic $A_\infty$-category to $\mathcal{F}^q(Sk\mathcal{C})$ defines a functor 
    \begin{equation}\label{redef1}
    \mathcal{F}^q_{fin}:\ (\text{q.e.f. cyc}_dA_\infty cat)_{st}\rightarrow (\mathcal{D}\text{BD}^{(d-2)tw})^{op}. 
    \end{equation}

This redefinition makes sense as follows: Given a quantization of an essentially finite cyclic $A_\infty$-category $(\mathcal{C},Sk\mathcal{C})$ we have by \ref{BDqi} a map of BD algebras
    $$\mathcal{F}^{q}(\mathcal{C})\rightarrow \mathcal{F}^{q}(Sk\mathcal{C}),$$
    whose underlying map of chain complexes is a quasi-isomorphism. This defines a natural quasi-isomorphism  
$$\begin{tikzcd}
(\text{q.e.f.cyc}_dA_\infty )_{st} \arrow[rr, "\mathcal{F}^{q}_{fin}", bend left=25, ""{name=U, below}]
\arrow[rd," forget ", bend right=10]
&&(\mathcal{D}\text{BD}^{(d-2)tw})^{op}
\\
&(\text{q.cyc}
    _dA_\infty cat)_{st}\arrow[ur, "\mathcal{F}^{q}", bend right=10]\arrow[Rightarrow, to=U, ""]&
\end{tikzcd}$$
Regarding the dependence on how a quantum $A_\infty$-category is essentially finitely presented we note:
Given a cyclic $A_\infty$-category $\mathcal{C}$ which is finitely presented in two ways $(\mathcal{C},Sk\mathcal{C}_1)$, $(\mathcal{C},Sk\mathcal{C}_2)$ and with respective quantizations that coincide on $\mathcal{C}$ we get a zig-zag of BD algebras
   \begin{equation}\label{undfq}
    \begin{tikzcd}
        &\mathcal{F}^{q}(\mathcal{C})\arrow{dl}[swap]{\simeq}\arrow{dr}{\simeq}&\\
        \mathcal{F}^{q}(Sk\mathcal{C}_1)&&\mathcal{F}^{q}(Sk\mathcal{C}_2),
    \end{tikzcd}
    \end{equation}
    where the underlying maps of chain complexes are quasi-isomorphisms.
    Thus the functor \eqref{redef1} does not depend on the choice of skeleton and compatible quantization, up to isomorphism.
\subsection{Commutative World}\label{cw}
The second class of algebraic objects that feature in our constructions are cyclic $L_\infty$-algebras.
We will first introduce those, then explain various constructions one can perform with them and finally explain a natural source of those.\par
Recall that given $V$ a chain complex $Sym_+V=\bigoplus_{n=1}V^{\otimes n}/S_n$ has the structure of a dg co-algebra.
\begin{df}
    An \emph{$L_\infty$-algebra} is a chain complex $(L,d)$ together with $l$, a co-derivation on $\big(Sym_+(L[1])$ of degree $-1$ that squares to zero and commutes with the differential induced by $d$.\end{df}
    We call the resulting chain complex  $$C_*^+(L):=\big(Sym_+(L[1]),d+l\big)$$ the (reduced) Chevalley-Eilenberg chains of $L$.
    We further adopt the notation ${C_*(L):=\big(Sym(L[1]),d+l\big)},$ the (non-reduced) Chevalley-Eilenberg chains.
 \begin{remark}
The definition of an $L_\infty$-algebra can be entangled as follows:
\begin{itemize}
    \item 
It is equivalently given by maps, for $n\geq1$
$$l^n:L^{\otimes n}\rightarrow L $$
of degree $n-2$ that satisfy certain relations, e.g. $l^1$ is just a differential. 
\item 
In fact we assume that $l^1=0$ since we can just absorb it into the differential of the underlying chain complex.
\end{itemize}
 \end{remark}  
 We define $$C^*(L):=Hom(C_*(L),k),$$ the Chevalley-Eilenberg cochains of the $L_\infty$-algebra $L$ and denote $C^*_+(L):=Hom(C_*^+(L),k).$
\begin{df}[$L_\infty$-morphism]
An $L_\infty$-morphism 
$L_1\rightarrow L_2$ is a morphism of dg co-algebras 
$$\phi: (C_*^+(L_1),d+l_1)\rightarrow (C_*^+(L_2),d+l_2).$$
\end{df}
\begin{remark}
 Again this can be entangled in saying that there are maps 
 $$\phi_n: L_1^{\otimes n}\rightarrow L_2$$
 of certain degree and satisfying relations as encoded in the fact that $\phi$ commuted with the differentials. 
 For instance $\phi_1$ is just a chain map. 
\end{remark}
\begin{df}
    An \emph{$L_\infty$-quasi-isomorphism} is an $L_\infty$-morphism such that $\phi_1$ is a quasi-isomorphism. 

 We say that an $L_\infty$-morphism is strict if $\phi_n=0$ for all $n\geq 2$.     
 \end{df}

\begin{df}
    A \emph{cyclic $L_\infty$-algebra} of dimension $d$ is an $L_\infty$-algebra $L$ together with a symmetric non-degenerate pairing 
    $$\langle\_,\_\rangle:L\otimes L\rightarrow k[-d]$$
    such that $S:=\langle d+l\_,\_\rangle\in C^*(L),$ i.e. such that it is symmetrically invariant. 
    Note that in this case $S$ is an element of degree $3-d$. We denote $I:=\langle (l)\_,\_\rangle.$
\end{df}
We limit ourselves to strict cyclic $L_\infty$-morphisms:
\begin{df}
    A \emph{cyclic strict $L_\infty$-morphism} is a strict $L_\infty$-morphism such that 
    $$\langle F_1(x),F_1(y)\rangle=\langle x,y\rangle.$$ 
\end{df}
\begin{remark}
 Note that a cyclic $L_\infty$-morphism is automatically injective.  
\end{remark}
\begin{remark}
 One should be able to give a more abstract definition of a cyclic $L_\infty$-algebra and $L_\infty$-morphisms, analogous to the $A_\infty$-case.
 Compare for instance section 5.2 of \cite{GG15}, in the context of $L_\infty$-spaces.
\end{remark}
\subsubsection{Shifted Poisson Structure}
A cyclic chain complex $V$ is a cyclic $L_\infty$-algebra for which $l=0$.
Given such denote 
$$\mathcal{O}^{fr}(V):=C^*_+(V).$$
By extending the inverse of the pairing 
$$\langle\_,\_\rangle^{-1}:V^\vee[-1]\otimes V^\vee[-1]\rightarrow k[d-2],$$
    according to the Leibniz rule to the symmetric algebra underlying the Lie algebra cochains we can define a $(d-2)$-shifted Poisson structure.
    The differential is compatible with the shifted Poisson bracket because $\langle\_,\_\rangle$ is closed with respect to the differential.
    Summarizing (for details see section 3.1 of \cite{GGHZ21}) we have
\begin{theorem}
  Let $(V,d,\langle\_,\_\rangle)$ be a cyclic chain complex of degree $d$. Then 
  $$\big(\mathcal{O}^{fr}(V),d,\{\_,\_\}\big)$$
  is a $(d-2)$-shifted Poisson dg-algebra.
  \end{theorem} 
  
\begin{lemma}\label{ham_l}
     A cyclic $L_\infty$-structures on a cyclic chain complex $V$ induces a Maurer-Cartan elements of $\mathcal{O}^{fr}(V)$ of degree $(3-d)$. 
  \end{lemma}
  \begin{proof}
      This is a standard result and follows as in the $A_\infty$-case.
      The Maurer-Cartan element is given by $I_L:=\langle l\_,\_\rangle.$
  \end{proof}
  \begin{df}
      Given $L$ a cyclic $L_\infty$-algebra with underlying cyclic chain complex of degree d denote by $\mathcal{O}^{cl}(L)$ the $(d-2)$-shifted Poisson algebra obtained by twisting $\mathcal{O}^{fr}(V)$ by the MCE determined by lemma \ref{ham_l}. 
  \end{df}
Denote by $(cyc_dL_\infty alg)_{st}$ the category whose objects are cyclic $L_\infty$-algebras of degree $d$ and whose morphisms are cyclic strict $L_\infty$-morphisms.
\begin{lemma}\label{nattr2}
    We find a functor 
    $$(cyc_dL_\infty Alg)_{st}\rightarrow (Poiss_{(d-2)})^{op}$$
    $$\big(\phi:(L_1\rightarrow L_2)\big)\mapsto \big(\phi^*:\mathcal{O}^{cl}(L_2)\rightarrow \mathcal{O}^{cl}(L_1)\big).$$ 
\end{lemma}
\begin{proof}
Indeed, this can be verified as in the $A_\infty$-case. See lemma \ref{cov}.    
\end{proof}

\subsubsection{Beilinson-Drinfeld Structure}
From now on let us restrict again to $d$ odd.
We introduce $\hbar$, a formal variable of cohomological degree $3-d$.
For $(V,d,\langle\_,\_\rangle)$ a cyclic chain complex denote 
$$\mathcal{O}^{pq}(V):=C^*(V)\llbracket\hbar\rrbracket.$$
E.g. following \cite{GGHZ21}, we can define a BD differential $d^q$ on $\mathcal{O}^{pq}(V)$ as 
    $$d^q=d+\hbar \Delta.$$
    Indeed, by imposing the BD-relation (\eqref{BD}) and $\hbar$-linearity we only need to specify the values of $d^q$ on $C^2(V)$, which leaves us to specify the value of $\Delta$ on $C^2(V).$
    It is given by  $\Delta(a\cdot b)=\langle a,b\rangle^{-1}$ for $a,b\in L^\vee[-1]$.
    This is compatible with defining $\{\_,\_\}$ by $\hbar$-linearly extending the bracket on $\mathcal{O}^{fr}(V)$.
    Summarizing we have
\begin{theorem}\label{c_pq}
    For $V$ a cyclic chain complex 
    $$\mathcal{O}^{pq}(V):=\big(C^*(V)\llbracket\hbar\rrbracket,d^q,\{\_,\_\}\big)$$
    is a $(d-2)$ twisted BD algebra over $k$.  
\end{theorem}
Again, analogously to the $A_\infty$-case it follows:
\begin{lemma}
     The assignement 
$$(\phi: V_1\rightarrow V_2)\mapsto (\phi^*:\mathcal{O}^{pq}(V_2)\rightarrow \mathcal{O}^{pq}(V_1))$$
defines a functor from the full subcategory of cyclic chain complexes of degree $d$ of 
$(cyc_dL_\infty alg)_{st}$ to the category of (d-2) twisted BD algebras.
\end{lemma}
\begin{df}\label{fc_dq} 
Define 
    $$p: \mathcal{O}^{pq}(V)\rightarrow \mathcal{O}^{fr}(V)$$ 
    by setting $\hbar$ and the $n=0$ factor of the Chevalley-Eilenberg cochains to zero. 
    It is easy to verify that this defines a dequantization map in the sense of \ref{dq}.
\end{df}
\begin{df} We say that an
    $I^q\in MCE(\mathcal{O}^{pq}(V))$ of degree $(3-d)$ is a \emph{quantization} of $I\in MCE(\mathcal{O}^{fr}(V))$ if $p(I^q)=I$.
\end{df}
If $\phi: V_1\rightarrow V_2$ is a strict cyclic $L_\infty$-morphism of cyclic chain complexes and if $I^q\in MCE(\mathcal{O}^{pq}(V_2))$ then $\phi^*I^q\in MCE(\mathcal{O}^{pq}(V_1))$.
This allows us to define a category, which we denote by $$(\text{q-cyc}_dL_\infty alg)_{st}$$ whose objects are pairs of a cyclic $L_\infty$-algebra $L$ with underlying chain complex V and $I^q\in MCE(\mathcal{O}^{pq}(V))$ such that $p(I^q)=I$.
Its morphisms are cyclic strict $L_\infty$-morphisms 
$$\phi:\ L_1\rightarrow L_2$$
such that $\phi^*I^q_2=I^q_1.$ 
\begin{df}\label{O_q}
    Given a cyclic $L_\infty$ algebra $L$ denote by $V$ its underlying chain complex and by $I\in MCE(\mathcal{O}^{fr}(V))$ its associated $L_\infty$-hamiltonian.
    Assume we are given a quantization $I^q\in MCE(\mathcal{O}^{pq}(V)),$ then we define
    $$\mathcal{O}^{q}(L):=\mathcal{O}^{pq}(V)^{tw}$$
    as the twist of $\mathcal{O}^{pq}(V)$ by the MCE $I^q.$
    This assignement defines a functor 
    $$\mathcal{O}^q: (\text{q-cyc}_dL_\infty Alg)_{st}\rightarrow (\text{BD}^{(d-2)tw})^{op}.$$
    \end{df}
\begin{remark}
    Note that the notation $\mathcal{O}^{q}(L)$ is ambiguous and we implicitly have to keep in mind an $I^q\in MCE(\mathcal{O}^{pq}(V))$.
\end{remark}

\subsubsection{Commutator $L_\infty$-algebras of $A_\infty$-categories}\label{CLS}
We recall the standard definition (e.g definition 2.8 of \cite{GGHZ21})
of the functor 
\begin{equation}\label{Cmt}
[\_,\_]:A_\infty Alg\rightarrow L_\infty Alg,\end{equation}
 called the commutator functor, which generalizes the construction of the commutator Lie algebra out of an associative algebra.

In this section we wish to generalize the domain of this functor to $A_\infty$-categories.
We do so by defining a functor 
$$\ A_\infty cat\rightarrow A_\infty Alg.$$
 $$(F:\mathcal{C}\rightarrow \mathcal{D})\mapsto (\tilde{F}:A_\mathcal{C}\rightarrow A_\mathcal{D}).$$
  On objects the functor is defined as follows:
 \begin{cons}\label{obj}
    Let $\mathcal{C}$ be an $A_\infty$-category.
    Recall that this defines the datum of a degree $-1$ coderivation $m_\mathcal{C}$ on its bar complex $Bar \mathcal{C}$ that squares to zero.
\begin{itemize}
     \item We set
 \begin{equation}\label{sqc}
    A_\mathcal{C}:=\bigoplus_{(\lambda,\mu)\in Ob(\mathcal{C})^{\times2}}Hom_\mathcal{C}(\lambda,\mu). 
    \end{equation}
\item We define an $A_\infty$-algebra structure on $A_\mathcal{C}$ as follows.
Define maps
$$ \begin{tikzcd}
m^n: (A_\mathcal{C}[1]^{\otimes n})\arrow{r}{\iota_n}& Bar \mathcal{C}\arrow{r}{m_\mathcal{C}}&Bar \mathcal{C}[1]\arrow{r}{\pi[1]}& A_\mathcal{C}[2]\end{tikzcd}.$$
 Here $\iota_n$ denotes the projection onto length $n$ tensor factors of the form in the bar complex - that is those with `matching boundary conditions', except for the first and last one.\footnote{See equation~\eqref{bar}.}
 Further 
 $$\pi:Bar \mathcal{C}\rightarrow A_\mathcal{C}[1]$$
 denotes projection to tensors of $Bar\mathcal{C}$ of length one.
 Its image is by definition $A_\mathcal{C}[1]$. 
 \item
 We extend the maps $m^n$ as a coderivation to $Bar(A_\mathcal{C})$, which we denote by $m_{A_\mathcal{C}}$ and which has degree $-1$.
 Thus it remains to show that $m_{A_\mathcal{C}}^2=0$ to prove that this defines an $A_\infty$-algebra structure on ${A_\mathcal{C}}$.
 \item We observe that the $\iota_n$ maps induce a chain map 
 $$\iota:Bar(A_\mathcal{C})\rightarrow Bar\mathcal{C}.$$
 Note that this map has a `canonical' complement $W$ to its kernel, given by tensors with`matching boundary conditions, except the first and last one'.
 By definition $\iota|_W$ is injective.
\item It follows by the definition of $m_{A_\mathcal{C}}$ that $\iota\circ m_{A_\mathcal{C}}=m_\mathcal{C}\circ\iota$.
Thus we have that 
$$\iota \circ m_{A_\mathcal{C}}^2=m_\mathcal{C}^2\circ\iota=0.$$
 It holds that $m_{A_\mathcal{C}}(W)\subseteq W.$ Thus since $\iota$ is injective on $W$, it follows that ${m_{A_\mathcal{C}}^2|_W=0}$. 
 However, by definition also $m_{A_\mathcal{C}}|_{ker(\iota)}=0$, thus also $m_{A_\mathcal{C}}^2|_{ker(\iota)}=0.$ 
 Since $W$ is complementary to ${ker(\iota)}$ this together implies  $$m_{A_\mathcal{C}}^2=0$$ in general.
 \end{itemize}
 \end{cons}
 \begin{df}
  Given $\mathcal{C}$ an $A_\infty$-category denote by 
  $\big(A_\mathcal{C},d+m_{A_\mathcal{C}}\big)$ the thus constructed $A_\infty$-algebra, where $d$ denotes the internal differential.
 \end{df}
 Next we define the  functor on morphisms:
 \begin{cons}\label{func}
  Let  $F: \mathcal{C}\rightarrow \mathcal{D}$ be a functor of $A_\infty$-categories which by definition gives a map of coalgebras
  $$Bar\mathcal{C}\rightarrow  Bar\mathcal{D}.$$
  \begin{itemize}
      \item 
  We define 
  $$\tilde{F}_*: Bar(A_\mathcal{C})\rightarrow Bar(A_\mathcal{D})$$
  by extending 
$$  \begin{tikzcd}    
\tilde{F}_*: Bar(A_\mathcal{C})\arrow{r}{\iota_\mathcal{C}}& Bar \mathcal{C}\arrow{r}{F_*}& Bar \mathcal{D} \arrow{r}{\pi_\mathcal{D}}& A_\mathcal{D}[1]\end{tikzcd}$$
  as a map of coalgebras. 
  It has degree zero. 
  We check that $\tilde{F}_*$ intertwines the coderivation $m_{A_\mathcal{C}}$ with $m_{A_\mathcal{D}}$:
  \item
We have a commutative diagram of chain complexes
 \begin{equation}\label{iota_func}
\begin{tikzcd}
    Bar \mathcal{C}\arrow{r}{F_*}  &Bar \mathcal{D}\\
    Bar(A_\mathcal{C})\arrow{u}{\iota_\mathcal{C}}\arrow{r}{\tilde{F}_*}&Bar(A_\mathcal{D})\arrow{u}{\iota_\mathcal{D}},
 \end{tikzcd}
 \end{equation}
 which follows by construction.
 \item
 Thus we have that 
 $$\iota_\mathcal{D}\circ\tilde{F}_*\circ m_{A_\mathcal{C}}={F}_*\circ\iota_\mathcal{C}\circ m_{A_\mathcal{C}}={F}_*\circ m_\mathcal{C}\circ\iota_\mathcal{C}=m_\mathcal{D}\circ{F}_*\circ\iota_\mathcal{C}=m_\mathcal{D}\circ\iota_\mathcal{D}\circ\tilde{F}_*=\iota_\mathcal{D}\circ m_{A_\mathcal{D}}\circ\tilde{F}_*,$$
 
 where we first used commutativity of above diagram, then that $\iota_\mathcal{C}$ intertwines the respective coderivations.
 Then we used that $F_*$ intertwines the coderivations, next again commutatvity of above diagram and lastly that $\iota_\mathcal{D}$ intertwines the respective coderivations.
 Thus
 $$ \iota_\mathcal{D}\circ\tilde{F}_*\circ m_{A_\mathcal{C}}=\iota_\mathcal{D}\circ m_{A_\mathcal{D}}\circ\tilde{F}_*.$$
\item
We have 
$$\tilde{F}_*\circ m_{A_\mathcal{C}}|_{ker(\iota_\mathcal{C})}=0=m_{A_\mathcal{D}}\circ\tilde{F}_*|_{ker(\iota_\mathcal{C})},$$
which follows by definition. 
On the other hand we have that 
$$\tilde{F}_*\circ m_{A_\mathcal{C}}(W_\mathcal{C})\subseteq W_\mathcal{D}\ \text{and}\ \ m_{A_\mathcal{D}}\circ\tilde{F}_*(W_\mathcal{C})\subseteq W_\mathcal{D}.$$  
Since $\iota_\mathcal{D}|_{W_\mathcal{D}}$ is injective we conclude from the previous line and the previous bullet point that 
$$m_{A_\mathcal{D}}\circ\tilde{F}_*=\tilde{F}_*\circ m_{A_\mathcal{C}}.$$
\item Unwinding the definitions one can also show that $\widetilde{G\circ F}=\tilde{G}\circ \tilde{F}.$
\end{itemize}
\end{cons}
\begin{lemma}\label{Commu}
Constructions \ref{obj} and \ref{func} allow us to define 
$$(\tilde\_): A_\infty cat\rightarrow A_\infty Alg.$$
 $$(F:\mathcal{C}\rightarrow \mathcal{D})\mapsto (\tilde{F}:A_\mathcal{C}\rightarrow A_\mathcal{D}).$$
\end{lemma}
\begin{remark}
Note that the chain map $\iota$ is not a map of coalgebras.
However, one may ponder whether it is induced from a morphism in a certain bi-module category, see section 2.2.11 of \cite{Lo13}.
\end{remark}

Let us make observations that we will need in the next section.
By restricting $\iota_\mathcal{C}$ to cyclic words we get an induced map:
\begin{equation}\label{iota}
\iota_\mathcal{C}: (Cyc_*^+(A_\mathcal{C})[1],d+m_{A_\mathcal{C}})\rightarrow (Cyc_*^+(\mathcal{C})[1],d+m_{\mathcal{C}}).
\end{equation}
Additionally to the functor \eqref{Cmt} we also get a map 
\begin{equation}\label{f}
(C_*^+(L),d+l)\rightarrow (Cyc_*^+(A)[1],d+m_{A}),
\end{equation}
where $L$ is the commutator $L_\infty$-algebra of $A$,
see equation 2.6 of \cite{GGHZ21}.
Both these maps are functorial with respect to $A_\infty$-functors.

\subsubsection{Cyclic commutator $L_\infty$-algebras}\label{ccla}
Given a cyclic $A_\infty$-category we would like to associate to it a cyclic $A_\infty$-algebra, generalizing the previous section.
The relatively strict notions we use mean that we restrict ourselves to essentially finite cyclic $A_\infty$-categories to do so.
\begin{cons}\label{assgn1}
    Let $\big((\mathcal{C},\langle\_,\_\rangle_{\mathcal{C}}),(Sk\mathcal{C},\langle\_,\_\rangle_{Sk\mathcal{C}})\big)$ be an essentially finite cyclic $A_\infty$-category.
    Then by the previous section ${A_{Sk\mathcal{C}}}$ is an $A_\infty$-algebra.
    Furthermore we define a symmetric non-degenerate chain map
    $$\langle\_,\_\rangle_{A_{Sk\mathcal{C}}}:\ A_{Sk\mathcal{C}}\otimes A_{Sk\mathcal{C}}\rightarrow k[-d]$$
    by 
    recalling that 
    $A_{Sk\mathcal{C}}=\bigoplus_{(\lambda,\mu)\in Ob(Sk\mathcal{C})^{\times2}}Hom_{Sk\mathcal{C}}(\lambda,\mu)$ and by linearly extending
 \begin{equation}\label{sodef}
    \langle x,y\rangle_{A_{Sk\mathcal{C}}}=
    \begin{cases}
      \langle x,y\rangle_{Sk\mathcal{C}}, & \text{if}\ x\in Hom_{Sk\mathcal{C}}(\lambda,\mu) \text{and}\ y \in Hom_{Sk\mathcal{C}}(\mu,\lambda)\ \text{for some}\  \mu,\lambda\in Sk\mathcal{C} \\
      0, & \text{otherwise.}
    \end{cases}
  \end{equation}
 Non-degeneracy follows by non-degeneracy of the original pairing and since the cardinality of $Ob(Sk\mathcal{C})$ is finite.
 By noting that 
 $$\langle m_n(\_,\cdots,\_),\_\rangle_{A_{Sk\mathcal{C}}}=\iota_\mathcal{C}^*(\langle m_n(\_,\cdots,\_),\_\rangle_{Sk\mathcal{C}})\in Cyc^*_+(A_{Sk\mathcal{C}}) $$
 we conclude that we indeed defined a cyclic $A_\infty$-algebra. 
\end{cons}
We explain how to associate a morphism of cyclic $A_\infty$-algebras to a strict cyclic functor of essentially finite $A_\infty$-categories:
\begin{cons}
 Given $F$ a strict cyclic $A_\infty$-functor of essentially finite $A_\infty$-categories
 \begin{equation}\label{haha}
        \begin{tikzcd}
   & \mathcal{C}\arrow{r}{F} &\mathcal{D}&\\
   Sk\mathcal{C}\arrow{ur}{\simeq}&  &&Sk\mathcal{D}\arrow{ul}[swap]{\simeq},
\end{tikzcd}
\end{equation}
we note that by defining $\widebar{\mathcal{D}}$ to be the full subcategory of the union of $Sk\mathcal{D}$ and the image of $Sk\mathcal{C}$ we obtain a commutative diagram

 \begin{equation}\label{hihi}
        \begin{tikzcd}
   & \mathcal{C}\arrow{rr}{F} &&\mathcal{D}&\\ 
   &&\widebar{\mathcal{D}}\arrow{ur} &&\\
   Sk\mathcal{C}\arrow{uur}{\simeq}\arrow{urr}&  &&&Sk\mathcal{D}\arrow{uul}[swap]{\simeq}\arrow{ull},
\end{tikzcd}
\end{equation}
 where all the arrows are strict cyclic $A_\infty$-functors.
 Further $\widebar{\mathcal{D}}$ also has only finitely many objects and the functor
 $$Sk\mathcal{D}\rightarrow \widebar{\mathcal{D}}$$
 is an equivalence.
 This is because the induced homology functor is full and faithful by definition.
 Further it is essentially surjective since every object of $\mathcal{D}$ is isomorphic to an object in $Sk\mathcal{D}$, thus in particular every object of $\widebar{\mathcal{D}}$ as a full subcategory of $\mathcal{D}$.
 Summarizing, given a strict cyclic functor of essentially finite cyclic $A_\infty$-categories as \eqref{hihi}, we obtain a roof
\begin{equation}\label{wbd}
    \begin{tikzcd}
&\widebar{\mathcal{D}} &\\
   Sk\mathcal{C}\arrow{ur}&  &Sk\mathcal{D}\arrow{ul}[swap]{\simeq},
\end{tikzcd}
\end{equation}
where all functors are strict and cyclic and injective on objects.
 \end{cons}
 Let us denote by $W$ the class of morphisms of cyclic $A_\infty$-algebras inducing isomorphisms
 on cyclic cohomology.\footnote{Since we are working over a characteristic zero field we could equivalently consider cyclic homology or hochschild homology.}
 We find
\begin{lemma}\label{iknm}
    The assignment \ref{assgn1} extends to a functor
\begin{equation}\label{cyccomm}
(\tilde\_)_{fin}:\ (\text{e.f. cyc}
    _dA_\infty cat)_{st}\rightarrow (cyc_dA_\infty Alg)_{st}[W^{-1}]
     \end{equation}
     sending equivalences to isomorphisms.  
\end{lemma}
In particular, see \eqref{ssvr}, this functor is up to isomorphism independent from the choice of a skeleton.
\begin{proof}
 Given a strict cyclic functor of cyclic $A_\infty$-categories with finitely many objects which is additionally injective on objects we obtain
 a strict cyclic $A_\infty$-algebra morphism of the induced cyclic $A_\infty$-algebras (from construction \ref{assgn1}), as can be deduced easily.
 Furthermore given a functor which is an equivalence, we already saw that it induces an isomorphism on cyclic cohomology.
 We will see later in the proof of theorem \ref{LQT_conj} that the cyclic cohomology of an $A_\infty$-category and its associated cyclic $A_\infty$-algebra are isomorphic.
 Thus from diagram \eqref{wbd} it follows that the assignment on morphisms underlying \eqref{cyccomm} is well defined.
 It remains to verify functoriality.
 Given 
$$\begin{tikzcd}
    \mathcal{C}\arrow{r}{G}&\mathcal{D}\arrow{r}{H}&\mathcal{E}\\
     Sk\mathcal{C}\arrow{u}{\simeq}&Sk\mathcal{D}\arrow{u}{\simeq}&Sk\mathcal{E}\arrow{u}{\simeq}
\end{tikzcd}$$
strict cyclic functors of essentially finite cyclic $A_\infty$-categories. Then we need to compare the images of the diagram

\begin{equation*}
    \begin{tikzcd}
&\widetilde{\mathcal{E}} &\\
   Sk\mathcal{C}\arrow{ur}{G\circ H}&  &Sk\mathcal{E}\arrow{ul}[swap]{\simeq},
\end{tikzcd}\end{equation*}
and the diagram
\begin{equation*}
    \begin{tikzcd}
&\widebar{\mathcal{D}} &&\widebar{\mathcal{E}}&\\
   Sk\mathcal{C}\arrow{ur}&  &Sk\mathcal{D}\arrow{ul}{\simeq}\arrow{ur}&&Sk\mathcal{E}\arrow{ul}[swap]{\simeq},
\end{tikzcd}\end{equation*}
under the assignment \ref{cyccomm}.
Let us denote by $\mathcal{E}^*$ the full subcategory of $\mathcal{E}$ on the union of $Sk\mathcal{E}$ and the images of $Sk\mathcal{D}$ and $Sk\mathcal{C}$.
Then the images of both diagrams before are, in the localized category, equal to 
\begin{equation*}
    \begin{tikzcd}
&\mathcal{E}^* &\\
   Sk\mathcal{C}\arrow{ur}{G\circ H}&  &Sk\mathcal{E}\arrow{ul}[swap]{\simeq},
\end{tikzcd}\end{equation*}
 which proves that they are equal, which proves functoriality
\end{proof}
\begin{remark}
 Further, diagram \eqref{hihi} induces a natural quasi-isomorphism between the functor \ref{cyccomm}, composed with forgetting the cyclic structure,
 and the functor induced by sending an essentially finite $A_\infty$-category $(\mathcal{C},Sk\mathcal{C})$ to $\mathcal{C}$ and then applying the functor \ref{Commu}.
\end{remark}
Lastly, given an cyclic $A_\infty$-category $\mathcal{C}$ which is finitely presented in two ways
$(\mathcal{C},Sk\mathcal{C}_1)$, $(\mathcal{C},Sk\mathcal{C}_2)$ we get a roof of cyclic $A_\infty$-algebras
   \begin{equation}\label{ssvr}
   \begin{tikzcd}
&A_{\widebar{\mathcal{C}}}\arrow{dl}\arrow{dr}&\\
    A_{Sk\mathcal{C}_1}&&A_{Sk\mathcal{C}_2},
    \end{tikzcd}
    \end{equation}
   where the arrows are in $W$, which tells us that the choice of finite presentation in lemma \ref{cyccomm} does not  matter.\par
We will see a bit later in definition \ref{vfnh} that we can also associate to a quantizaiton of an essentially finite cyclic $A_\infty$-category a quantization of a cyclic $A_\infty$-algebra.
If we denote by $U$ the class of morphisms inducing an isomorphism on the BD algebras associated to a cyclic $A_\infty$-algebra, see definition \ref{F_q}, then we find
\begin{lemma}
 The assignment from lemme \ref{iknm} extends to a functor
\begin{equation}\label{qcyccomm}
(\tilde\_)_{fin}:\ (\text{q.e.f. cyc}
    _dA_\infty cat)_{st}\rightarrow (q. cyc_dA_\infty Alg)_{st}[U^{-1}]
    \end{equation}
 sending equivalences of quantum $A_\infty$-categories to isomorphisms.  
\end{lemma}
Further this functor is  independent up to isomorphism from the choice of a finite presentation and a compatible quantization.
\begin{proof}
    Indeed, we can argue in exactly the same way as in the previous lemma.
    The only additional argument that we need is that if a map of the BD algebras considered here induces a quasi-isomorphism when applying the dequantization functor, see definition \ref{nc_dq},
    then that the map of BD algebras itself is already a quasi-isomorphism.
    This follows by a spectral sequence argument associated to a suitable filtration, compare e.g. the arguments around \ref{qc12}.
\end{proof}
    We finish this subsection by recalling the standard functors
   \begin{equation}\label{cAL}
       (cyc_dA_\infty Alg)_{st}\rightarrow  (cyc_dL_\infty Alg)_{st}
   \end{equation}
   and
   \begin{equation}\label{qcAL}
   (q.cyc_dA_\infty Alg)_{st}\rightarrow  (q.cyc_dL_\infty Alg)_{st}
    \end{equation}
which descend to the localized versions if one the right hand side we invert morphisms inducing a quasi-isomorphism on the associated shifted Poisson algebras (see definition \ref{nattr1}),
the class of those denoted $V$, respectively on the associated shifted BD-algebras (see definition \ref{O_q}), the class of those denoted $P$.
This is a standard result and follows by functoriality e.g of proposition 4.5 of \cite{GGHZ21}:
We are in characteristic zero which means that we can chose a chain level map inducing an inverse in homology to a given quasi-isomorphism.  

%% file: LQT.tex
\section{Loday-Quillen-Tsygan Map}
We denote by $M_N(k)$ $N\times N$-matrices with values in the ground field $k$. We recall the following standard construction:
%We define a  'Morita' functor, for every $N\geq1,$ $$M_N:A_\infty cat\rightarrow A_\infty cat.$$
\begin{df}
Let $A$ be an $A_\infty$-algebra. One can define another $A_\infty$-algebra $M_NA$ as follows:\begin{itemize}
    \item 
Set $$  M_NA=M_N(k)\otimes A.$$ We define  structure maps by multilinearly extending 
   $$m_n^{M_NA}(M_1u_1,\cdots ,M_nu_n):=(M_1\cdot\ldots\cdot M_n)\otimes m_n^A(u_1,\cdots,u_n).$$
    Using equation \ref{A_inf_re} it is straightforward to verify that this indeed defines an $A_\infty$-algebra.
  \item  
    Given $F:A_1\rightarrow A_2$ an $A_\infty$-morphism define a morphism
    $M_NF:M_NA_1\rightarrow M_NA_2$ by multilinearly extending 
    $$(M_NF)_n(u_1M_1,\cdots u_nM_n):=(M_1\cdot\ldots\cdot M_N)F_n(u_1,\cdots,u_n),$$
    where $F_n$ are the structure maps of $F$.
    Using equation \ref{F_inf_rel} it is easy to verify that this defines an $A_\infty$-morphism.
    \end{itemize}
\end{df} 
%\begin{remark}
 %   We note that for $\mathcal{C}$ essentially finite we have that $M_N\mathcal{C}$ is essentially finite as well, where the necessary map/functor is provided by functoriality of $M_N$. Verifying that the provided map is fully faithful is straightforward and also essential surjectivity can be checked easily. 
%\end{remark}
\begin{remark}\label{cycM}
    As another comment that we will need in the next section we remark that given $(A,\langle\_,\_\rangle^{A})$ a cyclic $A_\infty$-algebra we have that $(M_NA,\langle\_,\_\rangle^{M_NA})$ is a cyclic $A_\infty$-algebra as well. Here we define 
    $$\langle\_,\_\rangle_{\mu\lambda}^{M_NA}:M_N A\otimes M_n A\rightarrow k[d],$$
    by 
    $$\langle M_1u_1,M_2u_2\rangle^{M_NA}=tr(M_1M_2)\langle u_1,u_2\rangle.$$
    Using the cyclic invariance of the trace it is easy to verify that this defines a cyclic $A_\infty$-algebra, ie condition \ref{cycl} is satisfied.
\end{remark}
\subsection{for $A_\infty$-categories}
We recall the functor from section \ref{CLS} assigning to an $A_\infty$-category $\mathcal{C}$ an $A_\infty$-algebra $A_\mathcal{C}$, see \ref{sqc} for its explicit form.
\begin{df}\label{LQT_d}
Let $\mathcal{C}$ be an $A_\infty$-category. Then we define the map
$$\begin{tikzcd}
LQT_N: Sym( Cyc_+^*(\mathcal{C})[-1])\arrow{r}{\iota_\mathcal{C}^*}& Sym (Cyc_+^*(A_\mathcal{C})[-1])\arrow{r}{LQT_N}& C^*_+(\mathfrak{gl}_NA_\mathcal{C})
\end{tikzcd}$$
as the composition of the map induced by \ref{iota} and the standard LQT map applied to the $A_\infty$-algebra $A_\mathcal{C}$, see eg section 2.4 of \cite{GGHZ21}. In particular $\mathfrak{gl}_NA_\mathcal{C}$ denotes the commutator $L_\infty$-algebra of $M_NA_\mathcal{C}$, the $A_\infty$-algebra of $A_\mathcal{C}$ with values in $N\times N$-matrices.
\end{df}
\begin{theorem}\label{LQT_c}
The LQT map from definition \ref{LQT_d} is a map of dg-algebras, that is we have 
$$LQT_N\circ(d+m_\mathcal{C})=(d+l_N)\circ LQT_N.$$ Further this map is functorial in $\mathcal{C}$ under $A_\infty$-functors.
 \end{theorem}
 \begin{proof}
The fact that this is a map of dg-algebras follows by the results around remark \ref{iota} and the standard results on the LQT map, see eg 2.4 of \cite{GGHZ21}. Similarly it is functorial under $A_\infty$-functors implied by diagram \ref{iota_func} and the standard results on the LQT map, see eg again 2.4 of \cite{GGHZ21}.
\end{proof}
We have a commutative diagram of dg algebras     
$$\begin{tikzcd}
   Sym_+( Cyc_+^*(\mathcal{C})[-1])\arrow[dddrr,"LQT_N",swap]\arrow[dddr,"LQT_{N+1}", swap]\arrow[ddd]\arrow[dashed,dddrrr]&&&\\
   \\ \\
   (\cdots)\arrow[r,dashed]&C^*_+(\mathfrak{gl}_{N+1}A_\mathcal{C})\arrow{r}\arrow{r}&C^*_+(\mathfrak{gl}_NA_\mathcal{C})\arrow[dashed,r]&(\cdots)
\end{tikzcd}$$
Commutativity of the diagram follows from the same diagram for the one object LQT map, see equation 2.8 of \cite{GGHZ21}.
\begin{theorem}\label{LQT_conj}
     If $\mathcal{C}$ is a unital $A_\infty$-category then 
    $$Sym_+( Cyc_+^*(\mathcal{C})[-1])\simeq \varprojlim_{N\to\infty}  C^*_+(\mathfrak{gl}_NA_\mathcal{C}) $$

\end{theorem}
%\begin{theorem}[Definition: LQT Map]
    
%We denote by $$\begin{tikzcd}
%LQT_N: Sym( Cyc_+^*(\mathcal{C})[-1])\arrow{r}{tr^*}& Sym (Cyc_+^*(M_N\mathcal{C})[-1])\arrow{r}{f}& Sym (C^*_+(\mathfrak{gl}_NL_\mathcal{C}))\arrow{r}{mult.}& C^*_+(\mathfrak{gl}_NL_\mathcal{C})\end{tikzcd}$$ the composition of the map induced by trace map, the map induced by \ref{f} and lastly the multiplication map\footnote{which we for convenience set to be zero on $Sym^0(\cdots)$.}. By the results \ref{Mor} and \ref{f} 
%\end{theorem}
%$$\begin{tikzcd}
%A_\infty cat \arrow[rr, "\mathcal{F}^{cl}", bend left=25, ""{name=U, below}]
%\arrow[rd," com_N ", bend right=10]
%&&\textit{(d-2)-}PoissAlg
%\\
%&L\arrow[ur, "\mathcal{O}^{cl}", bend right=10]\arrow[Rightarrow, from=U, "LQT^{cl}"]&
%\end{tikzcd}$$

\begin{proof}
The one object LQT theorem applied to the unital $A_\infty$-algebra $A_\mathcal{C}$ implies that the second map in the definition of the LQT map for categories is a quasi-isomorphism for $N\rightarrow \infty$ (see thm 2.11 of \cite{GGHZ21} for the LQT theorem for $A_\infty$-algebras). Thus it remains to show that $\iota_\mathcal{C}^*$ is a quasi-isomorphism for unital categories. Since we are in characteristic zero latter statement is equivalent to showing that the map \ref{iota} on cyclic chains induces a quasi-isomorphism. Note that this map is induced from a map on Hochschild chains $$\begin{tikzcd}\label{qwert}
\iota: CH_*(A_\mathcal{C})\rightarrow CH_*(\mathcal{C})\end{tikzcd}$$ and by corollary 2.2.3 of \cite{Lo13} the previous statement is equivalent\footnote{To be precise corollary 2.2.3 is only valid for maps on Hochschild chains induced by algebra maps, which $\iota$ is not. However it is straightforward to verify that $\iota$ still satisfies the required functoriality properties to make the argument there work.} to showing that this map \ref{qwert} on Hochschild chains is a quasi-isomorphism. Indeed, this seems to be known to experts under the slogan that both composable and non-composable Hochschild chains of a category compute its Hochschild homology. One way to prove this should be theorem 1.2.13 of \cite{Lo13} applied to $A_\mathcal{C}$ and the separable algebra $E$ spanned as a k-vector space by the original units of $\mathcal{C}$. As this statement is only provided for actual algebras we give a short alternative proof of above slogan, basically using similar ideas: We first note that we have a decomposition of chain complexes
$$\big(CH_*(A_\mathcal{C}),d_{Hoch}\big)=(ker(\iota),\partial)\oplus (ker(\iota)^\bot,\partial),$$
where $ker(\iota)^\bot$ denotes composable Hochschild chains, that is linear combinations of tensors with matching boundary conditions. Since $\iota$ is surjective it suffices to show that $H_*(ker(\iota),\partial)=0$ to prove the statement. Thus let $x\in ker(\iota)$ s.t. $\partial x=0$. Denote by $V_{\lambda\mu}^k$ the subspace of Hochschild chains of $A_\mathcal{C}$ such that at the k-th position for the first time the boundary conditions do not match and that those are $\lambda\neq\mu\in Ob(\mathcal{C})$. As a k-vector spaces we have
$$ker(\iota)=\oplus_{k\in \mathbb{N}}\bigoplus_{\lambda\neq\mu}V_{\lambda\mu}^k.$$

Thus we deduce that also 
$
\partial x|_{V_{\lambda\mu}^k}=0.$
Let us denote by $ e_\mu\otimes^k:V_{\lambda\mu}^k\rightarrow V_{\lambda\mu}^k$ the map given by inserting the unit of $Hom_\mathcal{C}(\mu,\mu)$ into the k-th position. We then have also 
\begin{equation}\label{123}e_\mu\otimes^k\partial x|_{V_{\lambda\mu}^k}=0.
\end{equation}
Consider $$\Tilde{x}:=\sum_k\sum_{\lambda\neq\mu}e_\mu\otimes^k x|_{V_{\lambda\mu}^k}.$$ We then have 
$$\partial \Tilde{x}=\sum_k\sum_{\lambda\neq\mu}e_\mu\otimes^k\partial x|_{V_{\lambda\mu}^k}+x.$$
Indeed, the summands of the Hochschild differential\footnote{See eg equation 7 of \cite{She19} for the formula.} always give zero when a unit is involved, except for the $m^2$ term, which produces the second summand in above equality when applied to the inserted unit and the tensor factor to its right. Further the signs work out since $|e_\mu|=0.$ By equation \ref{123} we conclude
$$\partial \Tilde{x}=x,$$
which proves that $H_*(ker(\iota),\partial)=0$ and thus $$HH_*(A_\mathcal{C})=HH_*(\mathcal{C}).$$ \end{proof}
\subsection{for essentially finite cyclic $A_\infty$-categories}\label{slqtq}
In this section we explain that the LQT map connects the `free and classical observables' from the commutative and non-commutative world that we saw in the previous sections. Further, slight modifications of the LQT map also connect the `prequantum and quantum observables' from the previous sections.

To start, given $V_B$ a collection of cyclic chain complexes of degree d where $B$ is finite and denoting $\mathfrak{gl}_NV_B$ its associated commutator cyclic chain complex we have
\begin{df}
Denote by $$LQT^{fr}_N:\ \mathcal{F}^{fr}(V_B)\rightarrow \mathcal{O}^{fr}(\mathfrak{gl}_NV_B)$$ the  LQT map from definition \ref{LQT_d} applied to $\widetilde{(V_B)}$, the cyclic chain complex under the functor \ref{cyccomm}.
\end{df}
\begin{theorem}\label{LQT_f} Given $V_B$ a collection of cyclic chain complexes of degree d where $B$ is finite
    $$LQT^{fr}_N:\ \mathcal{F}^{fr}(V_B)\rightarrow \mathcal{O}^{fr}(\mathfrak{gl}_NV_B),$$ 
     is a map of (d-2) shifted Poisson dg algebras, functorial with respect to strict cyclic $A_\infty$-morphisms. 
\end{theorem}

\begin{proof}
 This follows directly from corollary \ref{ibc} further below and  thm. 4.7 of \cite{GGHZ21}, the previously studied LQT map for cyclic $A_\infty$-algebras.  
\end{proof}
\begin{cons}
We want to define a map, which we call the pre-quantum LQT map, as follows
\begin{equation}
\label{LQT_r}
LQT^{pq}_N:\ Sym\big({Cyc}^{*+d-4}(V_B)\big)\llbracket\gamma\rrbracket[3-d]\rightarrow C^*(\mathfrak{gl}_NV_B)\llbracket\hbar\rrbracket
\end{equation} 
$$
(a_{1_1}\cdots a_{n_1})\dots (a_{1_m}\cdots a_{n_m}) \nu_{\lambda_1}^{b_1}\dots \nu_{\lambda_r}^{b_r} \gamma^s\mapsto Tr\big(a_{1_1}(\_)\cdots a_{n_1}(\_)\big)\dots Tr\big(a_{1_m}(\_)\cdots a_{n_m}(\_)\big)\hbar^{m+b-1+2s}N^{b} 
$$
Here $\lambda_i\in B$ and $b=\sum^r_{i=1}b_i\in \mathbb{N}$, further we understand the right hand side as polynomials on $\mathfrak{gl}_NV_B[1]$, adjointed $\hbar$. 
\end{cons}
\begin{remark}Note that this assignment does not define a map of algebras: We have that $$LQT^{pq}_N(\nu^2)=\hbar N^2\ \  \text{but}\ \ LQT^{pq}_N(\nu)\cdot LQT^{pq}_N(\nu)= N^2.$$
However, we will see that it is a weighted map of BD algebras\footnote{to be precise it is a weighted BD map on $Sym^{>0}$, but we ignore this subtlety.} in the sense of definition \ref{weimap}.
\end{remark}
More formally, we may understand/define the $LQT^{pq}_N$ as the composition of various maps:

$$LQT^{pq}_N:=m^{pq}\circ f^{pq}\circ tr^{pq}\circ Sym(\iota_\mathcal{C}^*),$$
We introduce these in the following, but first state the central, though immediate result:
\begin{lemma}\label{asgh}
For $\mathcal{C}$ an essentially finite cyclic $A_\infty$-category the map
    $${\iota_\mathcal{C}}^*:(Cyc^*(Sk\mathcal{C}),d+m_{Sk\mathcal{C}})\rightarrow (Cyc^*(A_{Sk\mathcal{C}}),d+m_{A_{Sk\mathcal{C}}}),$$
    induced by \ref{iota} and where we send $\nu_\lambda\mapsto \nu$ for all $\lambda\in Ob(Sk\mathcal{C}),$
    is a quasi-isomorphism of involutive shifted Lie bialgebras.
\end{lemma}
\begin{remark}Note that here it is essential that the cardinality of $ObSk\mathcal{C}$ is finite.
\end{remark}
\begin{proof}
We saw already in the proof of \ref{LQT_conj} that $\iota_\mathcal{C}^*$ is a quasi-isomorphism. Since $\iota_\mathcal{C}^*$ is just the inclusion of composable cyclic words into general cyclic words it follows immediately that it is a map of involutive shifted Lie bialgebras, recalling how we defined the cyclic structure on $A_{Sk\mathcal{C}}$, see definition \ref{sodef}.    
\end{proof}
\begin{corollary}\label{ibc}
   For $\mathcal{C}$ an essentially finite cyclic $A_\infty$-category we find a quasi-isomorphism of (d-2) shifted Poisson algebras
   $$\begin{tikzcd}
    \mathcal{F}^{cl}(\mathcal{C})\arrow{r}{\simeq}&\mathcal{F}^{cl}(Sk\mathcal{C})\arrow{r}{\iota^*}&\mathcal{F}^{cl}(A_{Sk\mathcal{C}}),\end{tikzcd}$$
   induced by the map from lemma \ref{asgh} and \ref{fff}.
\end{corollary}
\begin{df}\label{vfnh}
Given a quantization $I^q$ of an essentially finite cyclic $A_\infty$-category $\mathcal{C}$ denote by $(A_{Sk\mathcal{C}},I^q_f)$ the quantization of the cyclic $A_\infty$-algebra under \ref{cyccomm}, determined by $I^q_f:=\iota^*I^q,$ where $\iota^*$ is the map from corollary \ref{ibc}.
\end{df}
\begin{corollary}\label{qc12}
  For $\mathcal{C}$ a quantization of an essentially finite cyclic $A_\infty$-category we have a quasi-isomorphism of (d-2) twisted BD algebras
   $$\begin{tikzcd}
    \mathcal{F}^{q}(\mathcal{C})\arrow{r}{\simeq}&\mathcal{F}^{q}(Sk\mathcal{C})\arrow{r}{\iota^*}&\mathcal{F}^{q}(A_{Sk\mathcal{C}}),\end{tikzcd}$$
   induced by the map from lemma \ref{asgh} and the \ref{BDqi}.
\end{corollary}
\begin{proof}
Lemma \ref{ibc} implies that we have a map of twisted BD algebras. It remains to argue why this map is a quasi-isomorphism. This follows by considering the spectral sequence associated to the filtration by powers of $\gamma$, $\nu$ and length of symmetric words, which is complete and exhaustive (see example 6.3 of \cite{GGHZ21}). Note that here it is essential that we consider formal power series in $\gamma$. The map from this corollary induces an isomorphism on the first page of the spectral sequence, eg by lemma \ref{asgh}. By the Eilenberg-Moore comparison theorem, eg theorem 5.5.11 of \cite{wei94}, the result follows.   
\end{proof}
Let us introduce the other maps in defining the prequantum LQT map, which are just slight variants of the ones in \cite{GGHZ21}, taking care of the shifts needed for the $\mathbb{Z}$-grading we included:
%\subsubsection{Trace Map}
%Given $V_B$ a collection of cyclic chain complexes degree d, B not necessarily finite, so is
%$$M_NV_B:=\{M_NV_{ij}\}_{i,j\in B}$$ by remark \ref{cycM}.
\begin{itemize}
\item
We recall the trace map from \cite{GGHZ21} section 4.2, inducing for every cyclic chain complex $V$ a map of chain complexes
%$$tr^{fr}:\ (Sym\big(Cyc^{*-1}_+(V_B)\big),d)\rightarrow (Sym\big(Cyc^{*-1}_+(M_NV_B)\big),d)$$
%and also 
$$ (Cyc^*(V),d)\rightarrow (Cyc^*(M_NV),d),$$
where we further sent $\nu$ to $N\nu$. By suitably shifting and extending $\gamma$-linearly this further induces a map denoted
\begin{equation}\label{tracemap}
tr^{pq}:\ Sym\big(Cyc^{*+d-4}(V)\big)\llbracket\gamma\rrbracket[3-d]\rightarrow Sym\big(Cyc^{*+d-4}(M_NV)\big)\llbracket\gamma\rrbracket[3-d].\end{equation} 
   \item 
%As before the map \ref{f} induces
%$$f^{fr}:\ (Sym\big(Cyc^{*-1}_+(M_NV_B)\big),d)\rightarrow \big(Sym\big(C^*_+(\mathfrak{gl}_NV_B)\big),d\big).$$
We recall the map from our remark \ref{f} which allows us to define \begin{equation*}
\big(Cyc^*(M_NV)[-1],d\big)\rightarrow \big(C^*(\mathfrak{gl}_NV),d\big),\end{equation*}
by further sending $\nu\mapsto 1$.  Shifting both sides by $(d-3)$, applying the symmetric algebra functor, extending $\gamma$-linearly and then shifting by $(3-d)$ this induces a map of chain complexes
$$f^{pq}:\ (Sym\big(Cyc^{*+d-4}(M_NV)\big)\llbracket\gamma\rrbracket[3-d],d)\rightarrow \big(Sym\big(C^*(\mathfrak{gl}_NV)[d-3]\big)\llbracket\gamma\rrbracket[3-d],d\big).$$
 \item  
% As before multiplication defines a map
%$$mult^{fr}:\big(Sym\big(C^*_+(\mathfrak{gl}_NV_B)\big),d\big)\rightarrow \big(C^*_+(\mathfrak{gl}_NV_B),d\big).$$
 
There is also a map of degree zero
$$ \big(Sym\big(C^{*}(\mathfrak{gl}_NV)[d-3])\llbracket\gamma\rrbracket,d\big)\rightarrow \big(C^*(\mathfrak{gl}_NV)[d-3]\llbracket\hbar\rrbracket,d\big).
$$ First it sends $\gamma$ to $\hbar^2$, it applies multiplication, ie. sends a symmetric word in n letters of the algebra $C^*(\mathfrak{gl}_NV)[d-3]$ whose multiplication has degree $(d-3)$ to the product of these words, but multiplied by $\hbar^{n-1}$. Note that this map thus has degree 0. Again shifting by $(3-d)$ induces the map, denoted
\end{itemize}
$$m^{pq}: \big(Sym\big(C^*(\mathfrak{gl}_NV)[d-3])\big)\llbracket\gamma\rrbracket[3-d],d\big)\rightarrow \big(C^*(\mathfrak{gl}_NV)[d-3]\llbracket\hbar\rrbracket[3-d],d\big)
\cong \big(C^*(\mathfrak{gl}_NV)\llbracket\hbar\rrbracket,d\big).$$
%These maps preserve the present algebraic structures as follows:

%\begin{theorem}\label{Ass_Lie}
 %   $$m^{pq}\circ f^{pq}: \mathcal{F}^{pq}(M_NV_B)\rightarrow \mathcal{O}^{pq}(\mathfrak{gl}_NV_B)$$ is a weak map of dimension (d-2) BD algebras and $$m^{fr}\circ f^{fr}:\mathcal{F}^{fr}(M_NV_B)\rightarrow \mathcal{O}^{fr}(\mathfrak{gl}_NV_B)$$ is a map of (d-2) shifted Poisson algebras. Further both these maps are functorial with respect to strict cyclic $A_\infty$-isomorphisms of essentially finite collections of cyclic cochain complexes.
%\end{theorem}
%\begin{proof}
%By the same argument as in section 4.1, notably around map 4.2 of \cite{GGHZ21} it follows that both $m^{pq}\circ f^{pq}$ and $m^{fr}\circ f^{fr}$ intertwine the Poisson brackets. The fact that $m^{pq}\circ f^{pq}$ intertwines the BD-differentials follows by the same argument as in Proposition 4.5 of \cite{GGHZ21}. \textcolor{red}{The requirement \ref{wBD} necessary to be a map of weak BD algebras is easy to verify.}
%\end{proof}
These maps preserves the underlying algebraic structure as follows:
\begin{lemma}\label{tr}
For every cyclic chain complex $V$ of degree d
$$m^{pq}\circ f^{pq}\circ tr^{pq}:\mathcal{F}^{pq}(V)\rightarrow \mathcal{O}^{pq}(\mathfrak{gl}_NV)$$ is a weighted map of (d-2) twisted BD algebras and which is functorial with respect to strict cyclic morphisms. 
\end{lemma}
\begin{proof}
    This follows in exactly the same way as in \cite{GGHZ21}, noting that the shifts made do not change the argument in any way. More precisely the trace map is a map of (d-2) twisted BD algebras by lemma 4.6 of loc.cit. The composition of the other two maps are a weighted map of (d-2) twisted BD algebras by prop. 4.5 of loc.cit. To be precise it is a 2-weighted map, recalling the definition \ref{weimap} of a weighted map of BD algebras and examining the definition of the multiplication map above.
\end{proof}
\begin{df}
For $V_B$ a collection of cyclic chain complexes of degree d and $B$ finite denote 
$$LQT^{pq}_N:=m^{pq}\circ f^{pq}\circ tr^{pq}\circ Sym\iota^*:\ \mathcal{F}^{pq}(V_B)\rightarrow \mathcal{O}^{pq}(\mathfrak{gl}_NV_B).$$

\end{df}
It is a simple check to verify that this coincides with the explicit definition \ref{LQT_r}.
\begin{theorem}\label{LQT_pq}
Given $V_B$ a collection of cyclic chain complexes of degree d and $B$ finite the map
    $$LQT^{pq}_N:\ \mathcal{F}^{pq}(V_B)\rightarrow \mathcal{O}^{pq}(\mathfrak{gl}_NV_B),$$ 
    is a weighted map of (d-2) twisted BD algebras, functorial with respect to strict cyclic $A_\infty$-morphisms of essentially finite collections of cyclic cochain complexes.
\end{theorem}
\begin{proof}
Indeed this follows directly from corollary \ref{qc12} and lemma \ref{tr}. 
\end{proof}
Furthermore given $B$ finite and a collection of cyclic chain complexes $V_B$  the prequantum and free LQT map fit with the dequantization maps together in following commutative diagram
\begin{equation}\label{pqdfr}
\begin{tikzcd}
\mathcal{F}^{pq}(V_B)\arrow{r}{ LQT^{pq}_N}\arrow{d}{p}& \mathcal{O}^{pq}(\mathfrak{gl}_NV_B)\arrow{d}{p}\\
\mathcal{F}^{fr}(V_B)\arrow{r}{ LQT^{fr}_N}& \mathcal{O}^{fr}(\mathfrak{gl}_NV_B), \end{tikzcd}\end{equation}
which is a straightforward check.
As a corollary of theorems \ref{LQT_f} and \ref{LQT_pq} we have following two theorems:

\begin{theorem}\label{LQT_cl}
    Given an essentially finite cyclic $A_\infty$-category we denote by $\mathfrak{gl}_NA_{Sk\mathcal{C}}$ the associated cyclic $L_\infty$-algebra. The LQT map is a map of (d-2)-shifted Poisson algebras
    $$LQT^{cl}: F^{cl}(\mathcal{C})\rightarrow \mathcal{O}^{cl}(\mathfrak{gl}_NA_{Sk\mathcal{C}}).$$
    Further this is functorial with respect to  strict cyclic $A_\infty$-morphisms of essentially finite  $A_\infty$-categories. That is, it defines a natural transformation between the functor
    $$ (\text{e.f cyc}
    _dA_\infty cat)_{st}\rightarrow (\mathcal{D}Poiss_{(d-2)})^{op},$$
     given by
    \ref{nattr1} and the composite functor  $$ (\text{e.f cyc}
    _dA_\infty cat)_{st}\rightarrow(cyc_dA_\infty Alg)_{st}[W^{-1}]\rightarrow(cyc_dA_\infty Alg)_{st}[W^{-1}]\rightarrow\cdots$$
    $$\cdots\rightarrow(cyc_dL_\infty Alg)_{st}[V^{-1}] \rightarrow (\mathcal{D}Poiss_{(d-2)})^{op},$$
    determined by
    \ref{cyccomm}, the assignment \ref{cycM}\footnote{The fact that the functor defined by this assignment descends to the localized versions follows from Morita invariance of cyclic cohomology, ie thm 2.10 of \cite{GGHZ21}.} and \ref{cAL}\footnote{See there for the definition of V.} and \ref{nattr2}. These functors and the natural transformation are in fact independent up to isomorphism from the choice of a skeleton.
\end{theorem}
Because of this independence we sometimes omit the notation for the choice of skeleton, which is what we also did for the doamin of the $LQT^{cl}$ map.
\begin{proof}
    An essentially finite cyclic $A_\infty$-category has underlying  collection of cyclic chain complexes over a finite set $B$. Thus we have the free LQT map $$LQT^{fr}_N:\ \mathcal{F}^{fr}(V_B)\rightarrow \mathcal{O}^{fr}(\mathfrak{gl}_NV_B).$$ The cyclic $A_\infty$-category structure induces a Maurer-Cartan element in $\mathcal{F}^{fr}(V_B)$, further the commutator cyclic $L_\infty$-algebra structure induces a Maurer-Cartan element in $\mathcal{O}^{fr}(\mathfrak{gl}_NV_B)$ and $LQT^{fr}$ maps the prior one to the latter one. Twisting the $LQT^{fr}$ map by these compatible MCE is again a map of (d-2)-shifted Poisson algebras by general properties of the twisting procedure. When forgetting the shifted Poisson structures this map reduces to the standard LQT map applied to the underlying $A_\infty$-category, by construction.

    The functoriality statement follows by functoriality of the one object LQT map and functoriality of the map induced by \ref{iota}.

    The independence up to isomorphism for the choice of a skeleton follows for the first functor from \ref{undf}. The independence up to isomorphism for the choice of a skeleton for the second functor is induced from \ref{cyccomm}.
\end{proof}
\begin{theorem}\label{LQT_q}
   i) A quantization of an essentially finite cyclic $A_\infty$-category induces a quantization of the associated cyclic $L_\infty$-algebra with values in $N\times N$-matrices, for each $N\in\mathbb{N}$. There is a weighted map of (d-2) twisted BD algebras
    $$LQT^{q}: \mathcal{F}^{q}(\mathcal{C})\rightarrow \mathcal{O}^{q}(\mathfrak{gl}_NA_{Sk\mathcal{C}}).$$
  ii)  Further this is functorial with respect to strict cyclic $A_\infty$-morphisms of quantum essentially finite  $A_\infty$-categories in the sense that it defines a natural transformation between the functor
    $$ (\text{q.e.f. cyc}_dA_\infty cat)_{st}\rightarrow (\mathcal{D}\text{BD}^{(d-2)tw})^{op},$$ 
     given by
    \ref{F_q} and the composite functor  $$ (\text{q.e.f cyc}
    _dA_\infty cat)_{st}\rightarrow(q.cyc_dA_\infty Alg)_{st}[U^{-1}]\rightarrow(q.cyc_dA_\infty Alg)_{st}[U^{-1}]\rightarrow\cdots$$
$$\cdots\rightarrow (q. cyc_dL_\infty Alg)_{st}[P^{-1}] \rightarrow (\mathcal{D}BD^{(d-2)tw})^{op}$$
     determined by \ref{qcyccomm}, by \ref{cycM}\footnote{The fact that this functor descend to quantized version follows from the trace map of lemma \ref{tr}. Further that this functor descends to the localized version follows from Morita invariance of cyclic cohomology, thm. 2.10 of \cite{GGHZ21} and the filtration spectral sequence, see proof of part iii) of  theorem \ref{LQT_q}.} and by  \ref{qcAL}\footnote{See there for the definition of P.} and \ref{O_q}. These functors and the natural transformation are in fact independent from the choice of a skeleton, up to isomorphism.

  iii)    We have that 
    
\begin{equation}\label{lNq}
\mathcal{F}^{q}(\mathcal{C})\simeq \varprojlim_{N\to\infty}\mathcal{O}^{q}(\mathfrak{gl}_NA_{Sk\mathcal{C}})    .\end{equation}
\end{theorem}
Because of the independence from the choice of skeleton we sometimes omit the notation for the choice of such, which is what we also did for the doamin of the $LQT^{q}$ map.
\begin{proof}
A quantization of an essentially finite cyclic $A_\infty$-category has underlying collection of cyclic chain complexes over a finite set $B$. Thus we have the prequantum LQT map $$LQT^{pq}_N:\ \mathcal{F}^{pq}(V_B)\rightarrow \mathcal{O}^{pq}(\mathfrak{gl}_NV_B).$$
The datum of the quantization induces a Maurer-Cartan element in $\mathcal{F}^{pq}(V_B)$. Since the prequantum LQT map is a weighted map of BD algebras it sends a Maurer-Cartan element to a Maurer Cartan element. By diagram \ref{pqdfr} it follows that the so determined Maurer-Cartan element is a quantization of the cyclic commutator $L_\infty$-algebra. Twisting the $LQT^{pq}$ map by these compatible MCE is again a weighted map of (d-2) twisted BD algebras, which we denote by $LQT^q$. 

The independence up to isomorphism for the choice of a skeleton follows for the first functor from \ref{undfq}. The independence up to isomorphism for the choice of a skeleton for the second functor is induced from \ref{qcyccomm}.

The large $N$-statement follows again by considering the compete and exhaustive filtration from the proof of corollary \ref{qc12} and the one given by powers of $\hbar$. The LQT maps at each level $N$, which are filtered, see diagram 7.1 of \cite{GGHZ21}, induce the map \ref{lNq}. On the first page of the spectral sequence associated to the filtrations this map induces a quasi-isomorphism, by \ref{LQT_conj}. By the Eilenberg-Moore comparison theorem (theorem 5.5.11 of \cite{wei94}) the claim follows.
\end{proof}